\documentclass[12pt,twoside,a4paper]{article}

\usepackage{fullpage}
\usepackage{enumerate}
\usepackage{subfigure}

\usepackage[utf8]{inputenc}
\usepackage[T1]{fontenc}
\usepackage[english]{babel}
\usepackage{enumitem}
\usepackage{amsmath,amsfonts,amssymb,graphicx,bbm,stmaryrd}

\usepackage{amsthm}

\theoremstyle{plain}
\newtheorem{theorem}{Theorem}
\newtheorem{proposition}{Proposition}
\newtheorem{lemma}{Lemma}

\newtheorem{conjecture}{Conjecture}

\theoremstyle{remark} 
\newtheorem*{remark}{Remark}
\theoremstyle{definition} 
\newtheorem*{definition}{Definition}

\usepackage{mathtools}

\usepackage{dsfont} 

\newcommand{\calA}{\mathcal{A}}
\newcommand{\calB}{\mathcal{B}}
\newcommand{\calC}{\mathcal{C}}
\newcommand{\calD}{\mathcal{D}}
\newcommand{\calE}{\mathcal{E}}
\newcommand{\calF}{\mathcal{F}}
\newcommand{\calG}{\mathcal{G}}

\newcommand{\bbH}{\mathbb{H}}

\newcommand{\bbN}{\mathbb{N}}

\newcommand{\bbR}{\mathbb{R}}
\newcommand{\bbS}{\mathbb{S}}

\newcommand{\bbU}{\mathbb{U}}

\newcommand{\bbZ}{\mathbb{Z}}

\newcommand{\ep}{\varepsilon}

\newcommand{\Z}{\mathbb{Z}}

\newcommand{\g}{\gamma}

\newcommand{\N}{\mathbb{N}}
\newcommand{\R}{\mathbb{R}}
\renewcommand{\P}{\mathbb{P}}

\newcommand{\Ps}[1]{\P[\bbS][1/0]{#1}} 

\newcommand{\lr}[1][]{\xleftrightarrow{\: #1 \:\:}}

\renewcommand{\P}[1][]{
  \def\argI{{#1}}
  \PRelay
}
\newcommand{\PRelay}[2][]{
  \phi_{\argI}^{#1} \left [ #2 \right ]
}

\newcommand{\PP}[1][]{
  \def\argI{{#1}}
  \PPRelay
}
\newcommand{\PPRelay}[1][]{
  \phi_{\argI}^{#1}
}

\renewcommand{\cal}{\mathcal}

\usepackage{constants}
\newconstantfamily{small}{
symbol=c}

\newcommand{\cl}[1]{\Cl[small]{#1}}

\usepackage{verbatim}

\usepackage{tabularx}

\mathtoolsset{showonlyrefs}
\usepackage{hyperref}
\title{Continuity of the phase transition for planar random-cluster and Potts models with $1\le q\le 4$} \author{Hugo
  \textsc{Duminil-Copin} \and Vladas \textsc{Sidoravicius} \and Vincent
  \textsc{Tassion}}
\date{\today} 

\begin{document}
\maketitle

\begin{abstract}
This article studies the planar Potts model and its random-cluster representation. We show that the phase transition of the nearest-neighbor ferromagnetic $q$-state Potts model on $\bbZ^2$ is continuous for $q\in\{2,3,4\}$, in the sense that there exists a unique Gibbs state, or equivalently that there is no ordering for the critical Gibbs states with monochromatic boundary conditions. 

The proof uses the random-cluster model with cluster-weight $q\ge1$ (note that $q$ is not necessarily an integer) and is based on two ingredients:
\begin{itemize}[noitemsep,nolistsep]
\item The fact that the two-point function for the free state decays sub-exponentially fast for cluster-weights $1\le q\le 4$, which is 
derived studying parafermionic observables on a discrete Riemann surface.
\item A new result proving the equivalence of several properties of critical random-cluster models:
\begin{itemize}
\item the absence of infinite-cluster for wired boundary conditions, 
\item the uniqueness of infinite-volume measures,
\item the sub-exponential decay of the two-point function for free boundary conditions, 
\item a Russo-Seymour-Welsh type result on crossing probabilities in rectangles with {\em arbitrary boundary conditions}. \end{itemize}
\end{itemize}
The result leads to a number of consequences concerning the scaling limit of the random-cluster model with $q\in[1,4]$. It shows that the family of interfaces (for instance for Dobrushin boundary conditions) are tight when taking the scaling limit and that any sub-sequential limit can be parametrized by a Loewner chain. We also study the effect of boundary conditions on these sub-sequential limits. Let us mention that the result should be instrumental in the study of critical exponents as well. \end{abstract}

\tableofcontents

\section{Introduction}

\subsection{Motivation}

The Potts model is a model of random coloring of $\Z^2$ introduced as
a generalization of the Ising model to more-than-two components spin
systems. In this model, each vertex of $\Z^2$ receives a spin among $q$
possible colors. The energy of a configuration is proportional to the
number of frustrated edges, meaning edges whose endpoints have
different spins. Since its introduction by Potts \cite{Pot52}, the model has been a
laboratory for testing new ideas and developing far-reaching tools. In
two dimensions, it exhibits a reach panel of possible critical
behaviors depending on the number of colors, and despite the fact that
the model is exactly solvable (yet not rigorously for $q\ne2$), the
mathematical understanding of its phase transition remains restricted
to a few cases (namely $q=2$ and $q$ large). We refer to \cite{Wu82}
for a review on this model.

The question of deciding whether a phase transition is continuous or
discontinuous constitutes one of the most fundamental question in
statistical physics, and an extensive literature has been devoted to this
subject in the case of the Potts model. In the planar case, Baxter
\cite{Bax71,Bax73,Bax89} used a mapping between the Potts model and
solid-on-solid ice-models to compute the spontaneous magnetization at criticality
for $q\ge 4$. He was able to predict that the phase transition was discontinuous for
$q\ge 5$. While this computation gives
a good insight on the behavior of the model, it relies on unproved
assumptions which, forty years after their formulation, seem still
very difficult to justify rigorously (and are related to the nature of the phase transition itself). 
Furthermore, Baxter outlined another heuristic prediction that the phase transition is 
continuous for $q\le 4$, but, according to the author himself, the justification is not as solid 
for this case. Among other results, the present
article proves that the phase transition is continuous for
$q\in\{2,3,4\}$ without any reference to unproved assumptions.

Most of our article will be devoted to the study of the 
random-cluster model. This model is a probability measure on edge
configurations, each edge being declared open or closed, such that the
probability of a configuration is proportional to $p^{\#\,\text{open
    edges}}(1-p)^{\#\,\text{closed edges}}q^{\#\,\text{clusters}}$,
where clusters are maximal connected subgraphs of the configuration, and
$(p,q)\in[0,1]\times\R_+$. For $q=1$, the model is Bernoulli
percolation. 

Since its introduction by Fortuin and Kasteleyn
\cite{ForKas72}, the random-cluster model has become an important tool in
the study of phase transitions. The spin correlations of Potts models
are rephrased as cluster connectivity properties of their
random-cluster representations. As a byproduct, properties of the random-cluster model
can be transferred to the Potts model, and vice-versa. {\em The random-cluster model 
being an important model of its own, Theorems~\ref{thm:main} and \ref{thm:decide} below 
are of independent interest.}

While the understanding of critical Bernoulli percolation in dimensions two is
relatively advanced, the case of the random-cluster model remains
mysterious. The long range dependency makes the model challenging to
study by probabilistic tools, and some of its most basic properties were
not proved until recently. In this article, we derive several
properties of the critical model, including a suitable generalization
of the celebrated {\em Russo-Seymour-Welsh} approach available for
Bernoulli percolation. This powerful tool enables us to prove several
new results for the critical phase.

This article fits in the more general context of the study of
conformally invariant planar lattice models. In the early eighties,
physicists Belavin, Polyakov and Zamolodchikov postulated conformal
invariance of critical planar statistical models \cite{BelPolZam84,BelPolZam84a}.
This prediction enabled physicists to invoke Conformal Field Theory
in order to formulate many conjectures on these models. From a
mathematical perspective, proving rigorously the conformal invariance
of a model (and properties following from it) constitutes a formidable
challenge.

In recent years, the connection between discrete holomorphicity and
planar statistical physics led to spectacular progress in this
direction. Kenyon \cite{Ken00,Ken01}, Smirnov \cite{Smi10} and Chelkak
and Smirnov \cite{CheSmi12} exhibited discrete holomorphic observables in
the dimer and Ising models and proved their convergence to conformal
maps in the scaling limit. These results paved the way to the rigorous
proof of conformal invariance for these two models. Other discrete
observables have been proposed for a number of critical models,
including self-avoiding walks and Potts models. While these
observables are {\em not exactly discrete holomorphic}, their discrete
contour integrals vanish, a property shared with discrete holomorphic
functions.
It is a priori unclear whether this property is of any relevance for
the models. Nevertheless, in the case of the self-avoiding walk, it
was proved to be sufficient to compute the connective constant of the
hexagonal lattice \cite{DumSmi12}. One can consider that this article
is part of a program initiated informally in \cite{DumSmi12} and
devoted to the study of the possible applications of the property
mentioned above. In our case, we will use parafermionic observables introduced independently in
\cite{FraKad80,RivCar06,Smi06,Smi10} to prove our main theorem as well as
several corollaries. (More precisely, we will use a corollary which may be deduced from the
study of such observables.) Since such observables have been found in a large
class of planar critical models, we believe that similar applications
can be obtained for these models as well, and that the tools developed in this 
article should be instrumental there. Last but not least, these techniques improve the understanding of the scaling limit of these models, 
and we think that they will be useful for proving conformal invariance.

\subsection{Definition of the models and main statements}

\subsubsection{Definition of Potts models and statement of the main theorem}

Consider an integer $q\ge 2$ and a subgraph $G=(V_G,E_G)$ of the square
lattice $\bbZ^2$. Here and below, $V_G$ is the set of {\em vertices} of $G$ and
$E_G\subset V_G^2$ is the set of {\em edges}. For simplicity, the square
lattice will be identified with its set of vertices, namely $\Z^2$.
For two vertices $x,y\in V_G$, $x\sim y$ denotes the fact that $(x,y)\in
E_G$.

Let $\tau\in\{1,\dots,q\}^{\Z^2}$. The {\em $q$-state Potts model on
  $G$ with boundary conditions $\tau$} is defined as follows. The
space of configurations is $\Omega=\{1,\dots,q\}^{\Z^2}$. For a
configuration $\sigma=(\sigma_x:x\in\Z^2)\in \Omega$, the quantity
$\sigma_x$ is called the {\em spin} at $x$ (it is sometimes
interpreted as being a color). The {\em energy} of a configuration
$\sigma\in \Omega$ is given by the Hamiltonian
$$H_G^\tau(\sigma)~:=~\begin{cases}\displaystyle-\sum_{\substack{x\sim
      y\\\{x,y\}\cap G\neq\emptyset}}\delta_{\sigma_x,\sigma_y}&\text{ if $\sigma_x=\tau_x$ for $x\notin V_G$,}\\
  \ \ \ \ \infty&\text{ otherwise.}\end{cases} $$ Above,
$\delta_{a,b}$ denotes the Kronecker symbol equal to 1 if $a=b$ and 0
otherwise. The spin-configuration is sampled proportionally to its
Boltzmann weight: at an inverse-temperature $\beta$, the probability
$\mu_{G,\beta}^\tau$ of a configuration $\sigma$ is defined by
$$ \mu_{G,\beta}^\tau[\sigma]~:=~\frac{{\rm e}^{-\beta H_{G}^\tau(\sigma)}}{Z_{G,\beta}^\tau}\quad\text{ where }\quad Z_{G,\beta}^\tau~:=~\sum_{\sigma\in\Omega}{\rm e}^{-\beta H_{G}^\tau(\sigma)}$$
is the {\em partition function} defined in such a way that
the sum of the weights over all possible configurations equals 1. By
construction, configurations that do not coincide with $\tau$ outside
of $G$ have probability 0.

Infinite-volume Gibbs measures can be defined by taking limits, as $G$
tends to $\Z^2$, of finite-volume measures $\mu_{G,\beta}^\tau$. In
particular, if $(i)$ denotes the constant boundary configuration equal to
$i\in\{1,\dots,q\}$, the sequence of measures $\mu_{G,\beta}^{(i)}$
converges, as $G$ tends to infinity, to a Gibbs measure denoted by
$\mu_{\Z^2,\beta}^{(i)}$. This measure is called the {\em
  infinite-volume Gibbs measure with monochromatic boundary conditions
  $i$}.

The Potts models undergo a phase transition in infinite volume at a
certain {\em critical inverse-temperature} $\beta_c(q)\in(0,\infty)$
in the following sense
$$\mu_{\Z^2,\beta}^{(i)}[\sigma_0=i]=\begin{cases}\tfrac1q&\text{ if $\beta<\beta_c(q)$,}\\ \tfrac1q+m_\beta>\frac1q&\text{ if $\beta>\beta_c(q)$.}\end{cases}$$
The value $\beta_c(q)$ is computed in \cite{BefDum12b} and is equal to
$\log(1+\sqrt q)$ for any integer $q$ (this value was previously known
for $q=2$ \cite{Ons44} and $q\ge 26$ \cite{LaaMesRui86}). 

This article is
devoted to the study of the phase transition at $\beta=\beta_c(q)$.
The phase transition is said to be {\em continuous} if
$\mu_{\Z^2,\beta_c(q)}^{(i)}[\sigma_0=i]=\tfrac1q$ and {\em
  discontinuous} otherwise. The main result is the following.
\begin{theorem}[Continuity of the phase transition for 2, 3 or 4 colors]\label{thm:main Potts}
Let $q\in\{2,3,4\}$. Then for any $i\in\{1,\dots,q\}$, we have
$$\mu_{\Z^2,\beta_c(q)}^{(i)}[\sigma_0=i]=\tfrac1q.$$
\end{theorem}
This result was known in the $q=2$ case. For two colors, the model is
 the Ising model. Onsager computed the free energy in
\cite{Ons44} and Yang obtained a formula for the magnetization in
\cite{Yan52}. In particular, this formula implies that the
magnetization is zero at criticality. This results has been reproved
in a number of papers since then. Let us mention a recent proof
\cite{Wer09a} not using any exact integrability.

For $q$ equal to 3 or 4, the result appears to be new. As mentioned in
the previous section, exact yet non rigorous computations performed
by Baxter strongly suggest that the phase transition is continuous for
$q\le 4$, and discontinuous for $q>4$. This result therefore tackles
the whole range of $q$ for which the phase transition is continuous.
Let us mention that the technology developed here is not restricted to
the study of Potts models for $q\le 4$: a property of Potts models
with $q\ge 5$ colors witnessing some sort of ordering at criticality is also
derived, see Proposition~\ref{prop:qge5}. Unfortunately, we were
unable to show rigorously that the phase transition is discontinuous
in the sense defined above.

In dimension $d\ge 3$, the phase transition is expected to be
continuous if and only if $q=2$. The best results in this direction
are the following. On the one hand, the fact that the phase transition
is continuous for the Ising model ($q=2$) is known for any $d\ge 3$
\cite{AizDumSid15} (in fact, the critical exponents are known to be taking
their mean-field value for $d\ge4$ \cite{AizFer86}). On the other hand,
mean-field considerations together with Reflection Positivity enabled
\cite{BisChaCra06} to prove that for any $q\ge 3$, the $q$-state Potts model
undergoes a discontinuous phase transition above some dimension
$d_c(q)$. Finally, Reflection Positivity can be harnessed to prove
that for any $d\ge 2$, the phase transition is discontinuous provided
$q$ is large enough \cite{KotShl82}.
 
\subsubsection{The random-cluster model} The proof of
Theorem~\ref{thm:main Potts} is based on the study of a graphical
representation of the Potts model, called the {\em random-cluster
  model}. Let us define it properly. Consider $p\in[0,1]$, $q>0$ and a
subgraph $G=(V_G,E_G)$ of the square lattice. A configuration $\omega$ is
an element of $\Omega'=\{0,1\}^{E_G}$. An edge $e$ with $\omega(e)=1$ is
said to be {\em open}, while an edge with $\omega(e)=0$ is said to be
{\em closed}. Two vertices $x$ and $y$ in $V_G$ are said to be {\em
  connected} (this event is denoted by $x\longleftrightarrow y$) if
there exists a sequence of vertices $x=v_1,v_2,\dots,v_{r-1},v_r=y$
such that $\{v_i,v_{i+1}\}$ is an open edge for every $i<r$. A {\em
  connected component} of $\omega$ is a maximal connected subgraph of
$\omega$. Let $o(\omega)$ and $c(\omega)$ be respectively the number
of open and closed edges in $\omega$.

The {\em random-cluster measure on $E_G$ with edge-weight $p$,
  cluster-weight $q$, and {\bf free} boundary conditions} is defined
by the formula
$$\P[G,p,q][0]{\omega}=\frac{p^{o(\omega)}(1-p)^{c(\omega)}q^{k_0(\omega)}}{Z^0_{G,p,q}},$$
where $k_0(\omega)$ is the number of connected components of the graph
$\omega$, and $Z^0_{G,p,q}$ is defined in such a way that the sum of
the weights over all possible configurations equals 1. We also define
the {\em random-cluster measure on $E_G$ with edge-weight $p$,
  cluster-weight $q$, and {\bf wired} boundary conditions} by the
formula
$$\P[G,p,q][1]{\omega}=\frac{p^{o(\omega)}(1-p)^{c(\omega)}q^{k_1(\omega)}}{Z^1_{G,p,q}},$$
where $k_1(\omega)$ is the number of connected components of the graph $\omega$, except that all the connected components of vertices in the vertex boundary $\partial G$, i.e. the set of vertices in $ V_G$ with less than four neighbors in $G$, are counted as being part of the same connected component. Again, $Z^1_{G,p,q}$ is defined in such a way that the sum of the weights over all possible configurations equals 1. 

For $q\geq 1$, infinite-volume measures can be defined on $\Z^2$ by
taking limits of finite-volume measures for graphs tending to $\Z^2$.
In particular, the {\em infinite-volume random-cluster measure with
  free (resp. wired) boundary conditions} $\PP[\Z^2,p,q][0]$ (resp.
$\PP[\Z^2,p,q][1]$) can be defined as the limit of the sequence of
measures $\PP[{G},p,q][0]$ (resp.\@ $\PP[{G},p,q][1]$) for
$G\nearrow\Z^2$. We refer the reader to \cite{Gri06} for more details
on this construction.

The random-cluster model with $q\ge 1$ undergoes a phase transition in infinite volume in the following sense. There exists $p_c(q)\in(0,1)$ such that 
$$\P[\Z^2,p,q][1]{0\longleftrightarrow \infty}=\begin{cases}0&\text{ if $p<p_c(q)$}, \\\theta^1(p,q)>0&\text{ if $p>p_c(q)$},\end{cases}$$
where $\{0\longleftrightarrow\infty\}$ denotes the event that 0 belongs to an infinite connected component. The value of $p_c(q)$ was recently proved to be equal to $\sqrt q/(1+\sqrt q)$ for any $q\ge 1$ in \cite{BefDum12b} (see also \cite{DumMan15}). The result was previously proved in \cite{Kes80} for Bernoulli percolation ($q=1$), in \cite{Ons44} for $q=2$ using the connection with the Ising model and in \cite{LaaMesMir91} for $q\ge 25.72$. 

Similarly to the Potts model case, a notion of continuous/discontinuous phase transition can be defined: the phase transition is said to be {\em continuous} if $\P[\Z^2,p_c(q),q][1]{0\longleftrightarrow \infty} =0$ and discontinuous otherwise. The following theorem is the {\em alter ego} of Theorem~\ref{thm:main Potts}.

\begin{theorem}[Continuous phase transition for cluster-weight $1\le q\le 4$]\label{thm:main RCM}
Let $q\in[1,4]$, then $\P[\Z^2,p_c(q),q][1]{0\longleftrightarrow \infty} =0$.
\end{theorem}

Note that $q$ may not be an integer in this theorem. Let us now describe briefly the coupling between Potts models and their random-cluster representation. Fix $q\ge 2$  integer. From a random-cluster configuration sampled according to $\PP[G,p,q][1]$, color each component (meaning all the vertices in it) with one color chosen uniformly in $\{1,\dots,q\}$, except for the connected component containing the vertices in $\partial G$ which receive color $i$. The law of the random coloring thus obtained is $\mu_{G,\beta}^{(i)}$, where $\beta=-\log(1-p)$. This coupling between the random-cluster model with integer cluster-weight and the Potts models enables us to deduce Theorem~\ref{thm:main Potts} from Theorem~\ref{thm:main RCM} immediately: simply notice that
$$\mu_{G,\beta}^{(i)}[\sigma_0=i]=\tfrac1q+\phi^1_{\bbZ^2,p,q}[0\lr \partial G],$$
and take the limit as $G$ tends to $\bbZ^2$. For this reason, we now focus on Theorem~\ref{thm:main RCM}.

\subsubsection{An alternative for the behavior of critical random-cluster models} 
Proving Theorem~\ref{thm:main RCM} requires a much better understanding of the critical phase than the one available until now. Indeed, except for the $q=1$, $q=2$ and $q\ge 25.72$ cases, very little was known on critical random-cluster models. The following theorem provides new insight on the possible critical behavior of these models.

For an integer $n$, let $\Lambda_n$ denote the {\em box $[-n,n]^2$ of size $n$}. An open path is a path of adjacent open edges (we refer to the next section for a formal definition). Let $0\longleftrightarrow\partial\Lambda_n$ be the event that there exists an open path from the origin to the boundary of $\Lambda_n$. For a rectangle $R=[a,b]\times[c,d]$, let $\calC_h(R)$ be the event that there exists an open path in $R$ from $\{a\}\times[c,d]$ to $\{b\}\times[c,d]$. 

 \begin{theorem}\label{thm:main}
Let $q\ge 1$. The following assertions are equivalent :
 \begin{enumerate}
 \item[{\rm \bf P1}] (Absence of infinite cluster at criticality) $\P[\bbZ^2,p_c,q][1]{0\longleftrightarrow \infty}=0$.
 \item[{\rm \bf P2}] $\PP[\bbZ^2,p_c,q][0]=\PP[\bbZ^2,p_c,q][1]$.
     \item[{\rm \bf P3}] (Infinite susceptibility) $\displaystyle \chi^0(p_c,q):=\sum_{x\in \mathbb Z^2} \P[\bbZ^2,p_c,q][0]{0\longleftrightarrow x} =\infty.$
\item[{\rm \bf P4}] (Sub-exponential decay for free boundary conditions) 
 $$\displaystyle \lim_{n\to \infty} \tfrac 1 n \log \P[
     \bbZ^2,p_c,q][0]{0\longleftrightarrow \partial\Lambda_n} =0.$$
  \item[{\rm \bf P5}] (RSW) For any $\alpha>0$, there exists $c=c(\alpha)>0$ such that for all $n\geq 1$,
   $$\P[{[-n,(\alpha+1)n]\times[-n,2n]}, p_c,q][0]{\calC_h([0,\alpha n]\times[0,n])}\geq c.$$
  \end{enumerate}
\end{theorem}

The previous theorem does not show that these conditions are all satisfied, but that they are equivalent. In fact, whether the conditions are satisfied or not will depend on the value of $q$, see Section~\ref{sec:intro parafermion} for a more detailed discussion. 
\medbreak
The previous result was previously known in a few cases:
\begin{itemize}[nolistsep,noitemsep]
\item {\em Bernoulli percolation (random-cluster model with $q=1$).} In such case {\bf P2} is obviously satisfied. Furthermore, Russo \cite{Rus78} proved that {\bf P1}, {\bf P3} and {\bf P4} are all true (and therefore equivalent). Finally, {\bf P5} was proved by Russo \cite{Rus78} and Seymour-Welsh \cite{SeyWel78}. 
\item {\em Random-cluster model with $q=2$.} This model is directly related to the Ising via the coupling between Potts and random-cluster models mentioned above. Therefore, all of these properties can be proved to be true using the following results on the Ising model: Onsager proved that the critical Ising measure is unique and that the phase transition is continuous in \cite{Ons44}. Properties {\bf P3} and {\bf P4} follow from Simon's correlation inequality for the Ising model \cite{Sim80}. Property {\bf P5} was proved in \cite{DumHonNol11} using a proof specific to the Ising model. Interestingly, in the Ising case each property is derived independently and no direct equivalence was known previously. 
\item {\em Random-cluster model with $q\ge 25.72$.}
In this case, none of the above properties are satisfied, as proved by using the Pirogov-Sinai technology \cite{LaaMesMir91}. 
\end{itemize}
Except for these special cases, no general result was known, and Theorem~\ref{thm:main} represents, to the best of our knowledge, the first formal proof of the equivalence between these conditions for a relatively large class of dependent percolation models. We expect that a similar result can be stated for a large class of models, and that some of the tools developed in this article may be extended to these models.
\medbreak
\begin{remark}{\bf P4}$\Rightarrow${\bf P1} implies that whenever there is an infinite-cluster for the wired boundary conditions,  correlations decay exponentially fast at criticality  for free boundary conditions.  
\end{remark}
 \medbreak
Before proceeding further, let us discuss alternative conditions which could replace the conditions {\bf P1}--{\bf P5}. Once again, $q$ is assume to be larger or equal to 1.

The condition {\bf P1} has the following interpretation: it is equivalent  to \medbreak
{\rm \bf P1'} (Continuous phase transition) $\displaystyle\lim_{p\searrow p_c}\P[\bbZ^2,p,q][1]{0\longleftrightarrow \infty}=0$
\medbreak
\noindent (simply use  \cite[(4.35)]{Gri06}). Note that the (almost sure) absence of an infinite-cluster for $\PP[\bbZ^2,p_c,q][0]$ follows from  Zhang's argument \cite[Theorem~(6.17)(a)]{Gri06} and is true for any $q\ge1$. Nevertheless, it does not imply the (almost sure) absence of an infinite-cluster for $\PP[\bbZ^2,p_c,q][1]$ nor the continuity of the phase transition. 
 
The property {\bf P2} can be reinterpreted in terms of infinite-volume measures (see \cite{Gri06} for a formal definition). Then, {\bf P2} is equivalent to (see \cite[Theorem (4.34)]{Gri06})
\medbreak
{\rm \bf P2'} The infinite-volume measure on $\bbZ^2$ at $p_c$ and $q$ is unique.
\medbreak
Let us now turn to {\bf P4} which can be understood in terms of the so-called {\em correlation length} defined for $p<p_c(q)$ by the formula
$$\displaystyle \xi(p,q)=\Big(-\lim_{n\rightarrow\infty}\tfrac1n\log\phi_{\bbZ^2,p,q}^0[(0,0)\longleftrightarrow (n,0)]\Big)^{-1}.$$
Now, {\bf P4} is equivalent to
     \medbreak
{\rm \bf P4'} (vanishing mass-gap) $\xi(p,q)$ tends to $+\infty$ as $p\nearrow p_c(q)$.
\medbreak
Condition {\bf P1} together with {\bf P3} have an interesting consequence in terms of the order of the phase transition for the Potts model. We do not enter in the details here but let us briefly mention that properties {\bf P1} and {\bf P3} are respectively equivalent to the continuity and the non-differentiability with respect to the magnetic field $h$ of the free energy at $(\beta=\beta_c,h=0)$. Therefore, these properties mean that the phase transition of the corresponding Potts model is of {\em second order}.
The properties {\bf P1}--{\bf P4} (and their equivalent formulations) are classical definitions describing continuous phase transitions and are believed to be equivalent for many natural models, even though it is a priori unclear how this can be proved in a robust way. 

Now that we have an interpretation for properties {\bf P1}--{\bf P4}, let us explain why {\bf P5} is of particular interest: it provides an equivalent to the RSW theorem proved in \cite{BefDum12b} which is {\em uniform in boundary conditions} (see Proposition~\ref{prop:classical RSW} below).
This uniformity with respect to boundary conditions is crucial for applications, especially when trying to decouple events, see e.g. Section~\ref{sec:applications 1}. Let us also emphasize that the fact that {\bf P5} can be derived from the other properties requires the development of a Russo-Seymour-Welsh theory for dependent percolation models. As mentioned above, such a theory existed for Bernoulli percolation \cite{SeyWel78,Rus78}, and for $q=2$ \cite{DumHonNol11}, but in the latter case the proof was based on discrete holomorphicity, hence hiding the close connection between {\bf P5} and the other properties. This Russo-Seymour-Welsh theory is expected to apply to a large class of planar models, and we insist on the fact that uniformity on boundary conditions is crucial. \medbreak
 
\begin{remark}
The restriction on boundary conditions being at distance $n$ from the rectangle can be relaxed in the following way: if {\bf P5} holds, then for any $\alpha>0$ and $\ep>0$, there exists $c=c(\alpha,\ep)>0$ such that for every $n\ge1$,
$$\P[{[-\ep n,(\alpha+\ep)n]\times[-\ep n,(1+\ep)n]},p_c,q][\xi]{\calC_h([0,\alpha n]\times[0,n])}\ge c.$$ It is natural to ask why boundary conditions are fixed at distance $\ep n$ of the rectangle $[0,\alpha n]\times[0,n]$ and not simply on the boundary. The reason is the latter property is not equivalent to {\bf P1}--{\bf P5}. Indeed, it may in fact be the case that {\bf P5} holds but that crossing probabilities of rectangles $[0,\alpha n]\times[0,n]$ with free boundary conditions on their boundary converge to zero as $n$ tends to infinity.  Such phenomenon does not occur for $1\le q<4$ as shown in Theorem~\ref{thm:strongest RSW} below but is expected to occur for $q=4$. For this reason, we will always work with boundary conditions at ``macroscopic distance'' from the boundary.\end{remark}

\subsubsection{Random-cluster model with cluster weight $q\in[1,4]$}\label{sec:intro parafermion} The previous alternative provides us with a powerful tool to prove Theorem~\ref{thm:main RCM}. Namely, it is sufficient to prove one of the properties {\bf P2}--{\bf P5} when $1\le q\le 4$ to derive our result. We will therefore focus on property {\bf P4}, which is the easiest to check. 

In order to prove {\bf P4}, we will use estimates on the probability of being connected by an open path which can be deduced from the fact that  discrete contour integrals of the so-called {\em (edge) parafermionic observable} vanish. This observable was introduced in \cite{Smi06,Smi10} for $q\in(0,4)$ and then generalized to $q>4$ in \cite{BefDumSmi12} (the $q=4$ case also requires the introduction of a slightly different observable). It satisfies local relations that can be understood as discretizations of the Cauchy-Riemann equations when the model is critical. These relations imply that discrete contour integrals vanish. We do not recall the definition of the parafermionic observable nor do we describe its principal properties, and simply mention an important corollary (see Equation~\eqref{eq:main equation}). For more details on the parafermionic observable, we refer to \cite{Smi06} or  \cite[Chapter 6]{Dum13}.

In any case, the parafermionic observable can be used to show the following theorem dealing with random-cluster models with $1\le q\le 4$. 
\begin{theorem}[\cite{Dum12}]\label{thm:decide}
Let $1\le q\le 4$, then $\displaystyle \lim_{n\to \infty} \tfrac 1 n \log \P[\mathbb
     Z^2,p_c(q),q][0]{0\longleftrightarrow \partial\Lambda_n} =0.$
\end{theorem}
Theorem~\ref{thm:decide} together with Theorem~\ref{thm:main} implies that {\bf P1--P5} are satisfied for $1\le q\le 4$. In particular, these two theorems imply Theorem~\ref{thm:main RCM}.

As mentioned above, the proof uses the fact that the discrete contour integrals of parafermionic observables vanish. The intrinsic difficulty of this theorem relies on the fact that, for our proof to work, the random-cluster model needs to be considered on the universal cover of $\Z^2$ minus a face; see Section~\ref{sec:definition} for more details. Bootstrapping the information from the universal cover to the plane is not straightforward. 

Investigating more general lattice models on this universal cover seems an interesting direction of research. Indeed, lattice models on Riemann surfaces have been studied extensively. Nevertheless, the theory mostly deals with surfaces of higher genius, while in our case we are facing (a discretization of) a simply connected planar Riemann surface with a logarithmic singularity at the origin.

The proof of Theorem~\ref{thm:decide} can be found in \cite{Dum12}. Nevertheless, other applications of parafermionic observables which are based on the same principle will be derived in this article, and most of the tools required for the proof of Theorem~\ref{thm:decide} will be used in other places as well. In addition, some of the ideas of the proof of Theorem~\ref{thm:main} allows one to simplify drastically the proof exposed in \cite{Dum12}. We therefore chose to include a streamlined version of the proof here (still, the details on parafermionic observables are omitted).
\medbreak
In order to conclude this section, let us mention that because the discrete contour integrals vanish, the parafermionic observable is therefore a discretization of a divergence-free differential form. For $q=2$ (which corresponds to the Ising model), further information can be extracted from local integrability and the observable satisfies a strong notion of discrete holomorphicity, called {\em $s$-holomorphicity}. In this case, the observable can be used to understand many properties on the model, including conformal invariance of the observable  \cite{CheSmi12,Smi10} and loops \cite{CheDumHon14,HonKyt13}, correlations \cite{CheIzy13,CheHonIzy15,Hon14,HonSmi11} and crossing probabilities \cite{BenDumHon14,DumHonNol11,CheDumHon13}. It can also be extended away from criticality \cite{BefDum12a,DumGarPet14}. We do not discuss special features of the $q=2$ case and we refer to the extensive literature for further information.

\subsection{Applications to the study of the critical phase for $1\le q\le 4$}\label{sec:applications}

The previous theorems have a large number of consequences regarding the understanding of the critical phase. We list some of them now.
\subsubsection{Mixing properties at criticality}\label{sec:applications 1}

The bound {\bf P5} on crossing probabilities enables us to study the spatial mixing properties at criticality. One may decouple events which are depending on edges in different areas of the space, and therefore compensate for the lack of independence. The next theorem illustrates this fact. It will be used in many occasions in the reminder of the paper.

\begin{theorem}[Polynomial ratio weak mixing under condition {\bf P5}] \label{thm:spatial mixing}
Fix $q\ge1$ such that Property {\bf P5} of Theorem~\ref{thm:main} is satisfied. There exists $\alpha>0$ such that for any $2k\leq n$ and for any event $A$ depending only on edges in $\Lambda_k$,
\begin{equation}\label{mixing}
\big|\P[\Lambda_n,p_c,q][\xi]{A}-\P[\Lambda_n,p_c,q][\psi]{A}\big|\leq \left(\frac{k}{n}\right)^\alpha\P[\Lambda_n,p_c,q][\xi]{A}\end{equation}
uniformly in boundary conditions $\xi$ and $\psi$.\end{theorem}

Together with the Domain Markov property, this result implies the following inequality for $2k\le m\le n$ (with $n$ possibly infinite),
\begin{equation*}
\big|\P[\Lambda_n,p_c,q][\xi]{A\cap B}-\P[\Lambda_n,p_c,q][\xi]{A}\P[\Lambda_n,p_c,q][\xi]{B}\big|\leq \left(\frac{k}{m}\right)^\alpha\P[\Lambda_n,p_c,q][\xi]{A}\P[\Lambda_n,p_c,q][\xi]{B},\end{equation*}
where the boundary conditions $\xi$ are arbitrary, $A$ is an event depending on edges of $\Lambda_k$ only, and $B$ is an event depending on  edges of $\Lambda_n\setminus\Lambda_m$.

\begin{remark}For $p\ne p_c(q)$, estimates of this type (with an exponential speed of convergence instead of polynomial) can be established by using the rate of spatial decay for the influence of a single vertex \cite{Ale98}. At criticality, the correlation between distant events does not boil down to correlations between points and a finer argument must be harnessed. Crossing-probability estimates which are \emph{uniform in boundary conditions} are the key in order to prove such results.
\end{remark}
\begin{remark}We will see several specific applications of this theorem in the next chapters. The most striking consequence is the fact that the dependence on boundary conditions can be forgotten as long as the boundary conditions are sufficiently distant from the set of edges determining whether the events under consideration occur or not.  For instance, it allows us to state several theorems in infinite volume, keeping in mind that most of these results possess natural counterparts in finite volume by using the fact that
$$c\P[p_c,q][]{A}\le\P[p_c,q,\Lambda_{2n}][\xi]{A}\le C\P[p_c,q][]{A}$$
for any event $A$ depending on edges in $\Lambda_n$ only, and any boundary conditions $\xi$ (the constants $c$ and $C$ are universal).
\end{remark}

\subsubsection{Consequences for the scaling limit} \label{sec:scaling limit}The phase transition being continuous, the scaling limit of the model is expected to be conformally invariant. This work opens new perspectives in the study of this scaling limit and we now mention several possible directions of research.
We refer the reader to Section~\ref{sec:observable} for the definition of Dobrushin domains and the exploration path.

\begin{conjecture}[Schramm, \cite{Sch07}]\label{conj:conformal invariance}
Let $0\le q\le 4$ and $p=p_c(q)$. Let $(\Omega,a,b)$ be a simply connected domain with two marked points on its boundary $a$ and $b$. Let $(\Omega_\delta,a_\delta,b_\delta)$ be a sequence of Dobrushin domains converging in the Carath\'eodory sense towards $(\Omega,a,b)$. The law of the exploration path $(\gamma_\delta)$ for critical random-cluster model with cluster-weight $q$ and Dobrushin boundary conditions in $(\Omega_\delta,a_\delta,b_\delta)$ converges, as the mesh size $\delta$ tends to zero, to the Schramm-Loewner Evolution between $a$ and $b$ in $\Omega$ with parameter $\kappa=\frac{4\pi}{\arccos(-\sqrt q/2)}$.\end{conjecture}

Schramm-Loewner Evolutions are very-well studied objects; see e.g. \cite{Law05}. This convergence would therefore lead to a deep understanding of the critical phase of the random-cluster models. For the $q=0$ case corresponding to the perimeter curve of the uniform spanning tree, the conjecture was proved by Lawler, Schramm and Werner \cite{LawSchWer04b}. The $q=2$ case was formally proved in \cite{CheDumHon14} even though the fundamental contribution leading to this result was achieved in \cite{Smi10}. Conjecture~\ref{conj:conformal invariance} is open for any other values of $q$ (even for the $q=1$ case corresponding to Bernoulli percolation).

Lawler, Schramm and Werner proposed a global strategy for proving Conjecture~\ref{conj:conformal invariance}, which can be summarized as follows:
\begin{enumerate}
\item Prove compactness of the family of exploration paths $(\gamma_\delta)_{\delta>0}$ and show that any sub-sequential limits can be parametrized as a Loewner chain (with a continuous driving process denoted by $W$). 
\item Prove the convergence of some discrete observables of the model.
\item Extract from the limit of these observables enough information to evaluate
the conditional expectation and quadratic variation of $W_t$ and use L\'evy's theorem to prove that $W_t$ is equal to $\sqrt{\kappa}B_t$, where $B_t$ is the standard Brownian motion. As a consequence, any sub-sequential limits must be the Schramm-Loewner Evolution of parameter $\kappa$.
\end{enumerate}
Step 1 of this program is provided by Theorem~\ref{thm:Loewner} below. 
 We refer to \cite{Law05} for details on Loewner chains.

Let $X$ be the set of continuous parametrized curves and $d$ be the distance on $X$ defined by
$$d(\gamma_1,\gamma_2)=\min_{\varphi_1:[0,1]\rightarrow I, \varphi_2:[0,1]\rightarrow J\text{ increasing}}\sup_{t\in I}|\gamma_1(\varphi_1(t))-\gamma_2(\varphi_2(t))|,$$
where $\gamma_1:I\rightarrow \mathbb C$ and $\gamma_2:J\rightarrow \mathbb C$. Note that $I$ and $J$ can be equal to $\R_+\cup\{\infty\}$.
\begin{theorem}\label{thm:Loewner}
Fix $1\le q\le 4$, $p=p_c(q)$ and a simply connected domain $\Omega$ with two marked points on its boundary $a$ and $b$. Let $(\Omega_\delta,a_\delta,b_\delta)$ be a sequence of Dobrushin domains converging in the Carath\'eodory sense towards $(\Omega,a,b)$. Define $\gamma_\delta$ to be the exploration path in $(\Omega_\delta,a_\delta,b_\delta)$ with Dobrushin boundary conditions. Then, the family $(\gamma_\delta)$ is tight and any sub-sequential limit $\gamma$ satisfies the following properties: 
\begin{itemize}[noitemsep]
\item[R1] $\gamma$ is almost surely a continuous non-intersecting curve from $a$ to $b$ staying in $\Omega$.
\item[R2] For any parametrization $\gamma:[0,1]\rightarrow \infty$, $b$ is a simple point, in the sense that $\gamma(t)=b$ implies $t=1$. Furthermore, almost surely $\gamma(t)$ is on the boundary of $\Omega\setminus\gamma[0,t]$ for any $t\in[0,1]$.
\item[R3] Let $\Phi$ be a conformal map from $\Omega$ to the upper half-plane $\mathbb H$ sending $a$ to $0$ and $b$ to $\infty$. For any parametrization $\gamma:[0,1]\rightarrow \R_+$,  the $h$-capacity of the hull $\widehat K_s$ of $\Phi(\gamma[0,s])$ tends to $\infty$ when $s$ approaches $1$. Furthermore, if $(\widehat K_t)_{t\ge0}$ denotes $(\widehat K_s)_{s\in[0,1]}$ parametrized by $h$-capacity, then $(\widehat K_t)_{t\ge0}$ is a Loewner chain with a driving process $(W_t)_{t\ge0}$ which is $\alpha$-H\"older for any $\alpha<1/2$ almost surely. Furthermore, there exists $\varepsilon>0$ such that for any $t>0$, $\mathbb E[\exp(\varepsilon W_t/\sqrt t)]<\infty$.
\item[R4] There exists $\alpha>0$ such that $\gamma$ has Hausdorff dimension between $1+\alpha$ and $2-\alpha$ almost surely.
\end{itemize}
\end{theorem}

Tightness criteria for random planar curves were first introduced in
\cite{AizBur99}. They were used as a key step in the proof of convergence
of interfaces to the Schramm-Loewner Evolution for Bernoulli site
percolation on the triangular lattice \cite{CamNew07}. These criteria were improved in
\cite{KemSmi12} to treat the case of random non-self-crossing planar
curves parametrized as Loewner chains. 
\medbreak
Step 2 represents the main challenge in the program outlined above (Step 3 is easy once Steps 1 and 2 have been achieved, see e.g. \cite[Section 13.2]{Dum13}).
Smirnov succeeded to perform Step 2 for $q=2$ using the fermionic observable \cite{Smi10}. He also proposed to consider the parafermionic observables introduced in \cite{Smi06} as a potential candidate for Step 2 in the case of general cluster-weights $q<4$ (let us mention that the choice of the observables in Step 2 are not determined uniquely). For completeness, let us mention a conjecture which, together with the results of this paper, would imply Conjecture~\ref{conj:conformal invariance}.
\begin{conjecture}[Smirnov]\label{FK parafermion}
Let $0<q< 4$, $p=p_c(q)$ and $(\Omega,a,b)$ be a simply connected domain with two points on its boundary. For every $z\in \Omega$,
\begin{eqnarray}
\frac 1{(2\delta)^\sigma}F_{\delta}(z)~\rightarrow~\phi'(z)^\sigma\quad\text{when }\delta\rightarrow 0
\end{eqnarray}
where 
\begin{itemize}[noitemsep]
\item for $\delta>0$, $F_\delta$ is the vertex parafermionic observable of \cite{Smi06} at $p_c(q)$ in $(\Omega_\delta,a_\delta,b_\delta)$.

\item $\sigma=1-\frac 2\pi\arccos (\sqrt q/2)$,  

\item $\phi$ is any conformal map from $\Omega$ to $\mathbb R\times(0,1)$ sending $a$ to $-\infty$ and $b$ to $\infty$.
\end{itemize}
\end{conjecture}

Let us mention that even though we are currently unable to prove Conjecture~\ref{FK parafermion} (and therefore Conjecture~\ref{conj:conformal invariance}), we are still able to obtain  results on the scaling limit. Indeed, the geometry of the random curve $\gamma$ can be easily related to the geometry of clusters boundaries at a discrete level. Keeping in mind that we are not able to prove that the scaling limit of cluster boundaries is well-defined, we may still extract sub-sequential limits and ask simple properties about these objects. For instance, property {\bf R4} of the previous theorem implies that any sub-sequential scaling limit of cluster boundaries of random-cluster models with $1\le q\le 4$ is a random fractal. 
The next theorem corresponds to another property of these sub-sequential scaling limits: it shows that macroscopic clusters touch the boundary of a smooth domain, for instance a rectangle, with free boundary conditions. 
\begin{theorem}\label{thm:strongest RSW}
Fix $1\le q<4$ and $\alpha>0$. There exists $\cl{cst:lower true RSW}>0$ such that for any $n\ge1$,
$$\P[{[0,\alpha n]\times[0,n]},p_c(q),q][0]{\calC_h([0,\alpha n]\times[0,n])}\ge \Cr{cst:lower true RSW}. $$
\end{theorem}
In the previous theorem, the free boundary conditions are directly on the boundary of the rectangle $[0,\alpha n]\times[0,n]$. This corresponds to the most direct generalization of the Russo-Seymour-Welsh theorem. 

\subsubsection{Consequences for critical exponents} 
Theorem~\ref{thm:decide} has several implications which are postponed to a future paper for so-called arm-events. Let us quickly mention that one can prove a priori bounds on the probability of arm-events, the so-called quasi-multiplicativity and extendability of these probabilities, as well as universal exponents.  These tools are crucial in order to compute critical exponents via the understanding of the scaling limit. 
Let us also mention that universal bounds can be deduced between different critical exponents.

Theorem~\ref{thm:decide} should also instrumental in the understanding of the near-critical regime, and in particular to derive the scaling relations between critical exponents (see \cite[Section~13.2.3]{Dum13} for more details).

In another direction, Theorem~\ref{thm:main} provides the relevant criteria in order to prove that the critical exponents of random-cluster models are universal on isoradial graphs (see \cite{GriMan14,CheSmi12} for the case of percolation and Ising), see \cite{DumLiMan15}.

\subsubsection{Consequences for Potts model} The coupling between Potts and random-cluster models enables one to transfer properties from one model to the other. In order to illustrate this fact, let us state the following theorem (many other results could be proved, but this would make this article substantially longer) which is a direct consequence of the previous theorems.

 \begin{theorem}
\label{prop:uniqueness}
For $q\in\{2,3,4\}$, there exists a unique Gibbs measure $\mu_{\Z^2,\beta_c(q),q}$ for the critical $q$-state Potts model. Furthermore, there exist $\eta_1,\eta_2>0$ such that for any $x\in \Z^2\setminus\{0\}$,
$$\frac{1}{|x|^{\eta_1}}\le\mu_{\Z^2,\beta_c(q),q}[\sigma_x=\sigma_0]-\frac1q\le \frac{1}{|x|^{\eta_2}}.$$
\end{theorem}

The main result should also have consequences for the Glauber dynamics of the Potts model. 
Recently, Lubetzky and Sly \cite{LubSly12} used spatial mixing properties of the Ising model to establish a long standing conjecture on the mixing time of the Glauber dynamics of the Ising model at criticality. As a key step, they employ the equivalent of {\bf P5} together with tools from the analysis of Markov chains, to  provide polynomial upper bounds on the inverse spectral gap of the Glauber dynamics (and also on the total variation mixing time). We plan to prove similar results for 3 and 4 states Potts models in a subsequent paper.

\paragraph*{Acknowledgments.} 
This project started during stays at IMPA and the university of Geneva, we thank both institutions for their support.  HDC and VT were supported by the ERC AG CONFRA, the NCCR SwissMap, as well as by the Swiss {FNS}.

\section{Preliminaries}

The norm $|\cdot|$ will denote the Euclidean norm.

\subsection{Graph notation}

\paragraph{Primal and dual graphs.}
The {\em square lattice} $(\mathbb Z^2,\mathbb E)$ is the graph with vertex set $\mathbb Z^2=\{(n,m):n,m\in\mathbb Z\}$ and edge set $\mathbb E$ given by edges between nearest neighbors. The square lattice will be identified with the set of vertices, i.e. $\bbZ^2$. The \emph{dual square lattice} $(\bbZ^2)^*$ is the dual graph of $\bbZ^2$. The vertex set is $(\tfrac12,\tfrac12)+\bbZ^2$ and the edges are given by nearest neighbors. The vertices and edges of $(\bbZ^2)^*$ are called {\em dual-vertices} and {\em dual-edges}. In particular, every edge of $\bbZ^2$ is naturally associated to a dual-edge, denoted by $e^*$,  that it crosses in its middle. 

Except otherwise specified, we will only consider subgraphs of $\bbZ^2$, $(\bbZ^2)^*$, and use the following notations.
For a graph $G$, we denote by $V_G$ its vertex set and by $E_G$ its edge set. Two vertices $x$ and $y$ are {\em neighbors (in $G$)} if $\{x,y\}\in E_G$. In such case, we write $x\sim y$. Furthermore, if $x$ is an end-point of $e$, we say that $e$ is {\em incident} to $x$.
Finally, the \emph{boundary} of $G$, denoted by $\partial G$, is the set of vertices of $G$ with strictly fewer than four incident edges in $E_G$.

For a graph $G\subset \bbZ^2$, we define $G^*$ to be the subgraph of $(\bbZ^2)^*$ with edge-set $E_{G^*}=\{e^*:e\in E_G\}$ and vertex set given by the end-points of these dual-edges.

Let $\Lambda_n$ be the subgraph of $\bbZ^2$ induced by the vertex set $[-n,n]^2$.

\paragraph{Connectivity properties in graphs.}
A path in $\mathbb Z^2$ is a sequence of vertices $v_0,\dots,v_n$ in
$\mathbb Z^2$ such that $v_i\sim v_{i+1}$ for any $0\le i<n$. The path
will often be identified with the set of edges $(v_0,v_1),
\{v_1,v_2\},\dots,\{v_{n-1},v_n\}$. The path is said {\em to start at $x$
  and to end at $y$} if $v_0=x$ and $v_n=y$. If $x=y$, the path is
called a {\em circuit}.

Two vertices $x$ and $y$ of $G$ are {\em connected} if there exists a
path of edges in $E_G$ from $x$ to $y$. A graph is said to be {\em
  connected} if any two vertices of $G$ are connected. The {\em
  connected components} of $G$ will be the maximal connected subgraphs
of $G$. For three sets $X,Y\subset V_G$ and $F\subset E_G$, we say that
$X$ is connected to $Y$ in $F$ (we denote this fact by $X \lr[F] Y$) if there exist
two vertices $x\in X$ and $y\in Y$ and a path of edges in $F$ starting
at $x$ and ending at $y$.

\subsection{Background on the random-cluster model}\label{sec:definition}


This section is devoted to a very brief description of the tools we will use in the proofs of the next sections. The reader may consult \cite{Gri06} or \cite{Dum13} for more details, proofs and references. 
\paragraph{Random-cluster model and critical point.}
Let $G=(V_G,E_G)$ be a finite subgraph of $\Z^2$ or $\mathbb U$. A \emph{configuration} on $G$ is an element $\omega=\{\omega(e):e\in E_G\}\in \{0,1\}^{E_G}$. An edge $e$ is said to be {\em open} if $\omega(e)=1$, and {\em closed} otherwise. A configuration $\omega$ can be seen as a subgraph of $G$, whose vertex set is $V_G$ and edge set if the set of open edges $\{e\in E_G:\omega(e)=1\}$. A path (resp. circuit) in $\omega$ will be called an {\em open path} (resp. {\em open circuit}). Two sets $X$ and $Y$ are {\em connected} in $\omega$ if there exists an open path from $X$ to $Y$. The connected components of $\omega$ will be called 
\emph{clusters}. 

The boundary conditions on $G$ are given by a partition $\xi=P_1\sqcup \dots\sqcup
P_k$ of $\partial G$. From a configuration $\omega$, define
$\omega^\xi$ to be the graph with vertex set $V_G$ and edge set given by
edges of $\omega$ together with edges of the form $\{x,y\}$, where $x$
and $y$ belong to the same $P_i$. In such case, the vertices $x$ and $y$ are sometimes said to be {\em wired} together.

\begin{definition}\label{def:RCM}
  The {\em random-cluster measure on $G$ with edge-weight $p$,
    cluster-weight $q$, and boundary conditions $\xi$} is defined by
  the formula
$$\P[G,p,q][\xi]{\omega}=\frac{p^{o(\omega)}(1-p)^{c(\omega)}q^{k(\omega^\xi)}}{Z^\xi_{G,p,q}},$$
where $k(\omega^\xi)$ is the number of connected components of the
graph $\omega^\xi$. As usual, $Z^\xi_{G,p,q}$ is defined in such a way
that the sum of the weights over all possible configurations equals
1. 
\end{definition}
The following boundary
conditions play a special role in this article.
 The partition $\xi$ composed of singletons only is called the {\em free boundary conditions} and is denoted by $\xi=0$. It corresponds to no additional connections. 
The partition $\xi=\{\partial G\}$ is called the {\em wired boundary conditions} and is denoted by $\xi=1$. It corresponds to the fact that all the boundary vertices are connected by boundary conditions.

\paragraph{Infinite-volume measures and critical point.}
We do not aim for a complete description, or even a formal definition
of random-cluster measures on $\Z^2$ and we refer to \cite[Chapter
4]{Gri06} for details. When $q\ge 1$, infinite-volume random-cluster
measures can be defined by taking the limit of finite-volume measures.
In particular, the sequence of measures $\PP[\Lambda_n,p,q][1]$ (resp.
$\PP[\Lambda_n,p,q][0]$) converges to the {\em infinite-volume measure
  with wired (resp. free) boundary conditions} $\PP[\Z^2,p,q][1]$
(resp. $\PP[\Z^2,p,q][0]$). Furthermore, there exists
$p_c=p_c(q)\in(0,1)$ such that for $p\ne p_c(q)$, the infinite-volume
measure $\PP[\Z^2,p,q][]$ is unique and satisfies
\begin{equation}
  \P[\Z^2,p,q][]{0\leftrightarrow\infty}=
  \begin{cases} 0&\text{ if $p<p_c(q)$,}\\ 
    \theta(p,q)>0&\text{ if $p>p_c(q)$.}
  \end{cases}
\end{equation}

The critical parameter $p_c(q)$ was proved to be equal to $\sqrt
q/(1+\sqrt q)$ in \cite{BefDum12b}.

\paragraph{Positive association when $q\ge 1$.}
Denote the product ordering on $\{0,1\}^E$ by $\le$. An event $\cal A$
depending on edges in $E$ only is {\em increasing} if for any
$\omega'\ge \omega$, $\omega\in \cal A$ implies $\omega'\in \cal A$.

For $q\ge 1$, the random-cluster model satisfies important properties
regarding increasing events. The first such property is the FKG
inequality \cite[Theorem 3.8]{Gri06}: for any boundary conditions
$\xi$ and  for any increasing events $\cal A$ and $\cal B$,
\begin{equation}\label{FKG} \P[G,p,q][\xi]{\cal A\cap \cal B}\ge \P[G,p,q][\xi]{\cal A}\P[G,p,q][\xi]{\cal B}. 
\end{equation}
The second important property is the comparison between boundary
conditions \cite[Lemma 4.56]{Gri06}: for any increasing event $\cal A$ and for any $\xi\ge\psi$,
\begin{equation}\label{comparison}\P[G,p,q][\xi]{\cal A}\ge
  \P[G,p,q][\psi]{\cal A}. 
\end{equation}
Here, $\xi\ge \psi$ if the partition $\psi$ is finer than the
partition $\xi$ (any wired vertices in $\psi$ are wired in $\xi$). In
such case, $\xi$ is said to {\em dominate} $\psi$ or equivalently
$\psi$ is said to be {\em dominated by} $\xi$. The free (resp. wired)
boundary conditions are dominated by (resp. dominates) any other
boundary conditions.

 \paragraph{Domain Markov property and insertion tolerance.}
 Consider a subgraph $G'$ of $G$. The following proposition
 describes how the influence of the configuration outside $G'$ on the
 measure within $G'$ can be encoded using appropriate boundary
 conditions $\xi$. 
 \begin{proposition}[Domain Markov Property]
  \label{thm:markovProperty}
   Let $p\in[0,1]$, $q>0$ and $\xi$ some boundary conditions. Fix $G'\subset
  G$. Let $X$ be a random variable measurable which is measurable with respect to edges in
  $E_{G'}$. Then,
  $$\P[G,p,q][\xi]{X|\omega_{|E_{G}\setminus E_{G'}}=\psi
    }=\P[G',p,q][\psi^\xi]{X},$$
    for any $\psi\in\{0,1\}^{E_G\setminus E_{G'}}$. Above,
  $\psi^\xi$ is the partition on $\partial G'$ obtained as follows:
  two vertices $x,y\in \partial G'$ are in the same element of the
  partition if they are connected in $\psi^\xi$.
\end{proposition}
 
\newcommand{\Cit}{c_{\scriptscriptstyle \rm IT}}
 The previous proposition has the following corollary, called insertion tolerance.
 For $p,q>0$, there exists $\Cit>0$ such that for any finite graph $G$, any $\omega\in\{0,1\}^{E_G}$, and any boundary conditions $\xi$, $\P[G,p,q][\xi]{\omega}\ge\Cit^{|G|}$.

 \paragraph{Dual representation}\label{sec:dual}
 A configuration $\omega$ on $G$ can be uniquely associated to a {\em
   dual configuration} $\omega^*$ on the dual graph $G^*$ defined as
 follows:
 set $\omega^*(e^*)=1$ if $\omega(e)=0$ and $\omega^*(e^*)=0$ if
 $\omega(e)=1$. A dual-edge $e^*$ is said to be {\em dual-open} if
 $\omega^*(e^*)=1$, it is {\em dual-closed} otherwise. A {\em
   dual-cluster} is a connected component of $\omega^*$. We extend the
 notion of {\em dual-open} path in a trivial way. For two sets
 $X,Y\subset V_{G^*}$, we set $X\lr[*]Y$ if there exists a dual-open path
 from $X$ to $Y$.

If
$\omega$ is distributed according to $\PP[G,p,q][\xi]$, then
$\omega^*$ is distributed according to $\PP[G^*,p^*,q^*][\xi^*]$ with
$q^*=q$ and $\frac{pp^*}{(1-p)(1-p^*)}=q$. In particular, for $p=p_c(q)$ we find $p^*=p=p_c(q)$. The boundary conditions
$\xi^*$ can be deduced from $\xi$ in a case by case manner. We will
mostly be interested in the case of $\xi=0$ or $1$, for which
$\xi^*=1$ and 0 respectively.

\paragraph{Russo-Seymour-Welsh estimate with wired boundary conditions.}
For a rectangle $R=[a,b]\times[c,d]$, let $\calC_h(R)$ be the event that there exists an open path in $R$ from $\{a\}\times[c,d]$ to $\{b\}\times[c,d]\subset\bbZ^2$. 
Such a crossing is called a {\em horizontal crossing} of $R$. Similarly, one defines $\calC_v(R)$ to be the event that there exists an open path in $R$ from $[a,b]\times\{c\}$ to $[a,b]\times\{d\}$ (such a path is called a {\em vertical crossing}). For a rectangle $R^*=[s,t]\times[u,v]\subset (\bbZ^2)^*$ (note that $s,t,u,v$ are half-integers), let $\calC^*_h(R^*)$ be the event that there exists a dual-open dual-path in $R^*$ from $\{s\}\times[u,v]$ to $\{t\}\times[u,v]$ (such a path is called an {\em horizontal dual-crossing}), and similarly $\calC^*_v(R^*)$ is the event that there exists a dual-open dual-path in $R^*$ from $[s,t]\times\{u\}$ to $[s,t]\times\{v\}$ (such a path is called a {\em vertical dual-crossing}).

The following result  will be important in the proof. Let us restate it here.
\newcommand{\Crsw}{c_{\scriptscriptstyle \rm RSW}}
\begin{theorem}[{\cite[Corollary 9]{BefDum12b}}]
  \label{thm:weakRSW}
   For $\alpha>1$ and $q\geq 1$, there exists $\Crsw >0$ such that for every 
  $0<m<\alpha n$,
  \begin{equation}
    \P[\Z^2,p_c,q][1]{\mathcal C_h([0,m]\times[0,n])}\geq 
    \Crsw .
  \end{equation}
\end{theorem}
Note that this bound on crossing probabilities is much weaker than {\bf P5}, since it holds only for wired boundary conditions. Let us mention a particularly interesting case of non wired boundary conditions. We refer the
interested reader to \cite{BefDum12b} for more details.
Consider the random-cluster measure on $\Lambda_n$ with wired boundary conditions on top and
bottom, and free elsewhere (we call these boundary conditions mixed).
Then
\begin{equation}\label{crossing square mixed}\P[\Lambda_n,p_c,q][\rm mixed]{\mathcal C_h(\Lambda_n)}\ge \frac1{1+q^2}.\end{equation}
\bigskip
\bigskip

{\em From now on, we fix $q\ge 1$ and $p=p_c(q)$. In order to lighten the notation, we drop the reference to $p$ and $q$ and simply write $\PP[G][\xi]$ instead of $\PP[G,p_c(q),q][\xi]$. }

\section{Proof of Theorem~\ref{thm:main}}
\label{sec:proof1}

\subsection{A preliminary result}

%

Before diving into the proof, let us mention three useful equivalent formulations of {\bf P5}.
For $z\in \bbR^2$, define $\calA_n(z)$ to be the event that there exists
an open circuit in the annulus $z+(\Lambda_{2n} \setminus
\Lambda_{n})$ surrounding $z$.
Also define $\calA_n=\calA_n(0)$.

\begin{proposition}\label{prop:classical RSW}
Recall that we are at $q\ge1$ and $p=p_c(q)$. The following assertions are equivalent;
\begin{itemize}
\item[{\bf P5}] For any $\alpha>0$, there exists
  $c_1=c_1(\alpha)>0$ such that for all $n\ge 1$, we have
$$\P[{[-n,(\alpha+1)n]\times[-n,2n]}][0]{\calC_h([0,\alpha n]\times[0,n])}\ge c_1.$$
\item[{\bf P5a}]  There exists $c_2>0$ such that for all $n\geq 1$,
   $$\P[\Lambda_{2n}\setminus\Lambda_{n+1}][0]{\calA_n}\geq c_2.$$
\item[{\bf P5b}] For any $R\ge 2$, there exists $c_3=c_3(R)>0$ such that for all $n\geq 1$,
   $$\P[\Lambda_{Rn}][0]{\calA_n}\geq c_3.$$
   \item[{\bf P5c}] For any $\alpha>0$, there exists
  $c_4=c_4(\alpha)>0$ such that for all $n\ge 1$ and for
  all boundary conditions $\xi$ on the boundary of $[-n,(\alpha+1)n]\times[-n,2n]$, we have
$$c_4\le \P[{[-n,(\alpha+1)n]\times[-n,2n]}][\xi]{\calC_h([0,\alpha n]\times[0,n])}\le 1-c_4.$$
\end{itemize}
\end{proposition}

The last condition justifies the fact that the result is uniform with respect to boundary conditions.

\begin{proof} The proof of {\bf P5a}$\Rightarrow${\bf P5b} and {\bf P5c}$\Rightarrow${\bf P5} are obvious by comparison between boundary conditions. In order to prove {\bf P5}$\Rightarrow${\bf P5a}, consider the four rectangles
\begin{align*}
R_1&:=[4n/3,5n/3]\times[-5n/3,5n/3],\\
R_2&:=[-5n/3,-4n/3]\times[-5n/3,5n/3],\\
R_3&:=[-5n/3,5n/3]\times[4n/3,5n/3],\\
R_4&:=[-5n/3,5n/3]\times[-5n/3,-4n/3].
\end{align*}
If the intersection of $\calC_v(R_1)$, $\calC_v(R_2)$, $\calC_h(R_3)$ and $\calC_h(R_4)$ occurs, then $\calA_n$ occurs. In particular, the FKG inequality and the comparison between boundary conditions implies that $c_2$ can be chosen to be equal to $c_1(10)^4$.
\medbreak
Let us now turn to the proof of {\bf P5b}$\Rightarrow${\bf P5}. We start by the lower
bound.  Fix some $R\ge 2$ as in {\bf P5b}  and the corresponding $c_3>0$. Let $\alpha>0$. For $n\ge 4R$, the intersection of the events $\calA_{n/(2R)}[(j\lfloor\tfrac{n}{R}\rfloor,\tfrac n2)]$ for $j=0,\dots,\lceil R\alpha\rceil$ is included in $\calC_h ([0,\alpha n]\times[0,n])$. The FKG inequality implies 
$$\P[{[-n,(\alpha+1)n]\times[-n,2n]}][0]{\calC_h([0,\alpha n]\times[0,n])}\ge c_3^{1+R\lceil\alpha\rceil}.$$
By comparison between boundary conditions, we obtain the lower bound for every $\xi$. 

The upper bound may be obtained from this lower bound as follows. By comparison between boundary conditions once again, it is sufficient to prove the upper bound for the wired boundary conditions. In such case, the complement of $\calC_h([0,\alpha n]\times[0,n])$ is $\calC_v^*([\tfrac12,\alpha n-\tfrac12]\times[-\tfrac12,n+\tfrac12])$. Since the dual of the wired  boundary conditions are the free ones, the boundary conditions for the dual measure are free. We can now harness {\bf P5b} for the dual model to construct a dual path from top to bottom with probability bounded away from 0. This finishes the proof.
\bigbreak
\noindent It remains to prove that {\bf P5} implies {\bf P5c} to conclude. Define 
\begin{align*}
R&:= [0,\alpha n]\times[0,n],\\
\overline R&:=[-n,(\alpha+1)n]\times[-n,2n],\\
R^*&:=[\tfrac12,\alpha n-\tfrac12]\times[-\tfrac12,n+\tfrac12].\end{align*}
First, \eqref{comparison} implies that 
\begin{align*}c_1&\le\P[\overline R][0]{\calC_h(R)}\le \P[\overline R][\xi]{\calC_h(R)},\end{align*}
where the first inequality is due to {\bf P5}. Now, we wish to prove the upper bound. Let us assume without loss of generality that $\alpha<1$. As mentioned above, the complement of the event $\calC_h(R)$ is the event $\calC_v^*(R^*)$. We find that
\begin{align*} c_1&\le \P[\overline R][1]{\calC_h(R)}=1-\P[\overline R][1]{\calC_v^*(R^*)}\le 1-c_1(1/\alpha).\end{align*}
In the last inequality, we used the lower bound proved previously for $\alpha'=1/\alpha$ and $n'=\alpha n$. To justify that we can do so, observe that the dual of wired boundary conditions are the free ones, and that since $\alpha<1$, the dual graph of $\overline R$ contains a rotated version of $[\tfrac12-\alpha n,\alpha n-\tfrac12+\alpha n]\times[-\tfrac12-\alpha n,n+\tfrac12+\alpha n]$.\end{proof}

\subsection{Proof of Theorem~\ref{thm:main}: easy implications}In order to isolate the hard part of the proof, let us start by checking the four ``simple'' implications {\bf P1}$\Rightarrow${\bf P2}, {\bf P2}$\Rightarrow${\bf P3}, {\bf P3}$\Rightarrow${\bf P4} and {\bf P5}$\Rightarrow${\bf P1}.

\medbreak\noindent
\emph{Property {\bf P1} implies {\bf P2}:} This implication is classical, see e.g. \cite[Corollary~4.23]{Dum13}.
\medbreak\noindent
\emph{Property {\bf P2} implies {\bf P3}:} If ${\bf P2}$ holds, \begin{align*}
  (2n+1)\,\P[\bbZ^2][0]{0\leftrightarrow\partial\Lambda_n}
  &=(2n+1)\,\P[\bbZ^2][1]{0\leftrightarrow\partial\Lambda_n}\ge \sum_{x\in\{0\}\times[-n,n]}\P[\bbZ^2][1]{x\leftrightarrow(x+\partial\Lambda_n)}\\
  &\ge
  \P[\bbZ^2][1]{\mathcal C_v([-n,n]\times[0,n])}\ge c, 
\end{align*}
where $c>0$ is a constant independent of $n$.
The first equality is due to the uniqueness of the infinite-volume measure given by {\bf P2} and the second inequality by
Theorem~\ref{thm:weakRSW}. This leads to
$$\sum_{x\in\partial\Lambda_n}\P[\bbZ^2][0]{0\lr x}\ge
\P[\bbZ^2][0]{0\leftrightarrow\partial\Lambda_n}\ge\frac{c }{2n+1}.$$
As a consequence, $\displaystyle\sum_{x\in\bbZ^2}\P[\bbZ^2][0]{0\lr x}=\infty$ and {\bf P3} holds true. 

\medbreak\noindent
\emph{Property {\bf P3} implies {\bf P4}:} Assume that {\bf P4} does not hold. In such case, the fact that 
\begin{align*}\P[\bbZ^2][0]{0\longleftrightarrow (n+m)x}&\ge\P[\bbZ^2][0]{\{0\longleftrightarrow nx\}\cap\{nx\longleftrightarrow (n+m)x\}}\\
&\ge\P[\bbZ^2][0]{0\longleftrightarrow nx}\P[\bbZ^2][0]{0\longleftrightarrow mx}\end{align*}
(the second inequality follows from the FKG inequality) implies the existence of $M>0$ such that for every $x=(x_1,x_2)\in \bbZ^2$,
$$\P[\bbZ^2][0]{0\longleftrightarrow x}\le {\rm e}^{-|x|/M}.$$ 
Summing over every $x\in \bbZ^2$ gives $$\displaystyle\sum_{x\in\Z^2}\P[\bbZ^2][0]{0\leftrightarrow x}<\infty$$ and thus {\bf P3} does not hold.
\medbreak\noindent
\emph{Property {\bf P5} implies {\bf P1}:} Recall that {\bf P5} implies {\bf P5a}. We now prove a slightly stronger result which obviously implies {\bf P1} and will be useful  later in the proof.
\begin{lemma}\label{lem:one arm}
Property {\bf P5a} implies that there exists $\ep>0$ such  that for any $n\ge 1$,
$$\P[\bbZ^2][1]{0\leftrightarrow\partial\Lambda_n}\le n^{-\ep}.$$
\end{lemma}

\begin{proof}
Let $k$ be such that $2^k\le n<2^{k+1}$. Also define the annuli $A_j=\Lambda_{2^j}\setminus\Lambda_{2^{j-1}-1}$ for $j\ge 1$. We have
\begin{align*}\P[\bbZ^2][1]{0\longleftrightarrow\partial\Lambda_n}&\le \prod_{j=1}^k\P[\bbZ^2][1]{\Lambda_{2^{j-1}}\lr[A_j]\partial\Lambda_{2^j}\left|\bigcap_{i>j}\big\{\Lambda_{2^{i-1}}\lr[A_i]\partial\Lambda_{2^i}\big\}\right.}\\
&\le \prod_{j=1}^k\P[A_j][1]{\Lambda_{2^{j-1}}\lr[A_j]\partial\Lambda_{2^j}}.\end{align*}
In the second line, we used the fact that the event upon which we condition depends only on edges outside of $\Lambda_{2^j}$ or on $\partial\Lambda_{2^j}$ together with the comparison between boundary conditions.

Now, the complement of $\Lambda_{2^{j-1}}\lr[A_j]\partial\Lambda_{2^j}$ is the event that there exists a dual-open circuit in $A_j^*$ surrounding the origin. Property {\bf P5a} implies
 that this dual-open circuit exists with probability larger than or equal to $c>0$ independently of $n\ge1$. This implies that
\begin{align*}
\P[\bbZ^2][1]{0\longleftrightarrow\partial\Lambda_n}&\le \prod_{j=1}^k(1-c)=(1-c)^k\le (1-c)^{\log n/\log 2}.\end{align*}
The proof follows by setting $\ep=-\frac{\log(1-c)}{\log 2}$. 
\end{proof}

\begin{remark}The proof of the previous lemma illustrates the need for {\em bounds which are uniform with respect to boundary conditions}. Indeed, it could be the case that the $\PP[][1]$-probability of an open path from the inner to the outer sides of $A_j$ is bounded away from 1, but conditioning on the existence of paths in each annulus $A_i$ (for $i<j$) could favor open edges drastically, and imply that the probability of the event under consideration is close to $1$.
\end{remark}

\subsection{Proof of Theorem~\ref{thm:main}: {\bf P4} implies {\bf P5}}
Recall from Proposition~\ref{prop:classical RSW}  that {\bf P5} is equivalent to {\bf P5b} and we therefore choose to prove that {\bf P4} implies {\bf P5b} when $R=8$. 
The proof follows two steps. 
First, we prove that either {\bf P5b} holds or $\phi^0(0\leftrightarrow\partial\Lambda_n)$ tends to 0 stretched-exponentially fast.  Second, we prove that if the speed of convergence is stretched exponential, it is in fact exponential.

\begin{proposition}\label{prop:dualityP_n}
Exactly one of these
  two cases occurs :
  \begin{enumerate}
  \item $\displaystyle \inf_{n\geq 1} \P[\Lambda_{8n}][0]{\mathcal A _n} >0$.
  \item There exists $\alpha>0$ such that for any $n\ge1$,
  $$\P[\bbZ^2][0]{0\longleftrightarrow\partial\Lambda_n}\le \exp(-n^\alpha).$$
  \end{enumerate}
\end{proposition}

%

First, consider 
the strip $\bbS=\Z\times[-n,3n]$, and the boundary conditions $\xi$
defined to be wired on $\Z\times\{3n\}$, and free on $\Z\times
\{-n\}$. Let $\phi^{1/0}_{\bbS}$ be the associated random-cluster measure. Recall that boundary conditions at infinity are not relevant since the strip is essentially one dimensional.
\begin{lemma}\label{lem:stripRSW}
  For all $k\geq 1$, there exists a constant
  $c=c(k)>0$ such that, for all $n\geq 1$,
  \begin{equation}
    \label{eq:1}
    \Ps{\mathcal C_h([-kn,kn]\times[0,2n])} \geq c.  
\end{equation}
\end{lemma}
\begin{proof}
Fix $n,k\ge1$. We will assume that $n$ is divisible by 9 (one may adapt the argument for general values of $n$). By duality, the complement of 
$\calC_h ([-kn,kn]\times[0,2n])$
is $\calC_v^* ([-kn+\tfrac12,kn-\tfrac12]\times[-\tfrac12,2n+\tfrac12])$. Therefore, either \eqref{eq:1} is true for $c=1/2$, or
\[\Ps{\calC_v^* ([-kn+\tfrac12,kn-\tfrac12]\times[-\tfrac12,2n+\tfrac12])}\ge1/2.\] We
assume that we are in this second situation for the rest of the proof.

The dual of the measure on
the strip with free boundary conditions on the bottom and wired on
the top is the measure on the strip with free boundary conditions on
the top and wired on the bottom. This measure is the image of $\phi_{\bbS}^{1/0}$ under the orthogonal reflection with respect to the horizontal line $\mathbb R\times\{n-\tfrac14\}$ composed with a translation by the vector $(\tfrac12,0)$. We thus obtain that 
\begin{align*}\Ps{\calC_v
  ([-kn,kn]\times[0,2n])}&\ge \Ps{\calC_v
  ([-kn,kn-1]\times[-1,2n])}\\
  &=\Ps{\calC_v^* ([-kn+\tfrac12,kn-\tfrac12]\times[-\tfrac12,2n+\tfrac12])}\\
  &\ge1/2.\end{align*}
   Partitioning the segment $[-kn, kn]\times\{0\}$ into
the union of $18k$ segments of length $\lambda:=n/9$ (note that $\lambda$ is an integer), the union bound gives us
\begin{equation}
\label{eq:2}
  \Ps{I\lr{} \Z \times \{2n\}}\ge \frac1{36k}=:c_1,
\end{equation}
where $I=[4\lambda,5\lambda]\times\{0\}$. For future reference, let us also introduce the segment $J=[6\lambda,7\lambda]\times\{0\}$.

Define the rectangle $R=[0,9\lambda]\times[0,2n]$.  
When the event estimated in Equation~\eqref{eq:2} is realized, there
exists an open path in $R$ connecting
$I$ to the union of the top, left and right boundaries of $R$.
 Using the reflection with respect to the vertical line $\{\tfrac n2\}\times\bbR$, we find that at least one of the two following inequalities occurs:
 \medbreak
  Case 1: $\Ps{I\stackrel{R}{\longleftrightarrow} [0,n]\times\{2n\}}\ge c_1/3.$
  \medbreak
  Case 2: $\Ps{I\stackrel{R}{\longleftrightarrow}
    \{0\}\times[0,2n] }\ge c_1/3.$
     \begin{figure}[htbp]
       \begin{center}      
         \includegraphics[width=1.00\textwidth]{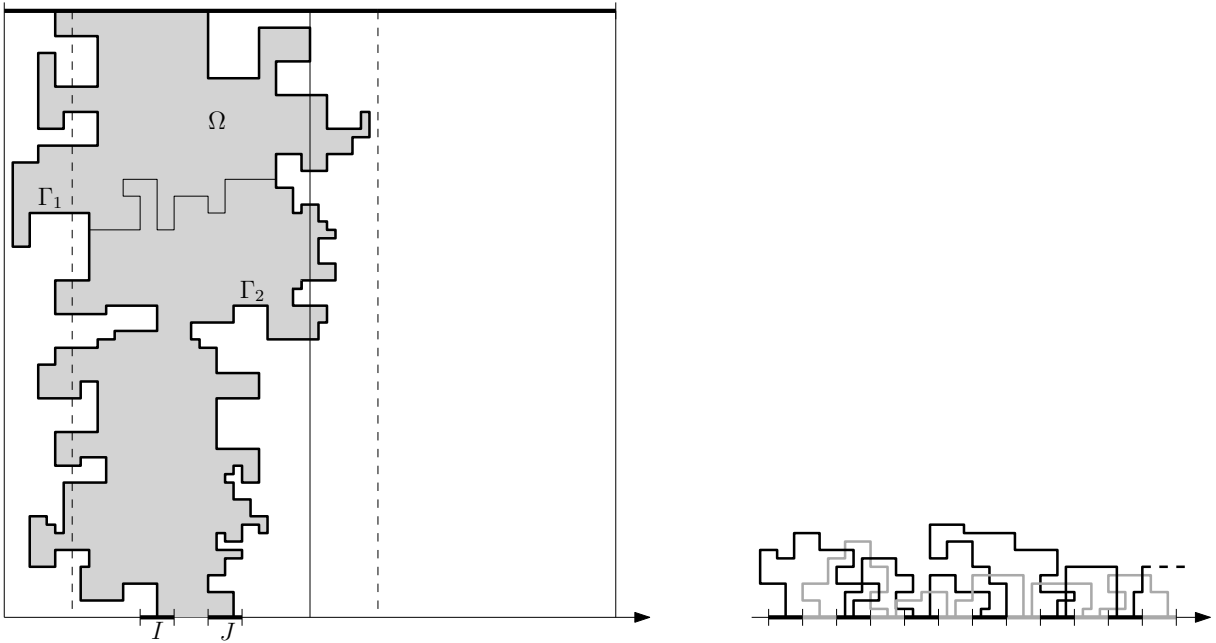}
         \caption{The construction in Case 1 with the two paths
           $\Gamma_1$ and $\Gamma_2$ and the domain $\Omega$ between
           the two paths. On the right, a combination of paths
           creating a long path from left to right.\label{fig:abcd}
         }\end{center}
 \end{figure}
\bigbreak
\noindent{\em Proof of~\eqref{eq:1} in Case 1:}    Consider the event that
  there exist
 \medbreak
 (i) an open path from $I$ to the top of $[0,2n]^2$ contained in $[0,2n]^2$,
 
 (ii) an open path from $J$ to the top of $[0,2n]^2$ contained in $[0,2n]^2$,
 
 (iii) an open path connecting these two paths in $[0,2n]^2$.
\medbreak 
 Each path in (i) and (ii) exists with probability larger than
 $c_1/3$ (since $R$ and $(2\lambda,0)+R$ are included in $[0,2n]^2$). Furthermore, let $\Gamma_1$ be the left-most path satisfying (i) and $\Gamma_2$ the right-most path satisfying (ii); see Fig.~\ref{fig:abcd}. The subgraph of $[0,2n]^2$
 between $\Gamma_1$ and $\Gamma_2$ is denoted by $\Omega$. Conditioning on
 $\Gamma_1$ and $\Gamma_2$, the boundary conditions on $\Omega$ are
 wired on $\Gamma_1$ and $\Gamma_2$, and dominate the free boundary
 conditions on the rest of $\partial \Omega$. We deduce that 
 boundary conditions on $\Omega$ dominate boundary conditions inherited by wired
 boundary conditions on the left and right sides of the box $[0,2n]^2$,
 and free on the top and bottom sides. As a consequence of
 \eqref{crossing square mixed}, conditionally on $\Gamma_1$ and
 $\Gamma_2$, there exists an open path in $\Omega$ connecting
 $\Gamma_1$ to $\Gamma_2$ with probability larger than $1/(1+q^2)$. In
 conclusion,  \begin{equation}\label{eq:4}
 \Ps{I\lr[{[0,2n]^2}]J}\ge\Ps{\text{(i), (ii) and (iii) occur}}\ge \left(\frac{c_1}3\right)^2\times\frac{1}{(1+q^2)}.
 \end{equation}
 For $x=j\lambda$, where $j\in\{-9k-5,\dots,9k-6\}$, define the translate of the event considered in \eqref{eq:4}:
 $$A_x:=\big\{\,(x+I)\lr[{x+[0,2n]^2}](x+J)\,\big\}.$$ If $A_x$ occurs for every such $x$, we obtain an open crossing from left to right in $[-kn,kn]\times[0,2n]$. The FKG inequality implies that this happens with probability larger than $\left(\frac{c_1^2}{9(1+q^2)}\right)^{20k}$.      

\bigbreak\noindent{\it Proof of~\eqref{eq:1} in Case 2:} Define the rectangle
  $R'=[4\lambda,9\lambda]\times[0,2n]$. Note that in Case 2, $J$ is connected to one side of $[2\lambda,11\lambda]\times[0,2n]$ with probability bounded from below by $c_1/3$, hence the same is true for $R'$ (since $[2\lambda,11\lambda]\times[0,2n]$ is wider than $R'$). Consider the event that
  there exist
\medbreak  
  (i) an open path from $I$ to the right side of $R$ contained in $R$,
  
  (ii) an open path from $J$ to the left side of $R'$ contained in $R'$,
  
   (iii) an open path connecting these two paths in $[0,2n]^2$.
\medbreak  
  The first path occurs with probability larger than $c_1/3$,
  and the second one with probability larger than $c_1/6$ (there exists a path to one of the sides with probability at least $c_1/3$, and therefore by symmetry in $R'$ to the left side with probability larger than $c_1/6$).
  By the
  FKG inequality, the event that both (i) and (ii) occur has probability larger than $c_1^2/18$. We now wish to prove that conditionally on (i) and (ii) occurring, the event (iii) occurs with probability bounded from below uniformly in $n$. 
  
  Define the segments $K(y,z)=\{4\lambda\}\times[y,z]$ for $y\le z\le \infty$. They are all subsegments of the vertical line of first coordinate equal to $4\lambda$.
  
  Consider the right-most open path $\Gamma_1$ satisfying (ii).  It intersects the segment $K(0,2n)$ at a unique point with second coordinate denoted by $y$. Also consider the left-most open path $\tilde\Gamma_2$ satisfying (i). Either $\Gamma_1$ and $\tilde\Gamma_2$ intersect, or they do not. In the first case, we are already done since (iii) automatically occurs. In the second, we consider the subpath $\Gamma_2$ of $\tilde\Gamma_2$ from $I$ to the first intersection with $K(y,2n)$ (this intersection must exist since $\tilde\Gamma_2$ goes to the right side of $R'$). Let us now show that $\Gamma_1$ and $\Gamma_2$ are connected with good probability. Note the similarity with the construction in \cite{BefDum12b} with symmetric domains, except that the lattice is not rotated here. The proof is therefore slightly more technical and we choose to isolate it from the rest of the argument.
  \medbreak
\noindent{\em Claim: There exists $c_2>0$ such that for any possible realizations $\gamma_1$ and $\gamma_2$ of $\Gamma_1$ and $\Gamma_2$,}
\begin{equation}\label{eq:5}\Ps{\g_1\stackrel{R}{\longleftrightarrow}\g_2\,\Big|\,\Gamma_1=\gamma_1,\Gamma_2=\gamma_2}\ge c_2.\end{equation}
\begin{proof}
Fig.~\ref{fig:domainOmega} should be very helpful in order to follow this proof.  Construct the subgraph $\Omega$ ``between $\gamma_1$ and $\gamma_2$'' formally delimited by:
\begin{itemize}[nolistsep,noitemsep]
\item  the arc $\gamma_2$,
\item  the segment $[0,n]\times\{0\}$,
\item  the arc $\gamma_1$,
\item the segment $K(y+1,2n)$ {\em excluded} (the vertices on this segment are not part of the domain). \end{itemize}
\medbreak\noindent
We wish to compare $\Omega$ (left of Fig.~\ref{fig:domainOmega}) to a reference domain $D$ (center of Fig.~\ref{fig:domainOmega}) defined as the upper half-plane minus the edges intersecting $\{4\lambda-\tfrac12\}\times(y,\infty)$ and
define the boundary conditions $mix$ on $D$ by:
\begin{itemize}[noitemsep,nolistsep]
\item  wired boundary conditions on $K(y,\infty)$ and $A:=(-\infty,4\lambda]\times\{0\}$;
\item  wired boundary conditions at infinity (by this we mean that we take the limit of measures on $D\cap\Lambda_n$, with wired boundary conditions on $\partial\Lambda_n$);
\item free boundary conditions elsewhere.
\end{itemize} 
The boundary conditions on $\Omega$ inherited by the conditioning $\Gamma_1=\gamma_1$ and $\Gamma_2=\gamma_2$ dominate wired on $\gamma_1$ and $\gamma_2$, and free elsewhere. 
Thus, we deduce that
 \begin{equation}\label{eq:7a}
   \Ps{\gamma_1\lr[\Omega]\gamma_2 \:\big |\: \Gamma_1=\gamma_1,\Gamma_2=\gamma_2}\ge
   \P[D][\mathrm{mix}]{K(y,\infty)\lr[D]A}.
 \end{equation}

 \begin{figure}[htbp]
 \begin{center}      \includegraphics[width=1.00\textwidth]{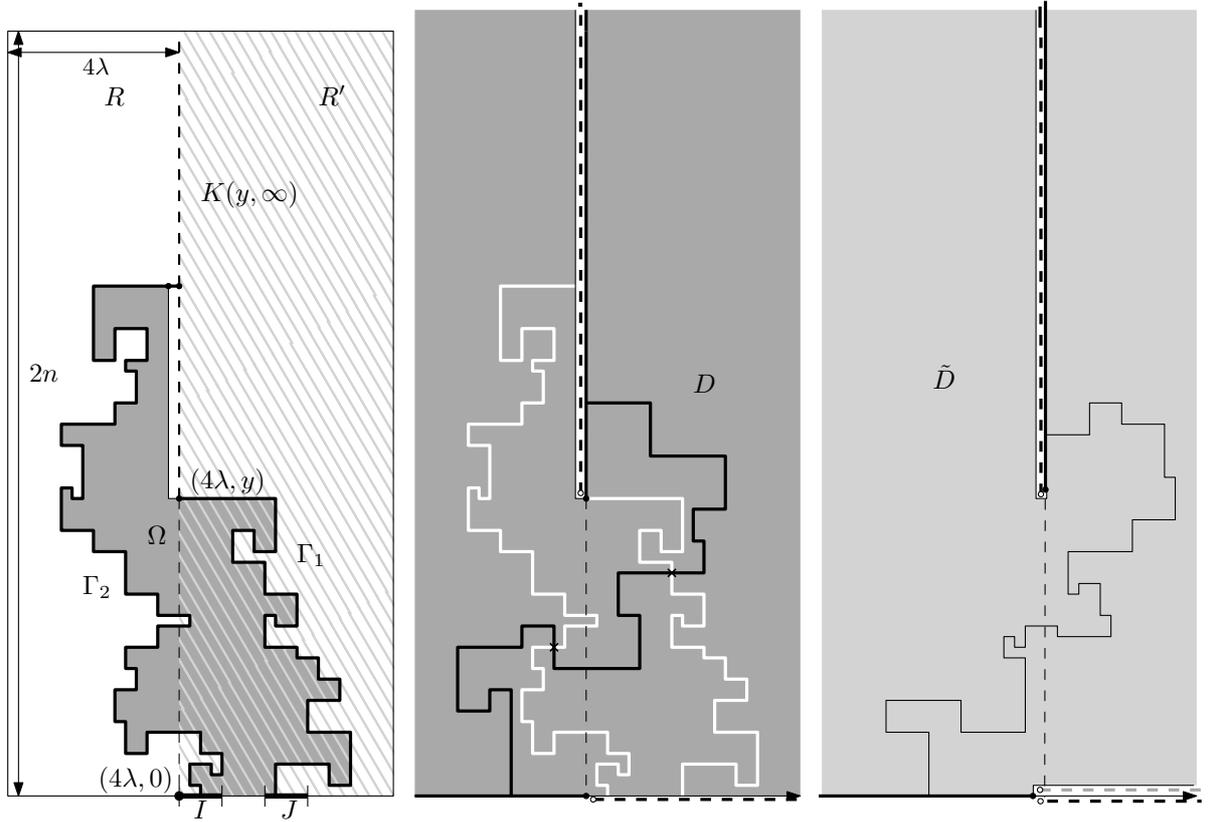}
   \caption{{\bf Left.} The domain $\Omega$. We depicted the part of the domain with free boundary conditions by putting dual wired boundary conditions on the associated dual arcs. The wired boundary conditions are depicted in bold. The rectangles $R$ and $R'$ are also specified ($R'$ is in dashed). {\bf Center.} The domain $D$. We depicted the domain $\Omega$ in white. The existence of an open path between $K(y,\infty)$ and $A$ implies the existence of an open path between $\g_1$ and $\g_2$ in $D$ (between the two crossings).  {\bf Right.} The domain $\tilde D$ with one path from $K(y+1,\infty)$ to $B$. The pre image of this path by the reflection mapping $D$ onto $\tilde D$ is a dual-path in $D$ preventing the existence of an open path from $K(y,\infty)$ to $A$. } \label{fig:domainOmega}\end{center}
   
 \end{figure}
  As mentioned above, the domain $D$ is not exactly a symmetric domain but it is still very close to be one. 
  Consider the domain $\tilde D$ (see on the right of Fig.~\ref{fig:domainOmega}) obtained from $D$ by the reflection with respect to the vertical line $d=\{(4\lambda-\frac14,y):y\in\bbR\}$ and a translation by $(\tfrac12,\tfrac12)$. 
Let $B=(-\infty,4\lambda-1]\times\{0\}$. Define the boundary conditions $mix$ on $\tilde D$ as
\begin{itemize}[noitemsep,nolistsep]
\item  wired boundary conditions on $K(y+1,\infty)\cup B$ (it is very important that the two arcs are wired together);
\item  free boundary conditions at infinity;
\item free boundary conditions elsewhere.
\end{itemize} 
  Using duality, we find that
  $$\P[D][\mathrm{mix}]{K(y,\infty)\stackrel{D}{\not\longleftrightarrow}A}=\P[\tilde D][\mathrm{mix}]{K(y+1,\infty)\stackrel{\tilde D}{\longleftrightarrow}B}$$
and thus
    \begin{equation}\label{eq:6}
   \P[D][\mathrm{mix}]{K(y,\infty)\lr[D]A}+\P[\tilde D][\mathrm{mix}]{K(y+1,\infty)\stackrel{\tilde D}{\longleftrightarrow}B}=1.
 \end{equation}
 Define the $mix'$ boundary conditions on $D$ as wired boundary conditions on $K(y+1,\infty)\cup B:=(-\infty,4\lambda-1]\times\{0\}$ (the two arcs are once again wired together) and free elsewhere (they correspond to the boundary conditions $mix$ on $\tilde D$). Since $\tilde D\subset D$,
$$\P[\tilde D][\mathrm{mix}]{K(y+1,\infty)\stackrel{\tilde D}{\longleftrightarrow}B}\le \P[D][\mathrm{mix'}]{K(y,\infty)\stackrel{D}{\longleftrightarrow}A}.$$
The boundary conditions for the term on the right can be compared to the boundary conditions {\em mix}. First, one may wire the vertices $(4\lambda,y)$ and $(4\lambda,y+1)$ together, and the vertices $(4\lambda-1,0)$ and $(4\lambda,0)$ together, which increases the probability of an open path between $K(y,\infty)$ and $A$. Second, one may unwire the arcs $B$ and $K(y+1,\infty)$, paying a multiplicative cost of $q^2$. Using the previous inequality and the comparison between the boundary conditions described in this paragraph, we deduce 
  $$\P[\tilde D][\mathrm{mix}]{K(y+1,\infty)\stackrel{\tilde D}{\longleftrightarrow}B}
  \le q^2 \P[D][\mathrm{mix}]{K(y,\infty)\stackrel{D}{\longleftrightarrow}A}.$$
  Putting this inequality in \eqref{eq:6} and then using \eqref{eq:7a}, we find that
  $$ \Ps{\gamma_1\lr[\Omega]\gamma_2 \:\big |\: \Gamma_1=\gamma_1,\Gamma_2=\gamma_2}\ge \P[D][\mathrm{mix}]{K(y,\infty)\stackrel{D}{\longleftrightarrow}A}\ge \frac1{1+q^2}.$$
  \end{proof}
 It follows from \eqref{eq:5} and the probabilities of (i) and (ii) that \begin{equation*}
 \Ps{I\stackrel{R}{\longleftrightarrow} J}\ge \frac1{1+q^2}\times\frac{c_1^2}{18}.
 \end{equation*}
 Here again, $20k$ translations of the event above guarantee the occurrence of an open crossing from left to right in $[-kn,kn]\times[0,2n]$. This occurs with probability larger than $(\frac{c_1^2}{18(1+q^2)})^{20k}$ thanks to the FKG inequality again.      
\end{proof}
In the next lemma, we consider horizontal crossings in rectangular shaped domains with free boundary conditions on the bottom and wired elsewhere.

\begin{lemma}\label{lem:push} 
  For all $k>0$ and $\ell\ge4/3$, there exists a constant $c=c(k,\ell)>0$ such that for all $n>0$,
  \begin{equation}
    \label{eq:7} 
    \P[D][1/0]{\mathcal C_h\left([-kn,kn]\times[0,n]\right)}\geq c
  \end{equation}
with $D=[-kn,kn]\times [0,\ell n]$, and $\PP[D][1/0]$ is the random-cluster measure with free boundary conditions on the bottom side, and wired on the three other sides.
\end{lemma}

\begin{proof}
 For $\ell=4/3$, the result follows directly from Lemma~\ref{lem:stripRSW} since boundary conditions  dominate boundary conditions in the strip $\mathbb Z\times[0,\tfrac {4n}3]$, and therefore there exists an horizontal crossing of the rectangle $[-kn,kn]\times[\tfrac n3,n]$ with probability bounded away from 0. 
 
 Now assume that the result holds for $\ell$ and let us prove it for $\ell+1/3$. By comparison between boundary conditions in $[-kn,kn]\times[\tfrac n3,\ell n+\tfrac n3]$, we know that 
$$
    \P[D][1/0]{\mathcal C_h([-kn,kn]\times[\tfrac n3,\tfrac{4n}3])}\geq c(k,\ell).
$$
Conditioning on the highest such crossing, the boundary conditions below this crossing dominate the free  boundary conditions on the bottom side of $[-kn,kn]\times[0,\tfrac{4n}3]$, and wired on the other three sides of $[-kn,kn]\times[0,\tfrac{4n}3]$. An application of the case $\ell=\tfrac43$ enables us to set $c(k,\ell+\tfrac13)=c(k,\ell)c(k,\tfrac43)$.

The proof follows from the fact that the probability in \eqref{eq:7} is decreasing in $\ell$.
\end{proof}

\begin{lemma}\label{lem:induction1}
  There exists a
  constant $C<\infty$ such that, for all $n\geq 1$, 
  \begin{equation}
    \label{eq:8}
    \P[\Lambda_{56n}][0]{\mathcal A _{7n}}\leq C\, \P[\Lambda_{8n}][0]{\mathcal A _n}^2.
  \end{equation}
\end{lemma}
\begin{proof}
 Define $z_\pm=(\pm 5n,0)$. If $\mathcal A_{7n}$ occurs, the boundary conditions on $\Lambda_{7n}$ dominate the wired boundary conditions on $\Lambda_{56n}$  due to the existence of the open circuit in $\Lambda_{14n}\setminus\Lambda_{7n}$. Theorem~\ref{thm:weakRSW} thus implies the existence of a constant $c_1>0$ such that, for all $n$,
  \begin{align}
    \label{eq:9}
    \phi^0_{\Lambda_{56n}}[\mathcal A_n(z_+) \cap \mathcal A_n(z_-) | \mathcal
      A_{7n}] &\geq  \phi^1_{\Lambda_{56n}}[\mathcal A_n(z_+) \cap \mathcal
      A_n(z_-)]\geq c_1.
  \end{align}
  It directly implies that for all $n$,
  \begin{equation}
    \label{eq:10}
    \phi^0_{\Lambda_{56n}}[\mathcal A_n(z_+) \cap  \mathcal A_n(z_-)] \geq c_1 \phi^0_{\Lambda_{56n}}[\mathcal A _{7n}].
  \end{equation}
  Now, examine the domain $D=[-56n,56n]\times[2n,56n]$ and consider the measure $\PP[D][1/0]$ with free boundary conditions on the bottom and wired boundary elsewhere. Also set $$R^*_+:=[-56n-\tfrac12,56n+\tfrac12]\times[2n+\tfrac12,3n-\tfrac12].$$ Under $\phi^0_{56n}[\,\cdot\,|\mathcal A_n(z_+)
  \cap \mathcal A_n(z_-)]$, the boundary conditions on $D$ are 
  dominated by wired boundary conditions on the bottom and free
  boundary conditions on the other sides. As a consequence,
  Lemma~\ref{lem:push} applied to $k=56$ and $\ell=54$ implies that
  \begin{align}
    \label{eq:11}
    \phi^0_{\Lambda_{56n}}\big[\mathcal
      C_h^*\left(R^*_+\right)& \big| \mathcal
      A_n(z_+) \cap  \mathcal A_n(z_-)\big]\geq
    \P[D][1/0]{\calC_h^*\left(R^*_+\right)}\geq c_2
  \end{align}
 for some universal constant $c_2>0$ independent of $n$.
 Similarly, with 
 $D'=[-56n,56n]\times[-56n,-2n]$
 and
 $$R_-^*:=[-56n-\tfrac12,56n+\tfrac12]\times[-3n+\tfrac12,-2n-\tfrac12],$$
 we find
  \begin{equation}
    \label{eq:12}
    \P[\Lambda_{56n}][0]{\mathcal
      C_h^*\left(R^+_-\right) \Big| \mathcal
      A_n(z_+) \cap  \mathcal A_n(z_-)}\geq c_2.
  \end{equation}
  Define the event $\mathcal B_n$, illustrated on Fig.~\ref{fig:eventsBH}, which is the intersection of the events
$\mathcal A_n(z_+)$, $\mathcal A_n(z_-)$, $\calC_h^*\left(R_+^*\right)$ and $\calC_h^*\left(R_-^*\right)$.  
 Equations \eqref{eq:10}, \eqref{eq:11} and \eqref{eq:12} lead to the
estimate
  \begin{equation}
    \label{eq:13}
    \P[\Lambda_{56n}][0]{\mathcal B_n}\geq c_3\P[\Lambda_{56n}][0]{\mathcal A _{7n}},
  \end{equation}
  where $c_3>0$ is a positive constant independent of $n$. 
  \begin{figure}[htbp]
    \centering
    \includegraphics[width=1.00\textwidth]{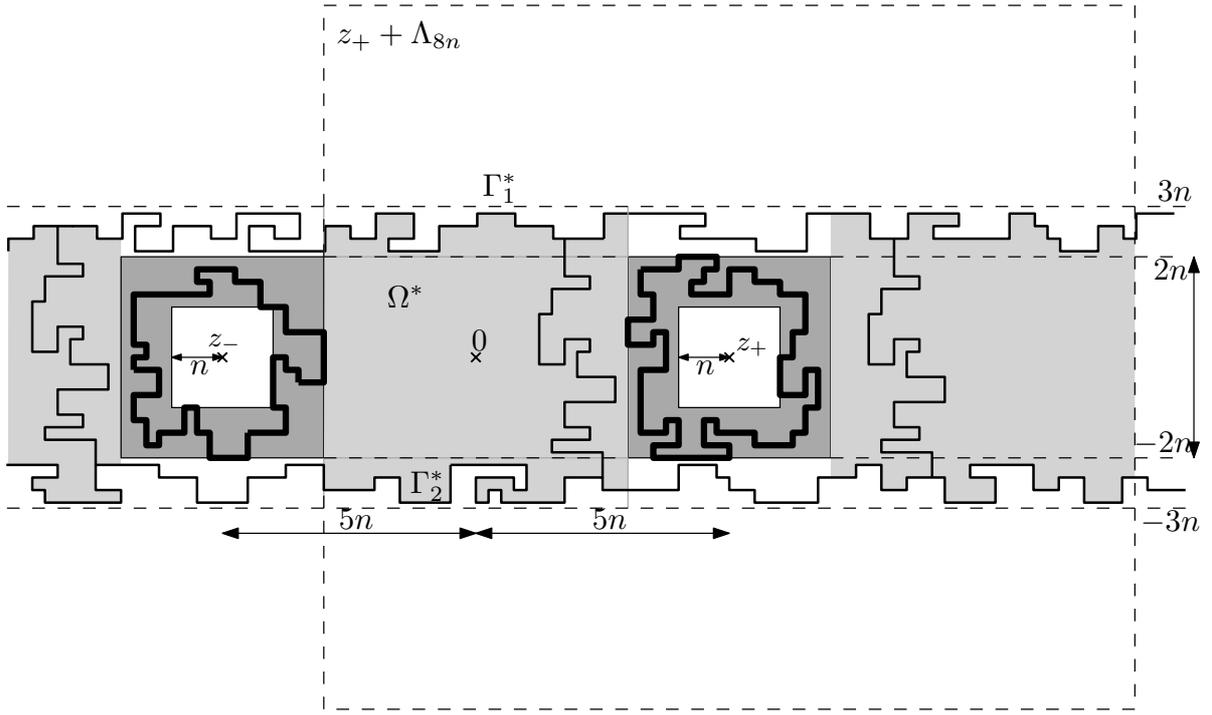}
    \caption{Primal open crossings are in bold, dual-open are in plain. The events $\mathcal A_n(z_+)$, $\mathcal A_n(z_-)$ and the existence of the dual horizontal crossings of $R_+^*$ and $R_-^*$ form $\mathcal B_n$. Conditionally on $\calB_n$, $\Gamma_1^*$ and $\Gamma_2^*$ are connected in $\Omega^*$ by a dual-open path with probability larger than $1/(1+q^2)$.}
    \label{fig:eventsBH}
  \end{figure}

  Assume $\mathcal B_n$ occurs and define $\Gamma_1^*$ to be the top-most horizontal dual-crossing of
  $R_+^*$ and $\Gamma_2^*$ to be the bottom-most horizontal dual-crossing of
  $R_-^*$. Note that these paths are dual paths. 
  Let $\Omega^*$ be the set of dual-vertices in $R^*:=[-3n+\tfrac12,3n-\tfrac12]^2$ below $\Gamma_1^*$
  and above $\Gamma_2^*$. Exactly as in the proof of
  Lemma~\ref{lem:stripRSW}, when conditioning on $\Gamma_1^*$, $\Gamma_2^*$
  and everything outside $\Omega^*$, the boundary conditions inside
  $\Omega^*$ are dual-wired on $\Gamma_1^*$ and $\Gamma_2^*$, and dual-free elsewhere. The dual measure inside $\Omega^*$ therefore dominates the restriction to $\Omega^*$ of the dual measure on $R^*$ with dual-wired boundary conditions on the top and bottom, and dual-free boundary conditions on the left and right sides. 
  Using \eqref{crossing square mixed}, we find
    \begin{equation*}
    \P[\Lambda_{56n}][0]{\calC_n\:\big|\:\mathcal
      B_n}\geq \frac1{1+q^2},
  \end{equation*}
  where $\calC_n=\{\Gamma_1\lr[*]\Gamma_2\text{ in }{R^*}\}$. Similar inequalities hold for the events
  \begin{align*}\calD_n&=\big\{\Gamma_1\lr[*]\Gamma_2\text{ in }{(-10n,0)+R^*}\big\},\\
  \calE_n&=\big\{\Gamma_1\lr[*]\Gamma_2\text{ in }{(10n,0)+R^*}\big\}.
  \end{align*}The FKG inequality thus implies
       \begin{equation}
    \label{eq:14}
    \P[\Lambda_{56n}][0]{\calC_n\cap\calD_n\cap\calE_n\:\big|\:\mathcal
      B_n}\geq \frac1{(1+q^2)^3}
  \end{equation}
   which, together with \eqref{eq:13}, leads to 
  \begin{equation}
    \label{eq:15}
    \P[\Lambda_{56n}][0] {\mathcal B_n \cap  \calC_n\cap\calD_n\cap\calE_n}\geq \frac{c_3}{(1+q^2)^3}  \P[\Lambda_{56n}][0]{\mathcal A _{7n}}.
  \end{equation}
  The event estimated in \eqref{eq:15} implies in particular
  the existence of dual circuits in $z_++\Lambda_{8n}^*$ and
  $z_-+\Lambda_{8n}^* $ disconnecting $z_++\Lambda_{2n}^*$ from
  $z_-+\Lambda_{2n}^*$. Writing $\mathcal F_n$ for the event that
  such dual circuits exist and using the comparison between boundary
  conditions one last time (more precisely a ``conditioning on the exterior-most circuit''-type argument), we obtain
  \begin{align*}
    \label{eq:16}
   \P[\Lambda_{8n}][0] {\mathcal A_n}^2&= \P[z_-+\Lambda_{8n}][0] {\mathcal A_n(z_-)}
    \P[z_++\Lambda_{8n}][0] {\mathcal A_n(z_+)}\\
    &\ge
    \P[\Lambda_{56n}][0] {\mathcal A_n(z_-) \: | \:  \mathcal A_n(z_+) \cap
      \mathcal F_n} \P[\Lambda_{56n}][0] {\mathcal A_n(z_+) \: | \:
      \mathcal F_n} \P[\Lambda_{56n}][0]{\mathcal F_n}\\
   &=  \P[\Lambda_{56n}][0] {\mathcal A_n(z_-)\cap\mathcal A_n(z_+) \cap
      \mathcal F_n }\\
      &\geq \frac {c_3}{(1+q^2)^3} \P[\Lambda_{56n}][0]{\mathcal A _{7n}}.  
  \end{align*}
This inequality implies the claim.
\end{proof}

We need a last lemma before being able to prove Proposition~\ref{prop:dualityP_n}.
 \begin{lemma}\label{exterior circuit}
Let  $1\le k\le n$, 
$$\P[\Lambda_n][0]{0\longleftrightarrow\partial\Lambda_k}\le \sum_{m\ge k}72m^4\max_{\substack{a\in\{0\}\times[0,m]\\ b\in\{m\}\times[0,m]}}\P[{[0,m]^2}][0]{a\longleftrightarrow b}.$$
\end{lemma}
\begin{proof}
For $x=(x_1,x_2)$, define $\|x\|_\infty=\max\{|x_1|,|x_2|\}$.

 Define $\calC$ to be the connected component of the origin. Consider the event that $a$ and $b$ are two vertices in $\calC$ maximizing the $\|\cdot\|_\infty$-distance between each other. Since these vertices are at maximal distance from each other, they can be placed on the two opposite sides of a square box $\Lambda$ in such a way that $\calC$ is included in this box. Let $\calA_{\rm max}(a,b,\Lambda)$ be the event that $a$ and $b$ are connected in $\Lambda$ and that their cluster is contained in $\Lambda$.
 
We now wish to estimate the probability of $\mathcal A_{\rm max}(a,b,\Lambda)$. Let $\Lambda^*$ be the subgraph of $(\bbZ^2)^*$ composed of dual-edges whose end-points correspond to faces touching $\Lambda$. Let $\mathbf{C}$ be the set of dual self-avoiding
circuits $\gamma=\{\gamma_0\sim \gamma_1\sim\dots\sim \gamma_m\sim \gamma_0\}$ included in $\Lambda^*$ surrounding $a$ and $b$. As before, we denote by $\overline{\gamma}$ the interior of $\gamma$. 

On the event $\calC$, there exists $\gamma\in\mathbf{C}$ which is dual-open\footnote{Note that this is true even when $\Lambda=\Lambda_n$ since free boundary conditions can be seen as dual-wired boundary conditions on $\Lambda_n^*$, and that therefore $\partial\Lambda_n^*$ provides us with a dual self-avoiding circuit in $\mathbf{C}$ which is dual-open. A similar reasoning applies when $\Lambda$ only shares some sides with $\Lambda_n$.}, and $a$ and $b$ are connected in $\overline\gamma$. 
As before, we may condition on the outermost dual-open circuit $\Gamma$ in $\mathbf C$. We deduce as in the last proof that
\begin{align*}\phi^\xi_{\Lambda_n}(a\longleftrightarrow b\text{ in }\overline\gamma\,|\,\Gamma=\gamma)&=\phi^0_{\overline\gamma}(a\longleftrightarrow b\text{ in }\overline\gamma)\le \phi^0_{\Lambda}(a\longleftrightarrow b\text{ in }\overline\gamma)\le \phi^0_{\Lambda}(a\longleftrightarrow b).\end{align*}
We now partition $\calA_{\rm max}(a,b,\Lambda)$ into the events $\{\Gamma=\gamma\}$ to find
\begin{equation*}\phi^\xi_{\Lambda_n}(\calA_{\rm max}(a,b,\Lambda))\le \phi^0_{\Lambda}(a\longleftrightarrow b)\end{equation*} and therefore 
\begin{equation}\label{eq:17}\phi^\xi_{\Lambda_n}(\calA_{\rm max}(a,b,\Lambda))\le \max_{\substack{a\in\{0\}\times[0,m]\\ b\in\{m\}\times[0,m]}}\P[{[0,m]^2}][0]{a\longleftrightarrow b},\end{equation}
where $m=\|a-b\|_\infty$.

We may now use the fact that if $0$ is connected to distance $k$, there exist $a$ and $b$ at distance $m\ge k$ of each others and a box $\Lambda$ having $a$ and $b$ on opposite sides such that $\calA_{\rm max}(a,b,\Lambda)$ occurs. Let us bound the number of choices for $a$, $b$ and $\Lambda$. 

For a fixed $m\ge k$, there are $|\Lambda_m|=(2m+1)^2$ choices for $a$ (since $a$ must be at distance smaller or equal to $m$ from the origin). Then $a$ must be on the boundary of $\Lambda$ and there are therefore $|\partial\Lambda|=4m$ choices for $\Lambda$. The number of choices  for $b$ is bounded by $m+1$ (it must be on the opposite sides of $\Lambda$). Therefore, for fixed $m$ we can bound the number of possible triples $(a,b,\Lambda)$ by $4m(2m+1)^2(m+1)\le 72m^4$. We have been very wasteful in the previous reasoning and the bound on this number could be improved greatly but this will be irrelevant for applications. 

Overall, \eqref{eq:17} and a union bound gives
 $$\P[\bbZ^2][0]{0\leftrightarrow\partial\Lambda_k}\le\sum_{m\ge k}72m^4\max_{\substack{a\in\{0\}\times[0,m]\\ b\in\{m\}\times[0,m]}}\P[{[0,m]^2}][0]{a\lr b}.$$
\end{proof}

\begin{proof}[Proof of Proposition~\ref{prop:dualityP_n}] Obviously the cases 1 and 2 cannot occur simultaneously. Suppose that the first case does not occur and let us prove that the second does. 

For all
  $n\geq 1$, set $u_n=C\P[\Lambda_{8n}][0]{\mathcal A _n}$, where $C$ is defined as in Lemma~\ref{lem:induction1}.
  With this notation, Lemma~\ref{lem:induction1} implies that $u_{7n}\le u_n^2$
  for any $n\ge 1$ which implies that for $0\le \ell,k\le n$,
  \begin{equation}
  u_{7^kn_0}\le u_{n_0}^{2^k}\label{eq:18}
\end{equation}
for any positive $k\ge0$ and $n_0\ge 1$. Now, if $\displaystyle
 \liminf_{n\to\infty}\P[\Lambda_{8n}][0]{\mathcal A _n}=0$, then we may pick $n_0$ such that $u_{n_0}<1$. By \eqref{eq:18}, there exists $c_1>0$ such that for all $n$ of the form $n=7^kn_0$,
  \begin{equation}
   \label{eq:19}
   u_n\leq \exp\left (-c_1{n^{\log 2/\log 7}} \right ).
 \end{equation}
Fix $n=7^kn_0$ and consider $\tfrac n{7}\le m< n$. The FKG inequality and the comparison between boundary conditions imply that
\begin{align*}\P[{[0,m]^2}][0]{(0,k)\longleftrightarrow(m,\ell)}&\le \Big(\P[{[-m,m]\times[0,m]}][0]{(-m,\ell)\longleftrightarrow(m,\ell)}\Big)^{1/2}\\
&\le\Big(\P[{\Lambda_{8n}}][0]{\calC_h([-2n,2n]\times[0,m])}\Big)^{1/14}\\
&\le\Big(\P[{\Lambda_{8n}}][0]{\calA_n}\Big)^{1/56}\le \exp\left (-c_2{n^{\log 2/\log 7}} \right ).\end{align*}
In the first inequality, we used that if $(0,k)\longleftrightarrow(m,\ell)$ and $(-m,\ell)\longleftrightarrow(0,k)$, then $(-m,\ell)\longleftrightarrow(m,\ell)$. In the second inequality, we have used that if $(x,\ell)\longleftrightarrow(x+2m,\ell)$ occur for $x=2mj$ with $j\in\{-7,\dots,7\}$, then $\calC_h([-2n,2n]\times[0,m])$ occurs. Finally in the third inequality we combined four crossings as in the proof of {\bf P5}$\Rightarrow${\bf P5a}. Lemma~\ref{exterior circuit} implies the claim.
 \end{proof}
 \newcommand{\Pl}[2][m,2n]{\P[#1][\theta]{#2}}

Theorem~\ref{thm:main} follows directly from Proposition~\ref{prop:dualityP_n} and the following proposition:

\begin{proposition}\label{prop:dualityQ_n}
  If there exists $\alpha>0$ such that for all $n\ge 1$, $$\P[\bbZ^2][0]{0\longleftrightarrow \partial\Lambda_n}\le \exp(-n^{\alpha}),$$ then there exists $c>0$ such that for all $n
    \geq 1$, $$\P[\bbZ^2][0]{0\longleftrightarrow \partial\Lambda_n}\leq \exp(-cn).$$ 
\end{proposition}

%
We start by a lemma. Let $n\geq 1$ and  $\theta \in [-1,1]$.  Define the {\em tilted strip} in direction $\theta$:
$$S(n,\theta):=\{(x,y)\in\bbZ^2: 0 \leq y-\theta x \leq n\}.$$    Write 
  $\PP[S(n,\theta)][1/0]$ for the random-cluster measure
  on the tilted strip $S(n,\theta)$ with wired boundary conditions on the top side and free on the bottom side (the boundary conditions at infinity are irrelevant since the tilted strip is essentially a one-dimensional graph). 
  
  We will also consider a  truncated version of the tilted strip $S(n,\theta)$. For $m\ge0$, consider the {\em truncated tilted strip} $$S(n,m,\theta):=S(n,\theta)\cap \Lambda_m.
$$
We will always assume that $\theta m\in\N$ (the general case can be treated similarly). Write $\PP[S(n,m,\theta)][1/0]$ for the random-cluster measure
  with free boundary conditions on the bottom side and wired on the other three sides. 
  
For simplicity, we will call the bottom side of the strip or the truncated strip the {\em free arc}, and the rest of the boundary the {\em wired arc}.
   \begin{lemma}\label{lem:touch}
 For all $m\ge n\ge1$ and $\theta\in
  [-1,1]$,
  \begin{equation}
    \label{eq:21}
    \P[S(n,m,\theta)][1/0]{0\longleftrightarrow \mathrm{wired\ arc}}\geq \frac1{5m^2n^2}.
  \end{equation}
\end{lemma}
\begin{proof}
  Fix $n\ge0$ and $\theta\in[-1,1]$. Let us work in the strip $S(2n,\theta)$. From now on, we drop the dependence in $n$ and $\theta$ and write for instance $S=S(n,\theta)$ and $S(m)=S(2n,m,\theta)$. Beware that there is a slightly confusing notation here: the height of the strip is $2n$ while the one of the truncated strip is $n$. 
  
  For $x\in S$, define the translate $S_{x}(m):=x+S(m)$ of $S(m)$. We extend the definition of wired and free arcs to this context. Let $\mathcal A(x)$ be the event that $x$ is connected to
  the wired arc of $S_{x}(m)$ and every open path from a vertex $y\notin S_x(m)$ to $x$ intersects the wired arc (of $S_x(m)$). In other words, no open path starting from $x$ ``exits'' $S_x(m)$ through the free arc (i.e. the bottom side).
  
    We consider the random function $F:\N \longrightarrow [0,2n]$ defined by 
  \begin{equation}
    \label{eq:22}
    F(k):=\min\{\ell : (k,\ell) \text{ is connected to the top side of $S$} \} -\theta k.
  \end{equation}
  Recall that $\theta m\in \N$. Therefore, $F$ can take only the $2nm+1$ following values: 
  $$\big\{0,\tfrac1m,\dots,\tfrac{2nm-1}{m},2n\big\}.$$ 
  On the event $\{F(0)\leq n \}$, there must exist
  $k\in\{-nm^2,\dots,nm^2\}$ such that $F(k)\le n$ and $F(k')\ge F(k)$
  for every $|k'-k|\le m$. Otherwise, if there is no such $k$, then there exists a sequence $0=k_0,\dots,k_{nm}$ with $|k_{i+1}-k_i|\le m$ and $0<F(k_{i+1})<F(k_i)$. But this provides $nm+1$ distinct values for $F$, all smaller or equal to $n$ and strictly larger than 0, which is contradictory.

 Now, for $k$ satisfying $F(k)\le n$ and $F(k')\ge F(k)$ for every $|k'-k|\le m$, the event $\calA((k,F(k)))$ is realized. In conclusion, if $F(0)\le n$, then there exists $x\in S(n,nm^2,\theta)$ such that $\calA(x)$ is realized and the union bound shows the existence of $x\in S(n,\theta)$ (the lower half of $S$) such that
 \begin{align*}
  \P[S][1/0]{\mathcal A(x)}\ge \frac{\P[S][1/0]{F(0)\le n}}{|S(n,nm^2,\theta)|}=\frac{\P[S][1/0]{F(0)\le n}}{n(2nm^2+1)}.
\end{align*} 
Consider the interface between the open cluster connected to the top side of the box and the dual-open cluster dual-connected to the bottom side. By duality, this interface intersects $\{0\}\times[0,n]$ with probability larger or equal to 1/2. Thus, $\P[S][1/0]{F(0)\leq n}\geq
\frac12 $ and therefore
 $$\P[S][1/0]{\mathcal A(x)}\ge \frac1{5n^2m^2}.$$
 In order to conclude, we simply need to prove that
\begin{equation}\label{eq:23}\P[S(m)][1/0]{0\longleftrightarrow \text{wired arc}}\ge \P[S][1/0]{\mathcal
    A(x)}.\end{equation}
\begin{figure}
  \centering
  \includegraphics[width=0.60\textwidth]{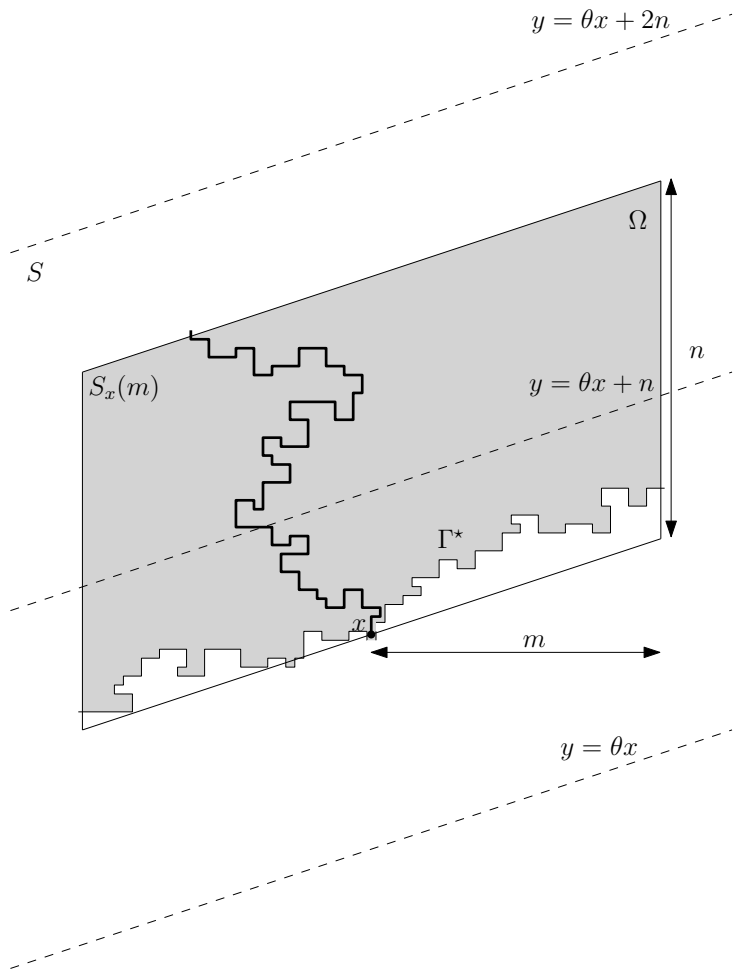}
  \caption{The event $\mathcal A(x)$. The bottom-most  dual-path $\Gamma^*$. }
  \label{fig:lowestPoint}
\end{figure}
First, observe that since $x$ is contained in the bottom half $S(n,\theta)$ of $S$, the set $S_x(m)$ is entirely included in $S$. Second, since there is no open path containing $x$ and exiting $S_x(m)$ by the free arc, there
exists a lowest dual-open path in $S_x(m)^*$, denoted $\Gamma^*$, preventing the existence of such a path, see Fig.~\ref{fig:lowestPoint}. Let $\Omega$ be the set of
vertices of $S_x(m)$ above $\Gamma^*$.  The boundary conditions on $\Omega$ are dominated by free boundary conditions on the bottom side of $S_x(m)$ and wired on the three other sides of $S_x(m)$. If $A(x)$ occurs, then conditionally on $\Gamma^*$,
$x$ is connected to the wired arc of $S_x(m)$ by an open path contained in $\Omega$.
Thus,
\begin{align*}\P[S][1/0]{x\longleftrightarrow \text{wired arc of }S_x(m)|\Gamma^*}&\le \P[S_x(m)][1/0]{x\longleftrightarrow \text{wired arc of }S_x(m)}\\
&=\P[S(m)][1/0]{0\longleftrightarrow \text{wired arc of }S(m)}.
\end{align*}
We omitted a few lines to get the first inequality since we already mentioned such an argument. The equality follows from invariance under translations. Since the previous bound is uniform in the possible realizations of $\Gamma^*$, 
we deduce
\begin{align*}\P[S][1/0]{\calA(x)}\le \P[S(m)][1/0]{0\longleftrightarrow \text{wired arc of }S(m)}.\end{align*}
The result follows readily.
\end{proof}

The next lemma  will be used recursively in the proof of Proposition~\ref{prop:dualityQ_n}.
\begin{lemma}\label{lem:induction2}
Assume that there exists $\alpha>0$ such that for all $n\ge 1$, $$\P[\bbZ^2][0]{0\longleftrightarrow \partial\Lambda_n}\le \exp(-n^{\alpha}).$$ Then for $\ep>0$ small enough, there exists a constant $C<\infty$  such that for any $n\geq 1$, any $u\in\{-n\}\times[-n,n]$ and any $v\in\{n\}\times[-n,n]$,
  \begin{align}
    \label{eq:24}
    \P[\Lambda_n][0]{u\lr v}\leq  e^{Cn^\ep} &\P[\bbZ^2][0]{0\lr\partial\Lambda_{n}}^2+
    Cn^6\sum_{\substack{k,\ell\ge n^\ep\\k+\ell= 2n}}\P[\bbZ^2][0]{0\lr\partial\Lambda_k}\P[\bbZ^2][0]{0\lr\partial\Lambda_\ell}.\nonumber
  \end{align}
\end{lemma}

\begin{proof}
  \newcommand{\Pn}[1]{\P[\tilde\Lambda_n][0]{#1}}
Fix $\ep>0$. Let us translate the box $\Lambda_n$ in such a way that $u=-v$; the new box is denoted by $\tilde\Lambda_n$. Define the set $$D=\left \{z\in\tilde\Lambda_n:
    \mathrm{d}(z,[u,v])<n^\ep\right\}.$$
  \begin{figure}[htbp]
    \centering
    \includegraphics[width=.495\textwidth]{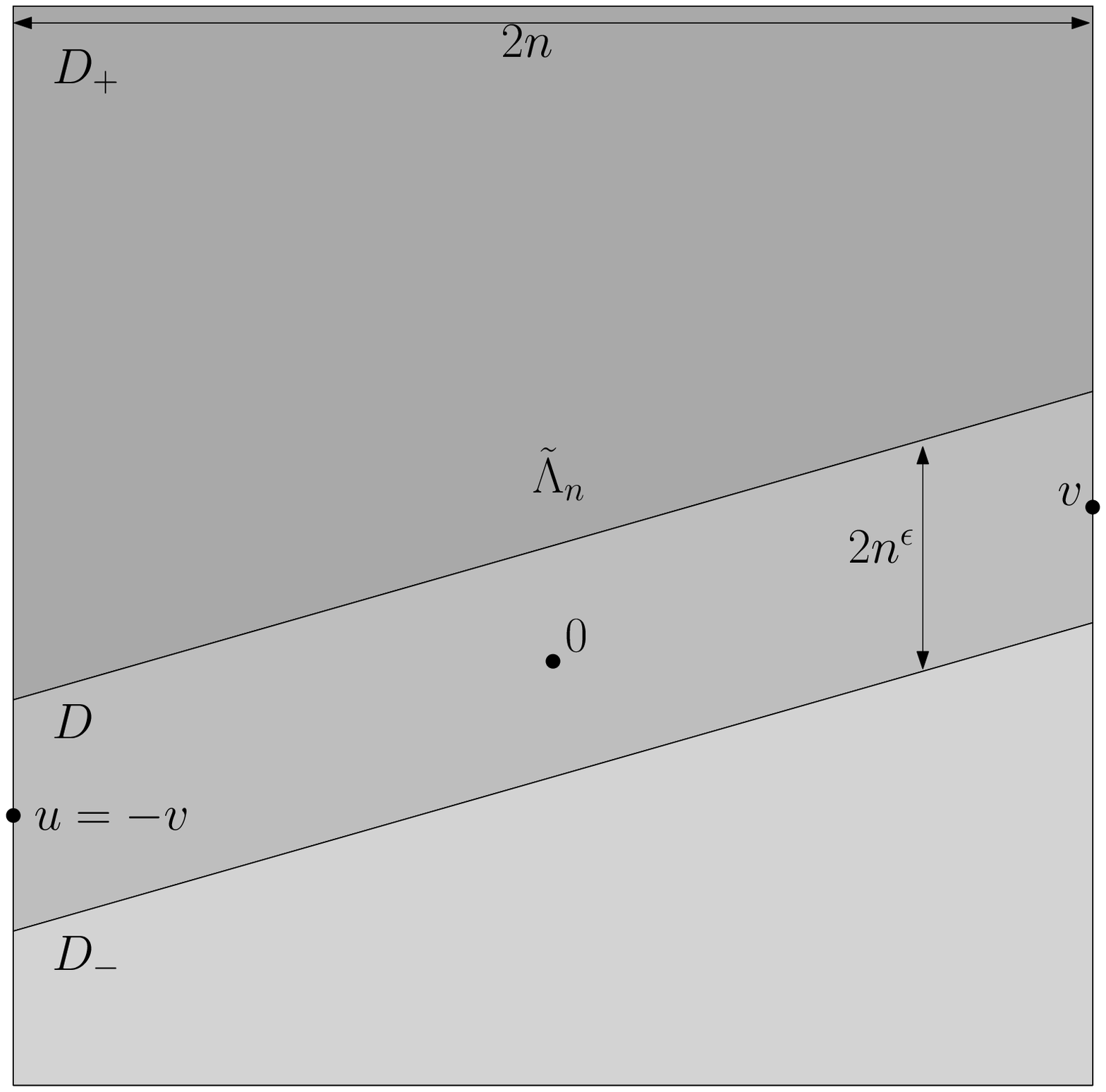}  \includegraphics[width=.495\textwidth]{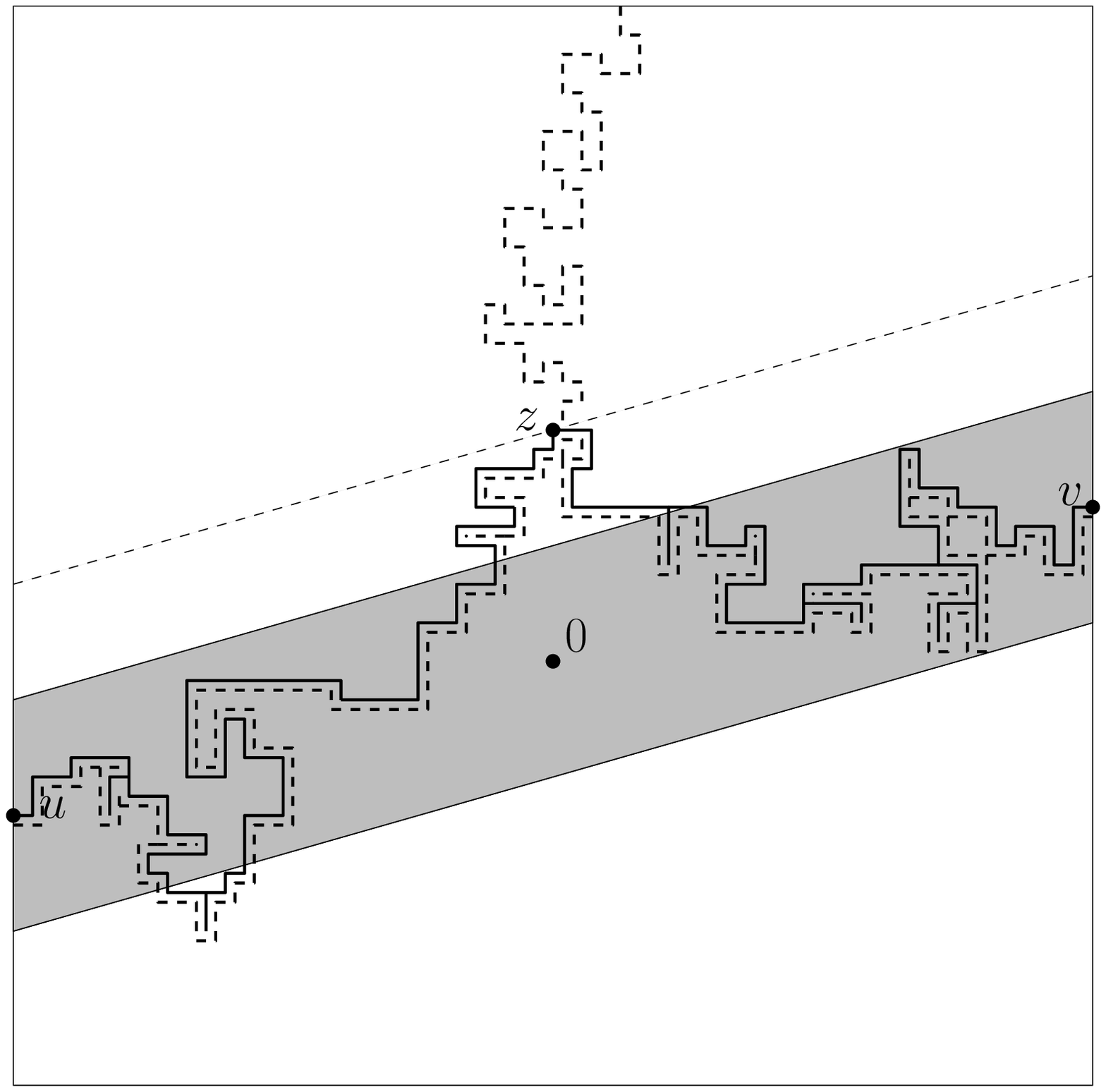}
    \caption{{\bf Left.} The regions $D$, $D_-$ and $D_+$. Note that $0$ is not necessarily at the center of $\tilde\Lambda_n$. {\bf Right.} The situation before closing the edges surrounding $z$ when $\calG_n(z)$ and $\{z+(\frac12,\frac12)\stackrel{*}{\longleftrightarrow}\partial \Lambda_n\}$ occurs. The dual-open paths are depicted in dash lines.}
    \label{fig:partitionLambda}
  \end{figure}
  As illustrated in Fig.~\ref{fig:partitionLambda}, we consider the sets
  $D_-$ and $D_+$ of points $z\in\tilde\Lambda_n$ lying respectively
  below $D$ and above $D$. On $\{u \lr[\tilde\Lambda_n] v\}$, define
  $\Gamma^-$ and $\Gamma^+$ to be respectively the lowest and highest
  open (non-necessarily self-avoiding) paths connecting $u$ to $v$. The event $\{u
  \lr[\tilde\Lambda_n] v\}$ is included in the union of the following three sub-events:
  \begin{align}
     &\mathcal E_-=\{u \lr[\tilde\Lambda_n]v\}\cap\{  \Gamma^-\cap D_+ \neq \emptyset \},\\
     &\mathcal E_+=\{u \lr[\tilde\Lambda_n]v\}\cap\{\Gamma^+\cap D_- \neq \emptyset \},\\
     &\mathcal E=\{u \lr[\tilde\Lambda_n]v\}\cap\{\Gamma^+\subset D_+\cup D \}\cap\{\Gamma^-\subset
      D_-\cup D \}.
  \end{align}
  In the rest of the proof, we will bound separately $\Pn{\mathcal E_-}$ (and therefore $\Pn{\mathcal E_+}$ by symmetry) and  $\Pn{\mathcal E}$, hence the two terms on the right-hand side of the inequality in the statement.
  
 \medbreak\noindent\emph{Estimation of $\Pn{\cal E_-}$.}
     For $z\in D_+\cap \Z^2$, let $\calG_n(z)$ be the event that:
     \begin{itemize}[noitemsep,nolistsep]
     \item $u$ is connected to $v$ in $\tilde \Lambda_n$,
     \item $z \in \Gamma^-$ and $\displaystyle\mathrm d
       (z,[u,v])=\max_{z'\in\Gamma^- \cap D_+} \mathrm d(z',[u,v])$.
     \end{itemize}
     Note that
     \begin{equation}
       \label{eq:27}
       \calE_-=\bigcup_{z\in D_+} \calG_n(z). 
     \end{equation}
      Conditionally on $\Gamma^-$, what is above $\Gamma^-$ follows a random-cluster measure with wired boundary conditions on $\Gamma^-$ and free on $\partial \tilde\Lambda_n$. Thus, by comparison between boundary conditions and Lemma~\ref{lem:touch} (with $m=n$ and $\theta=\frac{v_2-u_2}{v_1-u_1}$, where $u=(u_1,u_2)$ and $v=(v_1,v_2)$), we find that 
     \begin{equation}
       \label{eq:28}
       \Pn{z+(1/2,1/2) \lr[*] \partial B_n^* \: \Big| \: \calG_n(z) }
       \geq  \frac1{5(2n)^4}, 
     \end{equation}
     When both $\mathcal G_n(z)$ and $\{z+(1/2,1/2) \lr[*] \partial
     B_n^*\}$ occur, closing the four dual edges surrounding the vertex
     $z$ disconnects $\Gamma^-$ into two paths separated by dual-open circuits
     (see Fig.~\ref{fig:partitionLambda}).
     The respective end-to-end distances $\ell$ and $k$ of these paths satisfy
     $k+\ell\ge 2n-2$.

     Using the comparison between boundary
     conditions once-again, we find
     \begin{align}
       \label{eq:29}
       \Pn{\calG_n(z)} &\leq 5(2n)^4\, \Pn{\calG_n(z) \cap \{z+(1/2,1/2) \lr[*] \partial B_n^*\}} \\
       &\leq \frac{80}{\Cit^4}n^4 \, \sum_{\substack {k,\ell\ge
           n^{\ep}\\ k+\ell=2
           n-2}}\P[\bbZ^2][0]{u\lr u+\partial\Lambda_k}\P[\bbZ^2][0]{v\lr v+\partial\Lambda_\ell}.
     \end{align}
     The finite energy property is used in the second line to close the edges around $z$.  Summing over all possible $z\in D_+$ gives     \begin{equation*}
       \Pn{\cal E_-}\le \Cl{summing} n^6\sum_{\substack {k,\ell\ge n^{\ep}\\ k+\ell=2 n-2}}\P[\bbZ^2][0]{0\lr\partial\Lambda_k}\P[\bbZ^2][0]{0\lr\partial\Lambda_\ell}.
     \end{equation*}
     The finite energy property once again implies that $\P[\bbZ^2][0]{0\lr\partial\Lambda_{r+1}}\ge \Cit\P[\bbZ^2][0]{0\lr\partial\Lambda_{r}}$ for any $r\ge0$ and thus
       \begin{equation}
       \label{eq:30}
       \Pn{\cal E_-}\le \Cl{summing2} n^{6}\sum_{\substack {k,\ell\ge n^{\ep}\\ k+\ell=2 n}}\P[\bbZ^2][0]{0\lr\partial\Lambda_k}\P[\bbZ^2][0]{0\lr\partial\Lambda_\ell}.
     \end{equation}

 \medbreak\noindent\emph{Estimation of $\Pn{\mathcal E}$.}
    First, we wish to justify that conditionally on the occurrence of
    $\mathcal E$, there exists an open path between $u$ and $v$ which is staying in $D$ with probability close to 1. To see this, remark that any open path between $u$ and $v$ must
    lie in the region $\Omega$ between $\Gamma^-$ and
    $\Gamma^+$ (see Fig.~\ref{fig:construction2}). Furthermore, conditioning on $\Gamma^+$ and $\Gamma^-$, the boundary conditions on $\Omega$ are wired. In particular, the configuration in $\Omega$ dominates the restriction to $\Omega$ of a configuration $\tilde\omega$ sampled according to a random-cluster measure with wired boundary conditions at infinity. Since $\Gamma^+$ and $\Gamma^-$ are already open, $u$ and $v$ are connected in $D$ if there exists an open path in $\tilde\omega$ from left to right in $D$. The complement of this event is included in the event that a dual-vertex of $D^*$ is dual-connected to distance $n^\ep$ of itself in $\tilde\omega$. The probability of this event can thus be bounded by $4n^{1+\ep}\exp(-n^{\alpha\ep})$ thanks to the assumption made on connection probabilities. We deduce
    \begin{equation}
      \label{eq:31}
      \Pn{u\lr[D] v}\geq \big(1-4n^{1+\ep}\exp(-n^{\alpha\ep})\big)\Pn{\cal E}.
    \end{equation}
    Now, consider the set of edges $E$ of $D$ intersecting the line $\{\tfrac12\}\times\bbR$. Also define $w_-$ and $w_+$ to be respectively the highest point of $D_-^*$ and the lowest point of $D_+^*$ with first coordinate equal to $\tfrac12$. Let $\mathcal F$ be the event that
    \begin{itemize}[noitemsep,nolistsep]
    \item all the edges of $E$ are closed, 
    \item $w_-$ and $w_+$ are dual-connected to  $\partial\tilde\Lambda_n^*$ in $D_-^*$ and $D_+^*$ respectively.\end{itemize}
     Consider the event $u\lr[D] v$ and modify the configuration by closing all edges in $E$. The finite energy property implies that
     \begin{align}
       \label{eq:32}
       &\Pn{\calF\cap\{u\lr u+\partial\Lambda_{n}\}\cap\{v\lr v+\partial\Lambda_{n-1}\}}\\
       &\quad\quad\quad\quad\quad \ge \Pn{u\lr[D] v} \times \Cit^{2\sqrt 2n^\ep}\times \left(\frac1{5(2n)^4}\right)^2,\nonumber
     \end{align}
   where the term $\Cit^{2\sqrt 2n^\ep}$ is a uniform lower bound for the
   probability that all edges in $E$ are closed, and $[5(2n)^4)]^{-2}$
   comes from the fact that Lemma~\ref{lem:touch} gives 
   \begin{align*}
   \Pn{w_-\lr[*] \partial\tilde\Lambda_n^*\text{ in }D_-\,\Big|\,u\lr[D] v}&\ge \frac1{5(2n)^4}\quad\text{ and }\\
   \Pn{w_+\lr[*] \partial\tilde\Lambda_n^*\text{ in }D_+\,\Big|\,u\lr[D] v}&\ge \frac1{5(2n)^4}.
   \end{align*}
The event $\cal F$ forces the
   existence of a dual path disconnecting the cluster of $u$ and the
   cluster of $v$ (see Fig.~\ref{fig:construction2}).
   \begin{figure}[htbp]
     \centering
     \includegraphics[width=.495\textwidth]{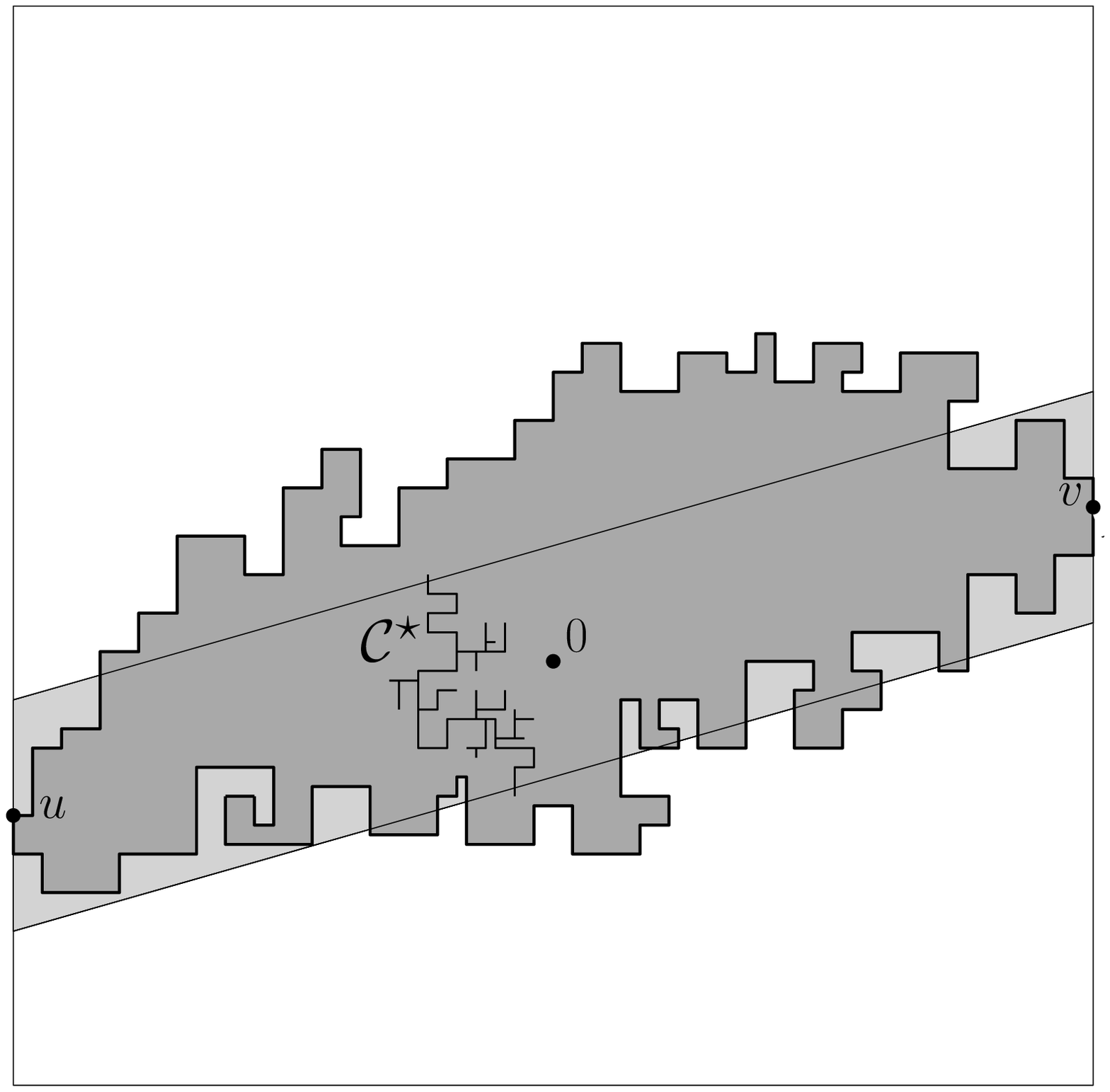} \includegraphics[width=.495\textwidth]{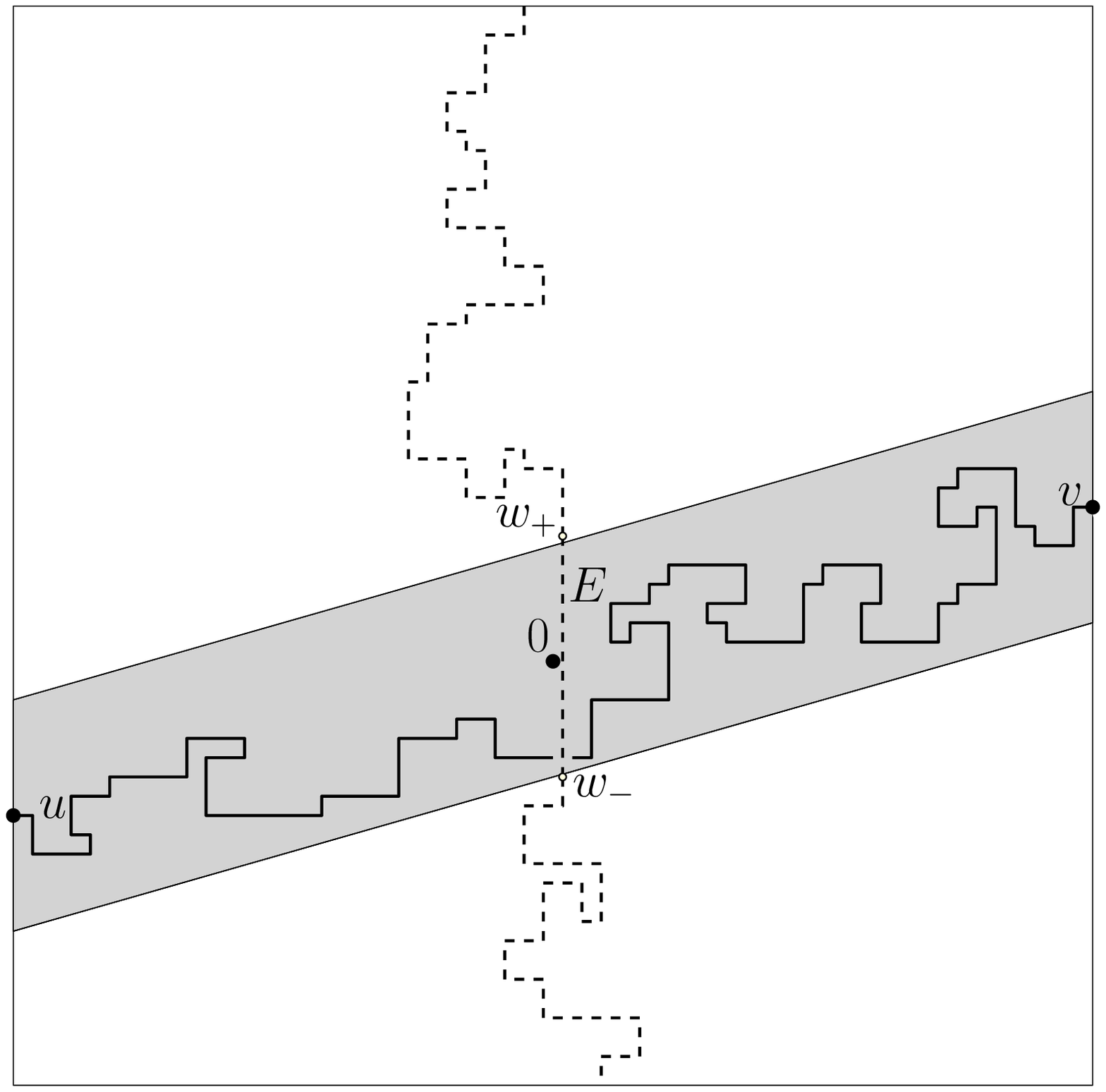}
     \caption{{\bf Left.} The domain $\Omega$ between $\Gamma_-$ and $\Gamma_+$. Inside, a dual cluster $\calC^*$ preventing the existence of an open path from $u$ to $v$ in $D$. Since we assumed that connection probabilities decay as a stretched exponential, this cluster exists with very small probability. {\bf Right.} Splitting the open path from $u$ to $v$ in two pieces.}
     \label{fig:construction2}
   \end{figure}
Conditioning on the cluster of $u$ and its boundary, the boundary conditions in what remains are dominated by free boundary conditions at infinity, and we deduce  that
\begin{equation}
\Pn{\cal E}\le e^{\cl{exp2}n^\ep}\P[\bbZ^2][0]{0\lr\partial\Lambda_{n}}\P[\bbZ^2][0]{0\lr\partial\Lambda_{n-1}}\le \frac{e^{\Cr{exp2}n^\ep}}{\Cit}\P[\bbZ^2][0]{0\lr\partial\Lambda_{n}}^2,
\end{equation}
where once again we used insertion tolerance in the last inequality. The claim follows readily.
 \end{proof}
 
 \begin{remark}
 The previous lemma implies that $\phi_{[0,2n]^2}^0(u\lr v)$ is bounded by the right hand-side of \eqref{eq:24} for any $u$ and $v$ on two opposite sides of $[0,2n]^2$. Let us argue that $\phi_{[0,2n-1]^2}^0(u'\lr v')$ is also bounded by a universal constant $C$ times the right-hand side of \eqref{eq:24} uniformly on $u'$ and $v'$ on opposite sides of $[0,2n-1]^2$. Indeed, the comparison between boundary conditions shows that 
 $$\phi_{[0,2n-1]^2}^0(u'\lr v')\le \phi_{[0,2n]^2}^0\big(u'\lr[{[0,2n-1]^2}] v'\big).$$
 Now let $u$ and $v$ be two neighbors of $u'$ and $v'$ on opposite sides of $[0,2n]^2$. The finite energy property implies that 
 $$\phi_{[0,2n-1]^2}^0(u'\lr v')\le \Cit\phi_{[0,2n]^2}^0(u\lr v)$$
 and we may apply the previous lemma.
  \end{remark}
\begin{proof}[Proof of Proposition~\ref{prop:dualityQ_n}]
Assume that there exists $\alpha>0$ such that 
\begin{equation}\label{eq:assump}\P[\bbZ^2][0]{0\longleftrightarrow\partial\Lambda_n}\le \exp(-n^\alpha)\end{equation}
for any $n\ge 0$.
Fix $\ep<\beta<\alpha$ to be chosen later. Set 
$$q_n=e^{n^\beta}\P[\bbZ^2][0]{0\leftrightarrow\partial\Lambda_{n}}.$$
Lemma~\ref{exterior circuit} applied to $2n$ and Lemma~\ref{lem:induction2} (together with the previous remark) imply that there exists $\Cl{20}>0$ such that
\begin{align*}
q_{2n}&\le e^{(2n)^\beta}\sum_{m\ge n}\Cr{20}m^4\Big(e^{\cl{C4}m^\ep}\phi^0(0\leftrightarrow\partial\Lambda_{m})^2\\
&\quad\quad\quad\quad\quad\quad\quad\quad\quad+\Cr{C4}m^6\sum_{\substack{k,\ell\ge m^\ep\\k+\ell= 2m}}\phi^0(0\lr\partial\Lambda_k)\phi^0(0\lr\partial\Lambda_\ell)\Big)\\
&\le e^{(2n)^\beta}\sum_{m\ge n}\Cr{20}m^4\Big(e^{\Cr{C4}m^\ep}e^{-2m^\beta}q_m^2+\Cr{C4}m^6\sum_{\substack{k,\ell\ge m^\ep\\k+\ell= 2m}}e^{-(k^\beta+\ell^\beta)}q_kq_\ell\Big)\\
&\le \Big(\max_{\substack{k,\ell\ge n^\ep\\k+\ell\ge 2n}}q_kq_\ell\Big)e^{(2n)^\beta}\sum_{m\ge n}\Cr{20}m^4\Big(e^{\Cr{C4}m^\ep}e^{-2m^\beta}+\Cr{C4}m^6\sum_{\substack{k,\ell\ge m^\ep\\ k+\ell= 2m}}e^{-(k^\beta+\ell^\beta)}\Big)\\
&\le \Cr{18}\max_{\substack{k,\ell\ge n^\ep\\k+\ell\ge 2n}}q_kq_\ell,
\end{align*}
where $\Cl{18}<\infty$ is a constant independent of $n$. The existence of $\Cr{18}$ follows from a simple computation using $\ep<\beta$ and the fact that $\beta<1$ and $k,\ell\ge n^\ep$ imply
$$e^{-(k^\beta+\ell^\beta)}\le e^{-(k+\ell)^\beta}e^{-\cl{1900}n^{\ep\beta}}$$
for some constant $\Cr{1900}>0$\  \footnote{Let us make a small remark before proceeding forward with the proof. It was crucial to keep the division in the inequality of Lemma~\ref{lem:induction2} between a term $k=\ell=m$ with a stretched exponential penalty $8m^3e^{\Cr{C4}m^\ep}$, and the general term $k+\ell=2m$, for which we have only a polynomial penalty $8m^3\Cr{C4}m^6$.  If we would have replaced the polynomial bound by a stretched exponential one for every $k$ and $\ell$, the values of $k$ or $\ell$ close to $n^\ep$ would have created difficulties since the correction would have been of the order of the largest of the two terms.}.

Let us now come back to the proof. The finite energy property implies the existence of $\cl{301}>0)$ such that
$\Cr{301}q_{k}\le q_{k+1}\le q_{k}/\Cr{301}$ for any $k\ge0$.
Using this fact, the previous inequality immediately extends to odd integers and there exists $\Cl{19}<\infty$ such that
\begin{align*}
q_{n}\le\Cr{19}\max_{\substack{k,\ell\ge n^\ep\\k+\ell\ge n}}q_kq_\ell.
\end{align*}
We unfortunately need to include the following technical trick. We do not know a priori that $(q_n)$ is decreasing. For this reason, we set $Q_n=\Cr{19}\max\{q_m:m\ge n\}$. For this definition, we get
\begin{align*}
Q_{n}\le\max_{\substack{k,\ell\ge n^\ep\\k+\ell\ge n}}Q_kQ_\ell.
\end{align*}
We are now in a position to conclude. The assumption implies that $(Q_n)$ tends to zero. Pick $n_0$ such that $Q_n<1$ for $n\ge n_0^{\ep}$.
Since $(Q_n)_{n\ge 0}$ is decreasing, the maximum of $Q_kQ_\ell$ is not reached for $k\ge n$ or $\ell\ge n$ and we obtain that for $n\ge n_0$,\begin{align*}
Q_{n}\le\max_{\substack{n>k,\ell\ge n^\ep\\k+\ell\ge n}}Q_kQ_\ell.
\end{align*}
We can now proceed by induction to prove  that for $n\ge n_0$,
\begin{align*}Q_{n}&\le \exp(-\cl{constant1}n)\quad\text{ where }\quad
\Cr{constant1}:=\max_{n_0^\ep\le n\le n_0}-\tfrac1n\log(Q_n)>0.\end{align*}
We therefore conclude that
$$\P[\bbZ^2][0]{0\longleftrightarrow\partial\Lambda_n}\le n\exp(n^\beta)\cdot\frac{1}{\Cr{19}}\exp(-\Cr{constant1}n).$$
\end{proof}

\section{Proof of Theorem~\ref{thm:decide}}
\label{sec:observable}

\subsection{An input coming from the theory of parafermionic observables}

We will use the parafermionic observable. For a complete exposition of the current
knowledge on parafermionic observables and a proof of this statement,
we refer to \cite[Chapter 6]{Dum13}.

\paragraph{Dobrushin domains.} In order to properly state and use the result, we first define the notion of Dobrushin domain. 

Let us start by defining the {\em medial lattice} $(\mathbb Z^2)^\diamond$, which is the graph with the centers of edges of $\mathbb Z^2$ as vertex set, and edges connecting nearest vertices. The vertices and edges of the medial lattice are called medial-vertices and medial-edges. This lattice is a rotated and rescaled (by a factor $1/\sqrt 2$) version of $\mathbb Z^2$. Edges of $(\mathbb Z^2)^\diamond$ are oriented counterclockwise around medial-faces having a vertex of $\bbZ^2$ at their center. Like that, the medial lattice can sometimes be seen as an oriented graph.

Let $a^\diamond$ and $b^\diamond$ be two distinct medial-vertices, and $\partial_{ab}^\diamond=\{v_0\sim v_1\sim \dots\sim v_n\}$,
$\partial_{ba}^\diamond=\{w_0\sim w_1 \sim\dots\sim w_m\}$ two paths of neighboring medial-vertices satisfying the following properties:
\begin{itemize}[nolistsep,nolistsep]
\item The paths start from $a^\diamond$ and end at $b^\diamond$, i.e. $v_0=w_0=a^\diamond$ and $v_n=w_m=b^\diamond$.
\item The paths follow the orientation of the medial lattice.
\item The path $\partial_{ab}^\diamond$ goes counterclockwise (around
  the set enclosed by $\partial_{ab}^\diamond\cup\partial_{ba}^\diamond$), while $\partial_{ba}^\diamond$ goes clockwise.
\item The paths are edge-avoiding.
\item The paths intersect only at $a^\diamond$ and $b^\diamond$.
\end{itemize}
Note that $\partial_{ab}^\diamond\cup\partial_{ba}^\diamond$ is a non self-crossing edge-avoiding polygon. However, some vertices might be visited twice. 

\begin{definition}[medial Dobrushin domains]
Let $\partial_{ab}^\diamond$ and $\partial_{ba}^\diamond$ be two paths as above, and let $\Omega^\diamond$ be the subgraph of $(\bbZ^2)^\diamond$ induced by the
medial-vertices that are enclosed by or in the path $\partial_{ab}^\diamond\cup\partial_{ba}^\diamond$.
Then, $(\Omega^\diamond,a^\diamond,b^\diamond)$ is called a {\em medial Dobrushin domain}. An example is given in Fig.~\ref{fig:medial domain}.
\end{definition}

As it stands, $a^\diamond$ and $b^\diamond$ have three incident medial-edges in $E_{\Omega^\diamond}$. Call $e_a$ and $e_b$ the fourth medial-edges incident to $a^\diamond$ and $b^\diamond$ respectively. 

\begin{figure}
\begin{center}
\includegraphics[width=0.80\textwidth]{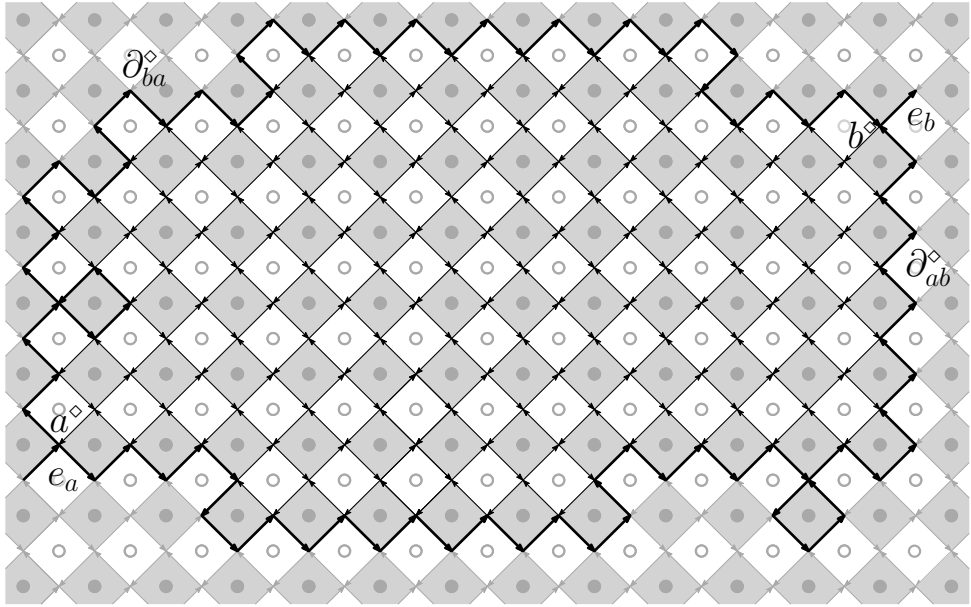}
\end{center}
\caption{\label{fig:medial domain} A medial Dobrushin domain. Note the position of $e_a$ and $e_b$.}
\end{figure}

\begin{definition}[primal and dual Dobrushin domains with two marked points]
Let $(\Omega^\diamond,a^\diamond,b^\diamond)$ be a medial Dobrushin domain. 

Let $\Omega\subset\bbZ^2$ be the graph with edge set composed of edges passing through end-points of medial-edges in $E_{\Omega^\diamond}\setminus\partial_{ab}^\diamond$ (if a medial-vertex is the end-point of a medial-edge in $E_{\Omega^\diamond}\setminus\partial_{ab}^\diamond$ and one in $\partial_{ab}^\diamond$, it is included) and vertex set given by the end-points of these edges. Let $a$ and $b$ be the two vertices of $\Omega$ bordered by $e_a$ and $e_b$. The triplet $(\Omega,a,b)$ is called a {\em primal Dobrushin domain}. We denote by $\partial_{ba}$ the set of edges corresponding to medial-vertices in $\partial\Omega^\diamond$ which are also end-points of medial-edges in $\partial_{ba}^\diamond$, and set $\partial_{ab}=\partial\Omega\setminus\partial_{ba}$.

Let $\Omega^*\subset\bbZ^*$ be the graph with edge set composed of dual-edges passing through medial-edges in $E_{\Omega^\diamond}\setminus \partial_{ba}^\diamond$ and vertex set given by the end-points of these dual-edges. Let $a^*$ and $b^*$ be the two dual-vertices of $\Omega^*$ bordered by $e_a$ and $e_b$. The triplet $(\Omega^*,a^*,b^*)$ is called a {\em dual Dobrushin domain}. We denote by $\partial_{ab}^*$ the set of dual-edges corresponding to medial-vertices in
$\partial\Omega^\diamond$ which are also end-points of medial-edges in $\partial_{ab}^\diamond$, and set $\partial_{ba}^*=\partial\Omega^*\setminus\partial_{ab}^*$.
\end{definition}

For a Dobrushin domain, let us define the Dobrushin boundary conditions on $(\Omega,a,b)$ to be wired on $\partial_{ba}$, and free on $\partial_{ab}$. The random-cluster measure with these boundary conditions and $p=p_c(q)$ is denoted by $\PP[\Omega][a,b]$. We enforce the boundary conditions as follows: we open every edge of $\partial_{ba}$ to create wired boundary conditions on this arc, and every dual-edge of $\partial_{ab}^*$ to create dual-wired boundary conditions on this dual arc (and therefore free boundary conditions on the corresponding primal arc by duality). 

\begin{figure}
\begin{center}
\includegraphics[width=0.80\textwidth]{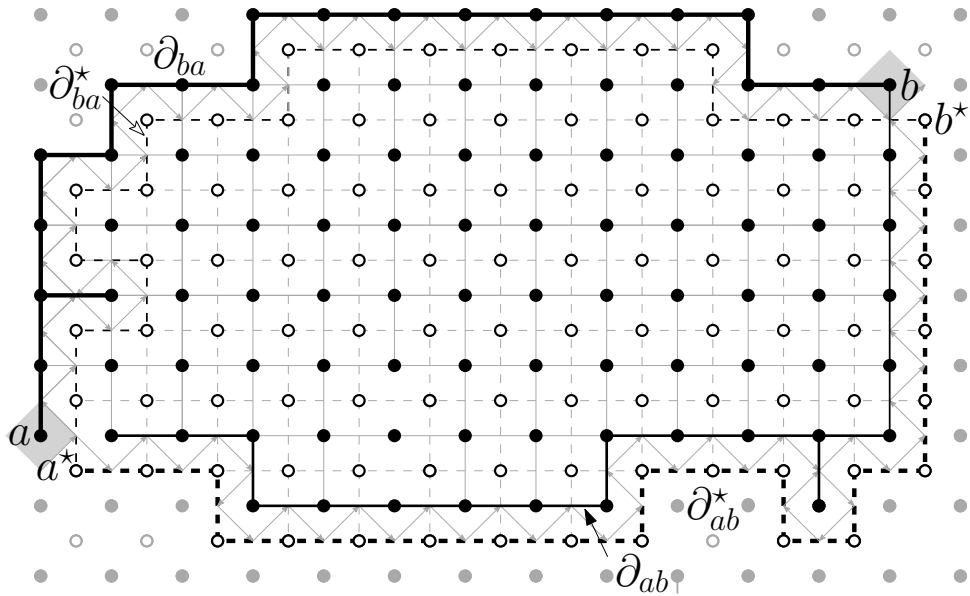}
\end{center}
\caption{\label{fig:primal domain} The primal and dual Dobrushin domains associated to a medial Dobrushin domain. Note the position of $a$, $a^*$, $b$ and $b^*$.}
\end{figure}

\begin{figure}[t]
\begin{center}
\includegraphics[width=0.80\textwidth]{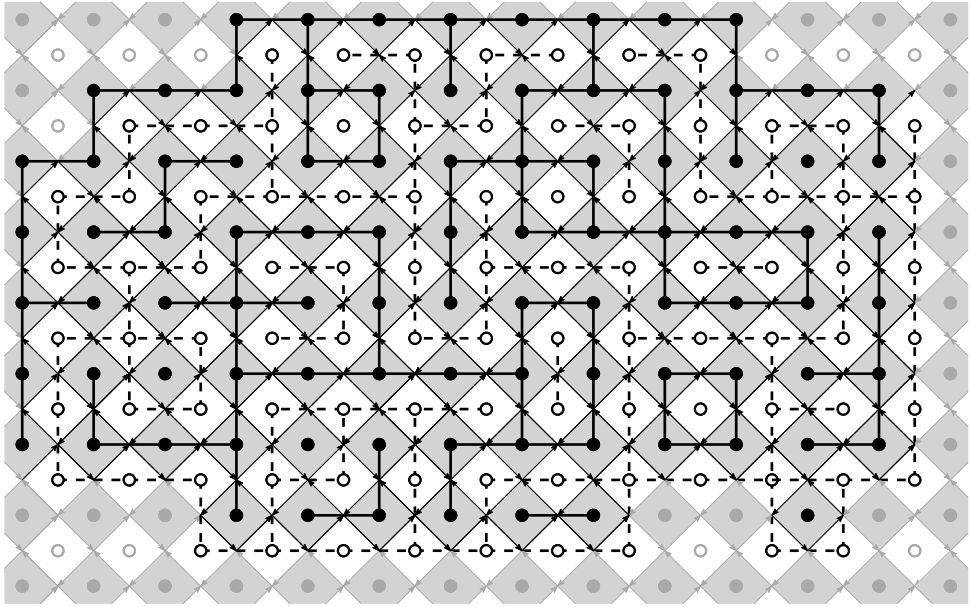}
\caption{\label{fig:1} The configuration $\omega$ with its dual configuration $\omega^*$.}
\end{center}
\end{figure}

\begin{figure}[t]
\begin{center}
\includegraphics[width=0.80\textwidth]{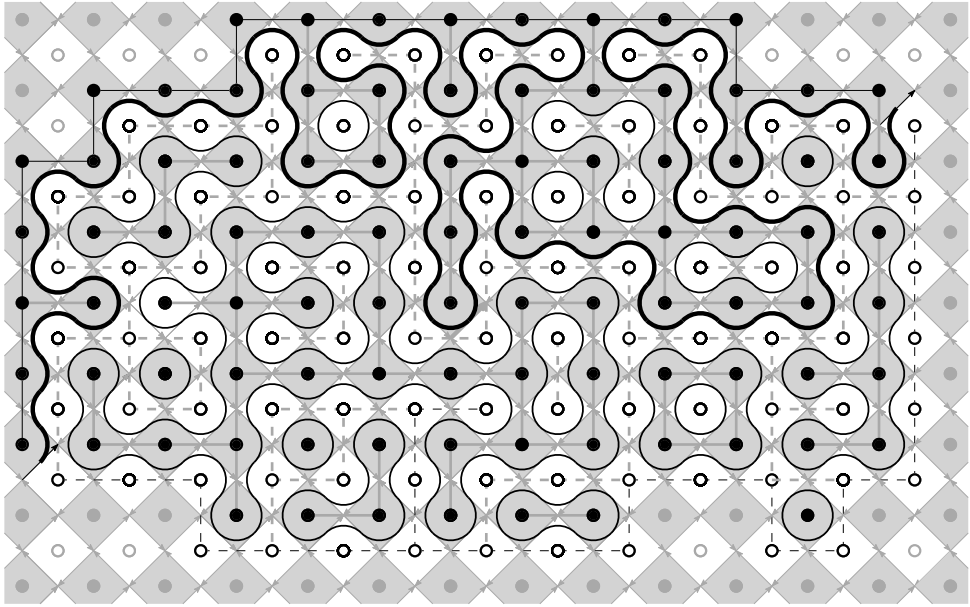} 

\caption{\label{fig:2} The loop representation associated to the primal and dual configurations in the previous picture. The exploration path is drawn in bold.}
\end{center}
\end{figure}

\paragraph{The main statement.} We start by defining the loop-configuration associated to a percolation-configuration. Fix a  Dobrushin domain $(\Omega,a,b)$ and consider a 
configuration $\omega\in\{0,1\}^{E_\Omega}$ together with its dual-configuration $\omega^\star\in\{0,1\}^{E_{\Omega^\star}}$. Through every vertex of the medial graph $\Omega^\diamond$ of $\Omega$ passes either an open 
bond of $\Omega$ or a dual open bond of $\Omega^\star$. Draw self-avoiding loops on $\Omega^\diamond$ as follows: a loop arriving at a vertex of the medial lattice 
always makes a $\pm \pi/2$ turn so as not to cross the open or dual open 
bond through this vertex, see Fig.~\ref{fig:2}. 
The loop representation contains loops together with a self-avoiding path going from $a^\diamond$ to $b^\diamond$, see Fig.~\ref{fig:2}. This curve is called the 
\emph{exploration path} and is denoted by $\gamma$. 

\begin{remark}The loops correspond to the interfaces separating clusters from dual 
clusters, and the exploration path  corresponds 
to the interface between the cluster connected to $\partial_{ba}$ and the 
dual cluster connected to $\partial_{ab}^*$.
\end{remark}
We are now in a position to define the parafermionic observable. 

\begin{definition}The \emph{winding} 
$\text{W}_{\Gamma}(e,e')$ of a curve $\Gamma$ (on the medial lattice) between two medial-edges $e$ and 
$e'$ of the medial graph is the total signed rotation in radians that the 
curve makes from the mid-point of the edge $e$ to that of the edge 
$e'$. By convention, if $\Gamma$ does not go through $e'$, we set $\text{W}_{\Gamma}(e,e')=0$.\end{definition}
For a curve drawn on the medial lattice, the winding can be computed in a very simple way: it corresponds to $\tfrac{\pi}2$ times the number of $\tfrac\pi2$-turns on the left minus the number of $\tfrac\pi2$-turns on the right.

\begin{definition}[Smirnov \cite{Smi10}]\label{def:parafermionic observable}
Consider a Dobrushin domain $(\Omega,a,b)$ and $q>0$. 
 The {\em edge parafermionic observable} $F=F(q,\Omega,a,b)$  is defined for any medial edge $e\in E_{\Omega^\diamond}$ by
\begin{equation*}
  F(e) ~:=~\begin{cases}\phi^{a,b}_{\Omega,p_c(q),q}[{\rm e}^{{\rm i}\sigma 
  \text{W}_{\gamma}(e,e_b)} \mathbf 1_{e\in \gamma}]&\text{ if $q\ne 4$,}\\
  \phi^{a,b}_{\Omega,p_c(4),4}[ 
  \text{W}_{\gamma}(e,e_b){\rm e}^{{\rm i} 
  \text{W}_{\gamma}(e,e_b)} \mathbf 1_{e\in \gamma}]&\text{ otherwise,}\end{cases}\end{equation*}
where $\gamma$ is the exploration path and $\sigma$ is given by the 
relation
$\displaystyle  \sin (\sigma \pi/2) = \sqrt{q}/2.$
\end{definition}

\begin{remark}
The observable $F_\delta(v)$ mentioned in Conjecture~\ref{FK parafermion} is defined on a vertex $v$ of $\Omega_\delta^\diamond=\delta(\bbZ^2)^\diamond\cap\Omega$ as half the sum of $F(e)$ over edges $e$ incident to $v$.
\end{remark}
\begin{remark}
Note that $\sigma$ is real for $q\le 4$, and belongs to $1+i\R$ for $q>4$.  For $q\in[0,4]$, $\sigma$ has the physical interpretation of a spin, which is fractional in general, hence the name parafermionic (fermions have half-integer spins while bosons have integer spins, there are no particles with fractional spin, but the use of such fractional spins at a theoretical level has been very fruitful in physics). For $q>4$, $\sigma$ is not real anymore and does not have any physical interpretation. \end{remark} 

The parafermionic observable satisfies a very specific property at criticality regarding contour integrals. 
One may define a dual $(\Omega^\diamond)^\star$ of $\Omega^\diamond$ in the following way: the vertex set of $(\Omega^\diamond)^\star$ is $V_\Omega\cup V_{\Omega^\star}$ and the edges of the dual connect nearest vertices together. We extend the definitions available for other graphs to this context.
  A {\em discrete contour} $\calC$ is a finite sequence $z_0\sim z_1\sim \dots\sim z_n=z_0$ in $(\Omega^\diamond)^\star$  such that the path $(z_0,\dots,z_n)$ is edge-avoiding. 
The discrete contour integral of the parafermionic observable $F$ along $\cal C$ is defined by
$$\oint_{\cal C}F(z)dz:=\sum_{i=0}^{n-1}(z_{i+1}-z_i)F\left(\{z_i,z_{i+1}\}^\star\right),$$
where the $z_i$ are considered as complex numbers and $\{z_i,z_{i+1}\}^\star$ denotes the edge of $\Omega^\diamond$ intersecting $\{z_i,z_{i+1}\}$ in its center. 

\begin{theorem}[Vanishing contour integrals]\label{thm:contours}
Let $(\Omega,a,b)$ be a Dobrushin domain, $q>0$ and $p=p_c$. For any discrete contour $\cal C$ of $(\Omega,a,b)$,
$$\oint_{\cal C}F(z)dz=0.$$
\end{theorem}

\begin{remark}The fact that discrete contour integrals vanish seems to be close to a well-known property of holomorphic functions: their contour integrals vanish. Nevertheless, one should be slightly careful when drawing such a parallel: the edge-observable is defined on edges, and should rather be understood as the discretization of a form than as the discretization of a function. As a form, the fact that these discrete contour integrals vanish should be interpreted as the discretization of the property of being closed. \end{remark}

For elementary contours surrounding a vertex of $\Omega^\diamond$ with four medial edges incident to it, the fact that discrete contour integrals vanish translate into the following relation:
\begin{equation}\label{eq:CR}F(A)-F(C)=\mathrm{i}(F(D)-F(B)),\end{equation}
where $A$, $B$, $C$ and $D$ are the four medial edges around the vertex, indexed in clockwise order. 
For clarity, we will use the following slightly modified version of the observable:
\begin{equation}
\widehat F(e)~:=~\begin{cases}\phi^{a,b}_{\Omega,p_c(q),q}[{\rm e}^{{\rm i}\widehat\sigma 
  \text{W}_{\gamma}(e,e_b)} \mathbf 1_{e\in \gamma}]&\text{ if $q\ne 4$,}\\
  \phi^{a,b}_{\Omega,p_c(4),4}[ 
  \text{W}_{\gamma}(e,e_b) \mathbf 1_{e\in \gamma}]&\text{ otherwise,}\end{cases}\end{equation}
where $\widehat \sigma=1-\sigma$ satisfies $\cos(\pi\widehat\sigma/2)=\sqrt q/2$.
With this definition, the relation that the discrete contour integrals
vanish can be restated as follows: \eqref{eq:CR} becomes
$$\widehat F(A)+\widehat F(C)=\widehat F(B)+\widehat F(D),
$$
and for any set $V$ of vertices of
$\Omega^\diamond$ having four incident edges in
$E_{\Omega^\diamond}\cup \{e_a,e_b\}$, we can sum the previous relation to get
\begin{equation}
  \sum_{e\text{ incident to exactly one vertex in }V} \eta(e)\widehat F(e)=0,
  \label{eq:main equation}
\end{equation}
where $\eta(e)$ equals 1 if $e$ points towards a vertex in $V$, and
$-1$ if it points away from a vertex in $V$. Notice that only edges with exactly one end-point in $V$ contribute to this sum.
\subsection{Proof of Theorem~\ref{thm:decide}}

The original proof of Theorem~\ref{thm:decide} can be found in \cite{Dum12}. However, we choose to present a shorter proof here which is based on some of the new arguments of the previous section. 

In this section, we fix $p=p_c(q)$.  For simplicity, we start by treating the $1\le q\le2$ case in order to illustrate the strategy. Then, we focus on the general $1\le q\le 4$ case. 
\begin{proof}[Proof of Theorem~\ref{thm:decide} in the $1\le q\le 2$ case]
  For $n\ge1$, consider the set $R_n:=[0,n]\times[-n,n]$ and recall
  that $\Lambda_n=[-n,n]^2$. See $R_n$ as a Dobrushin domain with
  wired arc equal to the vertex $(0,0)$ (see Fig~\ref{fig:Rn}). In
  such case, the exploration path is the loop going around $(0,0)$,
  and $e_a$ and $e_b$ are the two edges of the boundary bordering
  (0,0). 

\begin{figure}[htbp]
  \centering
  \hfill
  \begin{minipage}[b]{.47\linewidth}
    \includegraphics[width=\linewidth]{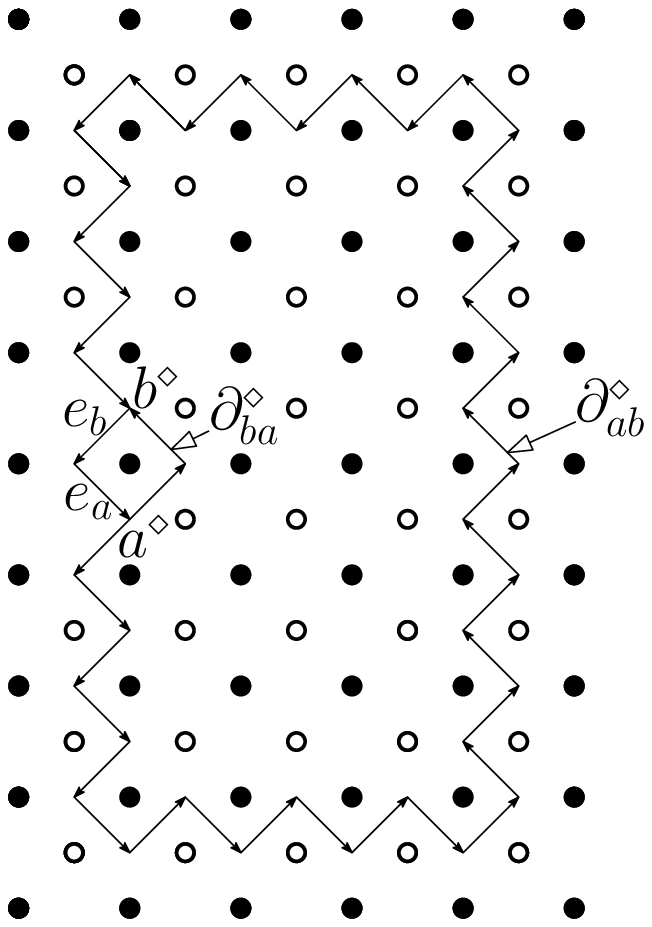}
   
  \end{minipage}
  \hfill
    \begin{minipage}[b]{.47\linewidth}
    \includegraphics[width=\linewidth]{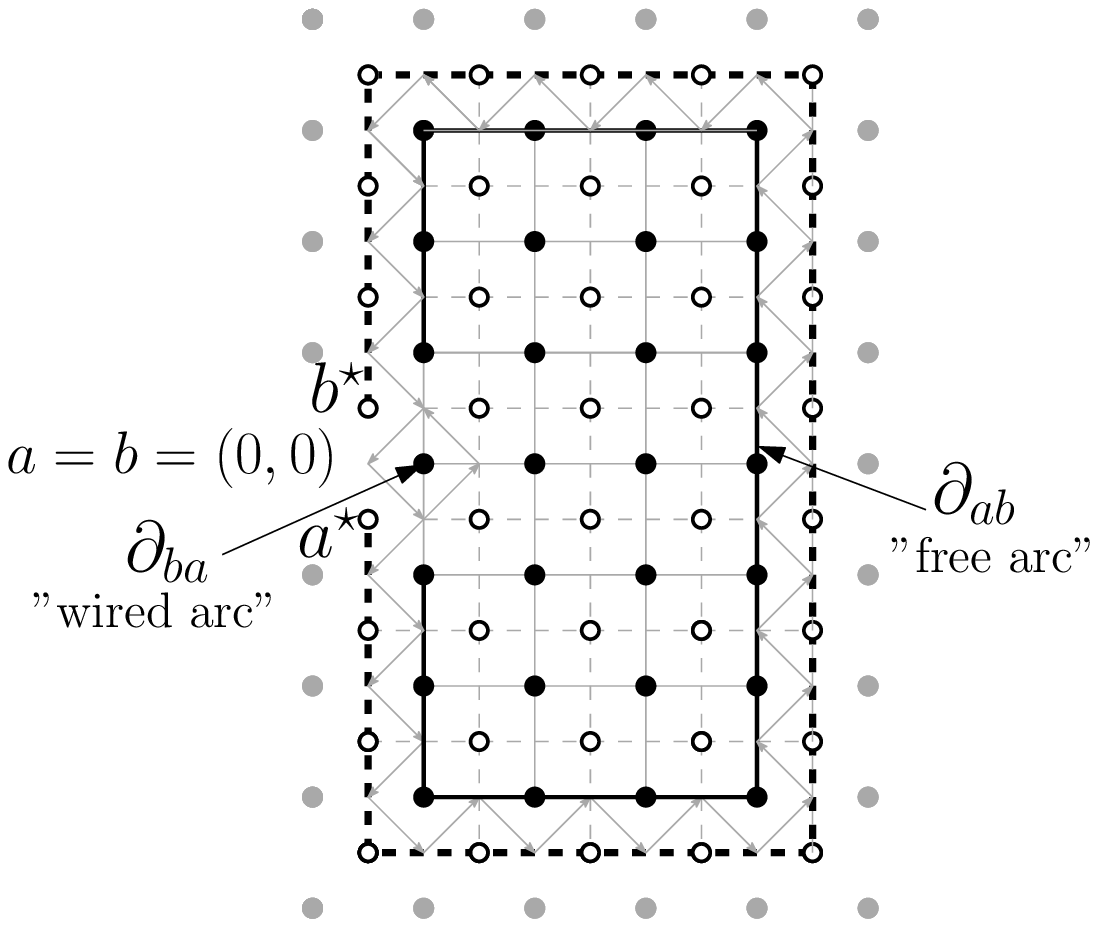}
  
  \end{minipage}
  \hfill
\caption{\textbf{Left.} The medial paths $\partial^\diamond_{ab}$ and
  $\partial^\diamond_{ba}$  allowing to define $R_n$ as a Dobrushin
  domain.    \textbf{Right.} The set $R_n$, seen as a Dobrushin
  domain.}
\label{fig:Rn}
\end{figure}

Let $V$ be the set of medial vertices of $R_n^\diamond$ with four
incident medial edges in $E_{R_n^\diamond}\cup \{e_a,e_b\}$.
Alternatively, $V$ can be defined as the set of medial-vertices
corresponding to the edges in $R_n$. Equation \eqref{eq:main equation}
implies
 \begin{equation}\sum_{e\in \partial_1}\eta(e) \widehat
   F(e)=-\sum_{e\in\partial_2}\eta(e) \widehat F(e),\end{equation}
 where $\partial_1$ (resp. $\partial_2$) is the set of medial-edges of
 $R_n^\diamond$ incident to exactly one medial-vertex of $V$, and  bordering a vertex of $\partial R_n\cap\Lambda_{n-1}$ (resp. $\partial R_n\cap\partial\Lambda_n$).
 \medbreak
 Let us bound the complex modulus of the sums on the left and right from below and above respectively. We start by the sum on the right. Observe that
$$|\widehat F(e)|\le \phi^0_{R_n}(e\in \gamma)=\phi^0_{R_n}(0\longleftrightarrow x)\le \phi^0(0\longleftrightarrow x),$$
where $x$ is the vertex of $\partial R_n\cap\partial\Lambda_n$ bordered by the medial edge $e$. In the first inequality, we used that $\widehat\sigma\in \bbR$ and therefore ${\rm e}^{{\rm i}\widehat\sigma {\rm W}_{\gamma}(e,e_b)}$ is of modulus 1, in the equality, the fact that $e$ is on the exploration path if and only if $0$ and $x$ are connected, and in the second inequality the comparison between the boundary conditions. Since every such vertex is bordered by two medial edges, we find that
\begin{equation}\label{eq:33}\Big|-\sum_{e\in\partial_2}\eta(e)\widehat
  F(e)\Big|\le 2\sum_{x\in\partial R_n\cap\partial\Lambda_n}\phi^0(0\longleftrightarrow x)\le 2\sum_{x\in\partial\Lambda_n}\phi^0(0\longleftrightarrow x).\end{equation}
Let us now bound the left-hand side from below. First, consider the two medial edges $e_a$ and $e_b$. Their contributions are given by 
$$F(e_a)-F(e_b):=e^{i\widehat\sigma3\pi/2}-1=2{\rm i}{\rm e}^{{\rm i}\widehat\sigma 3\pi/4}\sin(\widehat\sigma3\pi/4).$$
Second, consider the medial edges $e_1$, $e_2$, $e_3$ and $e_4$ of $\partial_1$ bordering the vertices $x=(0,x_2)$ and $-x=(0,-x_2)$ where $1\le x_2<n$. Observe that for edges on the boundary, the winding of any possible realization of the exploration path going through an edge $e$ is deterministic (we denote it by ${\rm W}(e)$). As a consequence,
$$F(e)={\rm e}^{{\rm i}\widehat\sigma {\rm W}(e)}\phi^0_{R_n}(e\in \gamma)={\rm e}^{{\rm i}\widehat\sigma {\rm W}(e)}\phi^0_{R_n}(0\longleftrightarrow x).$$
Therefore, using the explicit values of $W(e_1)$, $W(e_2)$, $W(e_3)$ and $W(e_4)$, we obtain
\begin{align*}\sum_{i=1}^4\eta(e_i)F(e_i)&=\Big(\sum_{i=1}^4\eta(e_i){\rm e}^{{\rm i}\widehat\sigma {\rm W}(e_i)}\Big)\phi^0_{R_n}(0\longleftrightarrow x)\\
&=\Big({\rm e}^{{\rm i}\widehat\sigma3\pi/2}-{\rm e}^{{\rm i}\widehat\sigma2\pi}+{\rm e}^{-{\rm i}\widehat\sigma\pi/2}-1\Big)\phi^0_{R_n}(0\longleftrightarrow x)\\
&=-4{\rm i}{\rm e}^{{\rm i}\widehat\sigma 3\pi/4}\cos(\widehat\sigma \pi)\sin(\widehat\sigma\pi/4)\phi^0_{R_n}(0\longleftrightarrow x).
\end{align*}
Altogether, for $q\le 2$, $1/2 \le\widehat\sigma\le1$ and we deduce
\begin{align*}\Big|\sum_{e\in \partial_1}\eta(e)F(e)\Big|&=2\sin(\widehat\sigma3\pi/4)-4\cos(\widehat\sigma \pi)\sin(\widehat\sigma\pi/4)\sum_{x\in \partial R_n\cap\Lambda_{n-1}}\phi^0_{R_n}(0\longleftrightarrow x)\\
&\ge2\sin(\widehat\sigma3\pi/4).\end{align*}
Putting this relation together with \eqref{eq:33}, we find that for every $n\ge1$,
\begin{equation}\label{eq:34}\sum_{x\in\partial\Lambda_n}\phi^0(0\longleftrightarrow x)\ge \sin(\widehat\sigma3\pi/4).\end{equation}
Summing over every $n$ shows {\bf P3} of Theorem~\ref{thm:main} and therefore Theorem~\ref{thm:decide}.
\end{proof}
\begin{remark}
It is worth mentioning that the strategy used for $1\le q\le 2$ cannot work directly for $q>2$. Indeed, predictions using conformal invariance give that 
$$\phi^0_{\bbN\times\bbZ}(0\longleftrightarrow \partial\Lambda_n)=n^{-\alpha(\pi,q)+o(1)}$$
where $\alpha(\pi,q)$ is a critical exponent. The value of the exponent increases from $1/3$ to $1$ when $q$ goes from $1$ to $4$. In particular, for $q>2$, the value becomes larger or equal to $1/2$, and it is natural to expect that
$$\phi^0_{R_n}(0\longleftrightarrow x)\approx \phi^0_{\bbN\times\bbZ}(0\longleftrightarrow \partial\Lambda_n/2)^2=n^{2\alpha(\pi,q)+o(1)}\ll 1/n$$
for $n$ large enough. This is in contradiction with \eqref{eq:34}. For $q>2$, we will therefore extend our reasoning and work with domains of the form $R_n^{\theta_0}:=\{re^{i\theta}\in\bbZ^2:\theta\in[-\theta_0,\theta_0]\}$ with $\theta_0\ge\pi/2$. Indeed, in this case, conformal invariance predicts that $\phi^0_{R_n^{\theta_0}}(0\longleftrightarrow\partial\Lambda_n)=n^{-\alpha(\theta_0,q)}$, where $\alpha(\theta_0,q)=\alpha(\pi,q)\cdot\pi/\theta_0$, and therefore
$$\phi^0_{R_n^{\theta_0}}(0\longleftrightarrow x)\approx n^{\alpha(\pi,q)+\alpha(\theta_0,q)+o(1)}.
$$
This gives hope that a strategy similar to the previous one
could work as long as $\alpha(\pi,q)\le \theta_0/(\pi+\theta_0)$. Since $\alpha(\pi,q)$ tends to 1 as $q\le 4$, it would in fact be necessary to work with domains which are not subsets of $\bbZ^2$ (for instance, we will consider subsets of the graph $\mathbb
U$, defined below). \end{remark}
\medbreak
We now focus on the proof of Theorem~\ref{thm:decide} in the $1\le q\le 4$ case. \medbreak
Set $C_n$ to be the slit domain obtained by removing from $\Lambda_n$
the edges between the vertices of  $\{(0,k):0\le k\le n\}$. We define
Dobrushin boundary conditions on $C_n$ to be wired on $\{(0,k):0\le
k\le n\}$ and free elsewhere. For simplicity, we now refer to
$\{(0,k):0\le k\le n\}$ as the {\em wired arc}. The measure on $C_n$
with these boundary conditions is denoted $\PP[C_n][\rm dobr]$.
Equivalently, one may obtain $\PP[C_n][\rm dobr]$ by taking
$\P[\Lambda_n][0]{\,\cdot\,|\omega(e)=1 \text{ for\ all\ }e\ \mathrm{ in\ wired\ arc}}$ and we therefore think of $C_n$ as the box $\Lambda_n$ with free boundary conditions and $\{(0,k):0\le k\le n\}$ wired; see Fig.~\ref{fig:case1}.
\begin{lemma}\label{lem:crucial}
Let $q\in[1,4]$, there exists $c>0$ such that for any $n\ge 1$, 
$$\P[C_n][\rm dobr]{(0,-n)\longleftrightarrow\mathrm{wired\ arc}}\ge \frac{c}{n^{16}}.$$
\end{lemma}

The main estimate used in the proof of this lemma will follow from \eqref{eq:main equation} applied in a well-chosen domain. Then, we
compare boundary conditions in this domain to Dobrushin boundary
conditions in $C_n$. To exploit the whole power of
\eqref{eq:main equation}, we will consider a domain which is non-planar.
Namely, let us introduce the following graph $\mathbb U$ (see
Fig.~\ref{fig:U}): the vertices are given by $\Z^3$ and the edges by
 \begin{itemize}[noitemsep,nolistsep]
 \item $[(x_1,x_2,x_3),(x_1,x_2+1,x_3)]$ for every $x_1,x_2,x_3\in \Z$,
 \item $[(x_1,x_2,x_3),(x_1+1,x_2,x_3)]$ for every $x_1,x_2,x_3\in \Z$ such that $x_1\neq 0$,
 \item  $[(0,x_2,x_3),(1,x_2,x_3)]$ for every $x_2\ge0$ and $x_3\in\Z$,
 \item $[(0,x_2,x_3),(1,x_2,x_3+1)]$ for every $x_2<0$ and $x_3\in \Z$.
 \end{itemize}
This graph is the universal cover of $\mathbb Z^2\setminus F$, where $F$ is the face centered at $(\tfrac12,-\tfrac12)$. It can also be seen at $\Z^2$ with a branching point at $(\tfrac12,-\tfrac12)$.
All definitions of dual and medial graphs extend to this context, as well as \eqref{eq:main equation}. \begin{figure}[t]
\begin{center}
\includegraphics[width=1\textwidth]{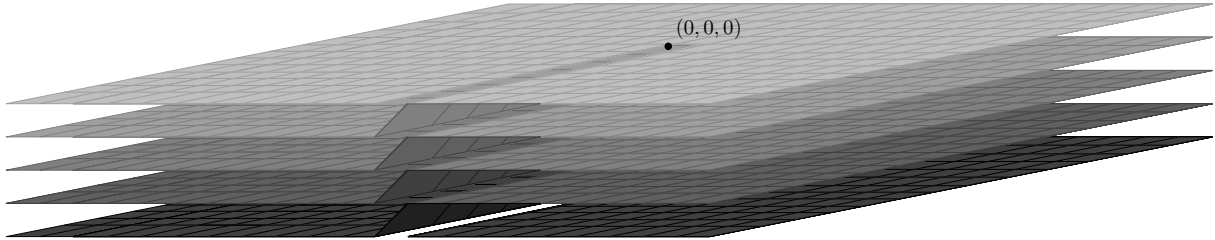}
\caption{\label{fig:U}The graph $\mathbb U$.}
\end{center}
\end{figure}

\begin{proof}
For $n\ge 1$, define 
$$U_n:=\big\{(x_1,x_2,x_3)\in\mathbb U:|x_1|,|x_2|\le n\text{ and
}|x_3|\le n^5\big\}.$$ We wish to apply \eqref{eq:main equation} to
the Dobrushin domain obtained from $U_n$ by fixing the wired arc to be
the vertex (0,0,0). Even if the domain is non-planar, one can easily
define a medial graph. The medial edges $e_a$ and $e_b$ both correspond to the medial edge
$e$ of the boundary of this domain which is adjacent to 0, and the
exploration path is the loop going through this edge. For simplicity,
and in order to define the wiring of the curve between $e_a$ and $e_b$
correctly, we will consider $e_a$ and $e_b$ as begin two different
edges, $e_a$ (resp. $e_b$) being adjacent to only one medial vertex of
the medial graph, namely the endpoint of $e$ towards which (resp. from
which) $e$ is pointing. Let $V$ be the set of medial-vertices incident
to four edges of $E_{U_n^\diamond}\cup\{e_a,e_b\}$. Equation
\eqref{eq:main equation} applied to $V$ gives exactly as in the previous proof, that
$$F(e_a)-F(e_b)=-\sum_{e\in \partial}\eta(e)F(e),
$$
where $\partial$ is the set of medial edges different from $e_a$ and $e_b$ incident to only one vertex of $V$. Now, let us focus for a moment on the $q<4$ case. Then, $F(e_a)-F(e_b)=e^{i\widehat\sigma 2\pi}-1$, and as before $|-\sum_{e\in \partial}\eta(e)F(e)|\le 2\sum_{x\in\partial U_n}\P[U_n][0]{x\lr 0}$. We immediately deduce that
\begin{equation}\label{eq:35}
\sum_{x\in\partial U_n}\P[U_n][0]{0\lr x}\ge c_1
\end{equation}
for some constant $c_1=c_1(q)>0$ independent of $n$. The same computation may be performed for $q=4$ and we also get \eqref{eq:35}: indeed in such case, 
$$F(e_a)-F(e_b)=2\pi-1$$
and for two medial edges $e,e'\in\partial$ bordering the same primal vertex $x$ on the boundary,
$$|\eta(e)F(e)+\eta(e')F(e')|=|W(e)-W(e')|\cdot\P[U_n][0]{0\lr x}\le \frac{3\pi}2\P[U_n][0]{0\lr x},$$
where $W(e)$ (resp. $W(e')$) denotes the winding of any possible loop going from $e$ (resp. $e'$) to $e_b$.

We now wish to bootstrap this estimate to an estimate on $C_n$. Let us start by proving the following claim (observe that $|x_3|<n^5$ in the statement). 
\medbreak
\noindent{\em Claim: There exists $c_2>0$ (independent of $n$) such that  there exists $x=(x_1,x_2,x_3)\in\partial U_n$ with $|x_3|<n^5$ and}
$$\P[U_n][0]{0\lr x}\ge \frac{c_2}{n^{6}}.$$
\medbreak
We will prove this fact by showing that vertices $x$ with $|x_3|=n^5$ have very small probability of being connected to the origin and therefore cannot account for much in \eqref{eq:35}.
\medbreak
\noindent{\em Proof of the Claim.} 
Let $R_0^*$ be the dual graph of the subgraph of $U_n$ with vertex set $R_0:=[-n,n]\times[0,n]\times\{0\}$, i.e.\@ the graph with edge set $\{e^*:e\in E_{R_0}\}$ and vertex set given by the end-points of these edges. Note that uniformly in the state of edges outside $R_0$, the boundary conditions in $R_0$ are  dominated by wired boundary conditions on the ``bottom side'' $[-n,n]\times\{0\}\times\{0\}$ of $R_0$, and free elsewhere. Passing to the dual model, we deduce that uniformly in the state of edges outside $R_0$, \begin{align*}&\P[U_n][0]{(-\tfrac12,-\tfrac12,0)\lr[*]\partial U_n^*\text{ in }R_0^*\Big|\,\text{edges outside }R_0}\ge \P[R_0^*][1/0]{(-\tfrac12,-\tfrac12,0)\lr[*]\partial U_n^*\text{ in }R_0^*},\end{align*}
where $\PP[R_0^*][1/0]$ is the (dual) random-cluster measure on $R_0^*$ with free boundary conditions on the bottom and wired  boundary conditions everywhere else.  Lemma~\ref{lem:touch} (with $m=n$ and $\theta=0$) thus implies that 
\begin{align}&\P[U_n][0]{(-\tfrac12,-\tfrac12,0)\lr[*]\partial U_n^*\text{ in }R_0^*\Big|\,\text{edges outside }R_0}\ge \frac1{5n^4}\label{eq:36}.\end{align} 
The same is also true for $R_k^*=(0,0,k)+R_0^*$ with $|k|\le n^5$.

If a vertex $x=(x_1,x_2,x_3)\in \partial U_n$ with $x_3=n^5$ is connected to $(0,0,0)$, then none of the dual vertices $(-\tfrac12,-\tfrac12,k)$ are dual connected to $\partial U_n$ in $R_k^*$ for $0< k<x_3$ (the symmetric claim holds for $x_3=-n^5$). Equation~\eqref{eq:36} applied $|x_3|-1$ times implies that
\begin{align*}\P[U_n][0]{0\longleftrightarrow x}~&\le~ \Big(1-\frac1{5n^4}\Big)^{|x_3|-1}.\end{align*}
The probability is therefore exponentially small when $|x_3|=n^5$. 
Together with \eqref{eq:35}, the previous inequality implies that for $n$ large enough,
$$\sum_{x\in\partial U_n:|x_3|<n^5}\P[U_n][0]{0\lr x}\ge \frac{c_1}2.$$
The claim follows directly from the union bound, provided that $c_2$ is chosen small enough.
\begin{flushright}$\diamond$\end{flushright}

Fix $x$ given by the claim and rotate and translate vertically\footnote{Seen as a graph, $U_n$ is invariant by rotation by $\pi/2$ since the line where $x_3$ ``increases'' is invisible from inside $U_n$.}  $U_n$ in such a way that $x=(x_1,-n,0)$ for some $-n\le x_1\le n$. Consider $C_n$ as a subgraph of $U_n$. The boundary conditions on $C_n$ induced by the free boundary conditions on $U_n$ are dominated by the Dobrushin boundary conditions on $C_n$ defined above. Furthermore, the existence of an open path from $x$ to the origin implies the existence of a path from $x$ to the wired arc in $C_n$. Thus, the claim implies that
\begin{equation}\label{eq:eq}\P[C_n][\rm dobr]{\textstyle x\lr \mathrm{wired\ arc}}~\ge~ \displaystyle\frac{c_2}{n^6}.\end{equation}

To conclude the proof, we need to obtain a lower bound for the probability that the vertex $(0,-n,0)$ itself is connected to the wired arc. We use once again a ``conditioning on the right-most and left-most paths type argument''. Since we now work on a sub-domain of $\Z^2$, we drop the third coordinate from the notation. 

We may assume that $x_1\ge0$ and that the two vertices $x=(x_1,-n)$ and $(-x_1,-n)$ are connected to the wired arc. The FKG inequality implies that this occurs with probability $(\frac{c_2}{n^6})^2$. Consider the right-most open path from $(x_1,-n)$ to the wired arc, and the left-most open path from $(-x_1,-n)$ to the wired arc. Let $S$ be the part of $C_n$ between these two paths, see Fig.~\ref{fig:case1}. The  boundary conditions in $S$  dominate the free boundary conditions on the bottom of $C_n$, and wired elsewhere. 
We use a comparison between boundary conditions. The reasoning is the
same as usual: we compare boundary conditions on $S$ with the boundary
conditions induced by boundary conditions on $\Lambda_n$ with free
boundary conditions on the bottom and wired boundary conditions on the
other sides. Lemma~\ref{lem:touch} (applied to $2n$, $m=n$ and
$\theta=0$) thus implies that $(0,-n)$ is connected to the wired arc
with conditional probability larger than $\frac{1}{20n^4}$, and we finally
obtain
\begin{equation}
  \P[C_n][\rm dobr]{(0,-n)\longleftrightarrow\mathrm{wired\ arc}}\ge \left(\frac{c_2}{n^{6}}\right)^2\frac1{20n^4}.
\end{equation}
\end{proof}

\begin{figure}
\begin{center}
\includegraphics[width=7cm]{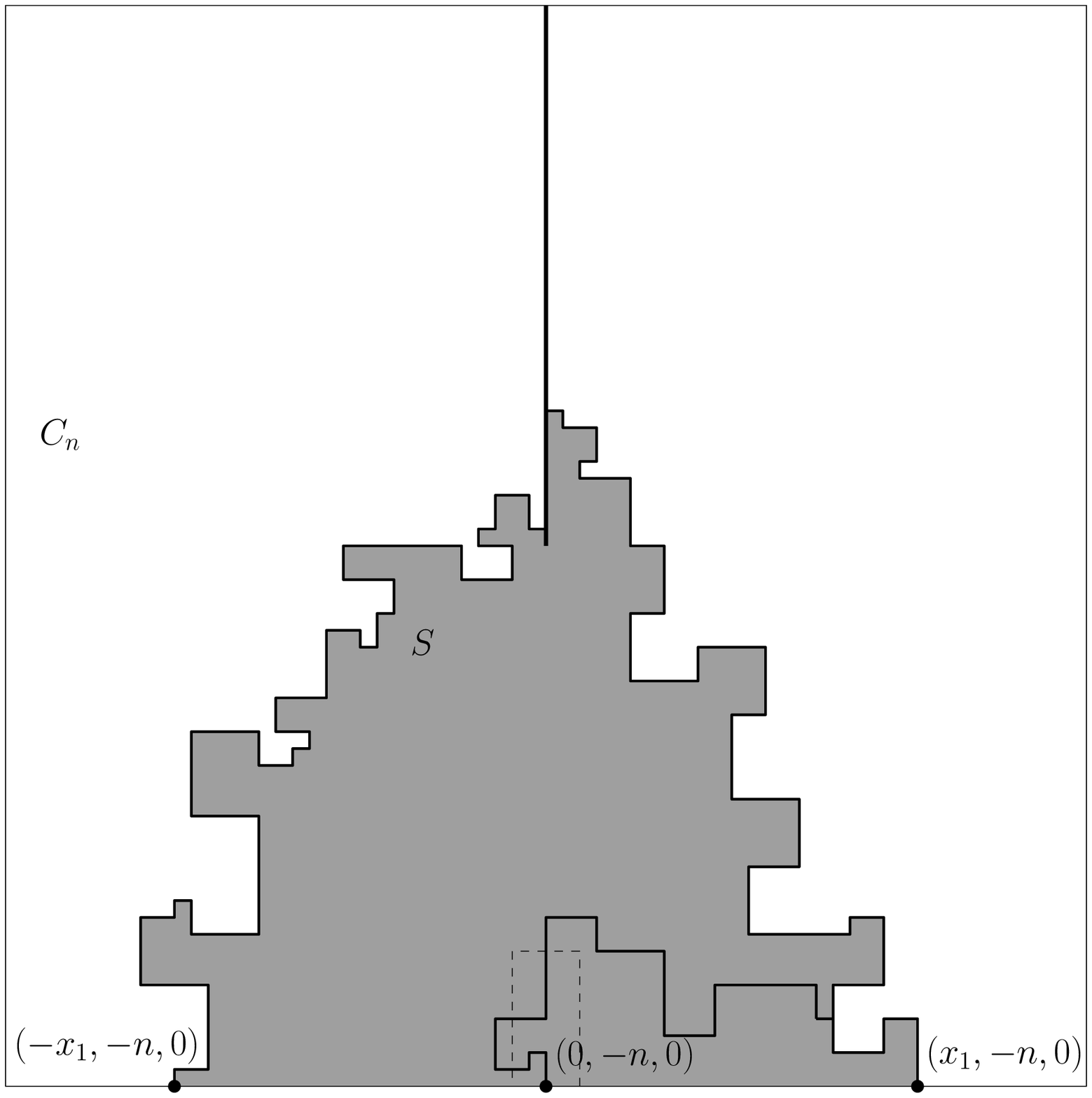}\quad\quad\quad \includegraphics[width=7cm]{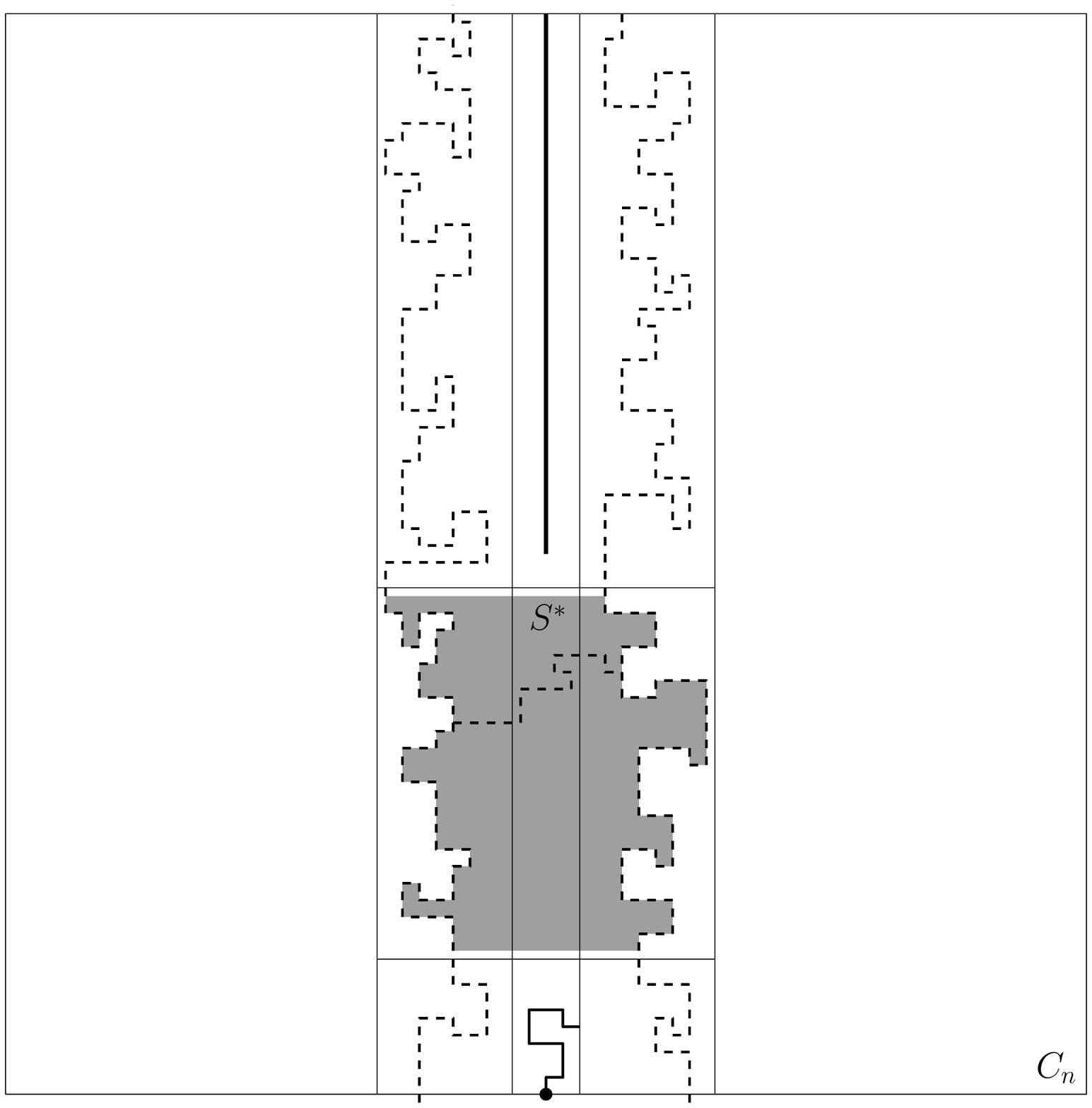}
\caption{\label{fig:case1}{\bf Left.} The two paths connecting the wired arc to $(x_1,-n,0)$ and $(-x_1,-n,0)$ (or simply $(x_1,-n)$ and $(-x_1,-n)$) and the area $S$ between them. \label{fig:case2}{\bf Right.} The two dual-open paths in the long rectangles $R_{\rm right}^*$ and $R_{\rm left}^*$.}
\end{center}
\end{figure}

We are now in a position to prove Theorem~\ref{thm:decide}. 
Let $\partial_n$ be the set of vertices at distance $\tfrac{n}{16}$ of the vertex $(0,-n)$ in $C_n$. The reasoning is similar to the proof of Lemma~\ref{lem:induction1} except that instead of isolating primal circuits around $z_-$ and $z_+$ from each other, we will isolate the primal path from $(0,-n)$ to $\partial_n$ from the wired arc. 
\begin{proof}[Proof of Theorem~\ref{thm:decide} for {$q\in(2,4]$}]
Introduce the three rectangles 
\begin{align*}&R_{\rm right}^*:=\left[\tfrac n{16}+\tfrac12,\tfrac{5n}{16}-\tfrac12\right]\times[-(n+\tfrac12),n+\tfrac12],\\
&R_{\rm left}^*:=\left[-\tfrac{5n}{16}+\tfrac12,-\tfrac{n}{16}-\tfrac12\right]\times[-(n+\tfrac12),n+\tfrac12],\\
&R^*:=\big[-\tfrac{5n}{16}+\tfrac12,\tfrac{5n}{16}-\tfrac12\big]\times\big[-\tfrac{3n}{4}+\tfrac12,-\tfrac{n}{8}-\tfrac12\big].\end{align*} Define the three events
$\mathcal E=\{(0,-n)\longleftrightarrow \partial_n\text{ in $C_n$}\}$,
$\mathcal F_{\rm right}$ and $\mathcal F_{\rm left}$ that there exists a dual-open dual-path from bottom to top in $R_{\rm right}^*$ and $R_{\rm left}^*$ respectively. Let $\mathcal C$ be the event that there exists a dual-open dual-path in $R^*$ connecting a dual open path crossing $R_{\rm left}^*$ from top to bottom to a dual open path crossing $R_{\rm right}^*$ from top to bottom.

Conditioning on $\mathcal F_{\rm left}\cap \mathcal F_{\rm right}\cap \mathcal C$, boundary conditions on $R_n$ are  dominated by free boundary conditions in the plane. Therefore
\begin{align*}\P[\Z^2][0]{\textstyle0\leftrightarrow \partial \Lambda_{n/16}}\ge \P[C_n][\rm dobr]{\mathcal E| \mathcal F_{\rm left}\cap \mathcal F_{\rm right}\cap \mathcal C}\ge \P[C_n][\rm dobr]{\mathcal E\cap \mathcal F_{\rm left}\cap \mathcal F_{\rm right}\cap \mathcal C}.\end{align*}
We now prove a lower bound on the term on the right:
\begin{align*}&\P[C_n][\rm dobr]{\mathcal E\cap \mathcal F_{\rm left}\cap \mathcal F_{\rm right}\cap \mathcal C}=\P[C_n][\rm dobr]{\mathcal E}\cdot\P[C_n][\rm dobr]{\mathcal F_{\rm left}\cap \mathcal F_{\rm right}|\mathcal E}\cdot\P[C_n][\rm dobr]{\mathcal C|\mathcal E\cap\mathcal F_{\rm left}\cap \mathcal F_{\rm right}}.\end{align*}

First, Lemma~\ref{lem:crucial} implies that $\phi^{\rm dobr}_{C_n}(\mathcal E)\ge \frac{c}{n^{16}}$. 
Second, conditioned on everything on the left of $\{\frac n{16}\}\times[-n,n]$, boundary conditions on $[\frac n{16},n]\times[-n,n]$ are  dominated by wired boundary conditions on the left side and free elsewhere. In particular, boundary conditions for the dual model  dominate free boundary conditions on the left side and wired elsewhere. Lemma~\ref{lem:push} implies that $\P[C_n][\rm dobr]{\mathcal F_{\rm right}|\calE}\ge c_2$ and the same lower bound holds true for $\P[C_n][\rm dobr]{\mathcal F_{\rm left}|\mathcal F_{\rm right}\cap\calE}$. We obtain
$$\P[C_n][\rm dobr]{\mathcal F_{\rm left}\cap \mathcal F_{\rm right}|\mathcal E}\ge c_2^2.$$
Third, we turn to $\P[C_n][\rm dobr]{\mathcal C|\mathcal E\cap\mathcal F_{\rm left}\cap \mathcal F_{\rm right}}$. Let $S^*$ be the area of the dual graph in $R^*$ between the right-most dual-open path from top to bottom in $R_{\rm right}^*$, and the left-most dual-open path crossing from top to bottom in $R_{\rm left}^*$, see Fig.~\ref{fig:case2}. The boundary conditions for the dual model on $S^*$ dominate free boundary conditions on top and bottom, and wired elsewhere. The domain Markov property and the comparison between boundary conditions imply that boundary conditions for the dual model on $S^*$ dominate free boundary conditions on top and bottom sides of $R^*$, and wired on the two other sides. Therefore, the probability of having a dual-open path in $S^*$ crossing from left to right is larger than $1/(1+q^2)$ thanks to \eqref{crossing square mixed}. In particular, 
$$\P[C_n][\rm dobr]{\mathcal C|\mathcal E\cap\mathcal F_{\rm left}\cap \mathcal F_{\rm right}}\ge \frac1{1+q^2}.$$
Putting everything together, we find that
$$\P[\bbZ^2][0]{\textstyle0\leftrightarrow \partial \Lambda_{n/16}}\ge \frac{c}{n^{16}}\cdot c_2^2\cdot \frac{1}{1+q^2}$$
and indeed
$$\lim_{n\rightarrow \infty}-\tfrac1n\log\big(\,\P[\bbZ^2][0]{\textstyle0\leftrightarrow \partial \Lambda_{n}}\big)=0.$$
\end{proof}

\begin{remark}
We worked in the specific domain $U_n$ but one may apply these techniques in more general subdomains of $\bbU$ to obtain exponential decay of correlations using the parafermionic observable and ideas from \cite{DumTas15,DumTas15a}. We refer to \cite{Dum15} for an exposition of a proof based on this idea in the $q\le 3$ case (the proof could be extended using the techniques of Section~\ref{subsection:scaling limit} to the $q\in[3,4]$ case, but the proof has not been written down). Note that for $q>4$, a proof of exponential decay of correlations using the observable only can be found in \cite{BefDumSmi12}.
\end{remark}

\subsection{Ordering for $q>4$}

As mentioned in the introduction, the phase transition of the random-cluster model (or equivalently of the Potts model) with $q>4$
is expected to be discontinuous. In particular, this would mean that there exists an infinite cluster almost surely for the critical measure with wired boundary conditions. We are currently unable to prove this result. Nevertheless, we are able to prove the following (much) weaker result.

The graph $\bbU$ is planar and we can define its dual graph, denoted
$\widetilde \bbU$. Notice that all the vertices of $\widetilde \bbU$
have degree $4$, except one vertex, denoted $\mathbf b$, which has
infinite degree (it corresponds to the vertex at the center of the
spiral). Given a finite subgraph $\widetilde U$ of $\widetilde \bbU$,
containing the vertex $\mathbf b$, we define the random cluster
measure in $\widetilde U$ with wired boundary conditions: the boundary
of $\widetilde U$ is given by $\mathbf b$ together with all the
vertices of $\widetilde U$ with degree strictly smaller than $4$ in
$\widetilde U$. We can then define the random-cluster measure on
$\widetilde \bbU$ with wired boundary condition, denoted by
$\phi^1_{\widetilde \bbU}$, by taking the limit as $\widetilde U$
converges to $ \widetilde \bbU$. 
We fix a vertex $\mathbf v$ of $\widetilde \bbU$, disjoint from $\mathbf
b$, and write $\mathbf v \longleftrightarrow\infty$ for the event that
there exists an infinite open path from $\mathbf v$ which is \emph{disjoint from $\mathbf b$}.
   
\begin{proposition}\label{prop:qge5}
For $q>4$,
$\phi_{\widetilde \bbU}^1(\mathbf v \longleftrightarrow\infty)>0.$
\end{proposition} 
It would be very interesting to improve this result by showing that it implies $\phi^1_{\bbZ^2}[0\lr\infty]>0$. As for today, we did not manage to do so. Let us mention a question whose understanding may help solving this problem. Consider the upper half-plane $\bbH=\bbZ\times\bbN$. Assume that there exists an infinite cluster in this plane (once again not using any boundary edge) for the random-cluster measure with wired boundary conditions on the boundary of $\bbH$. Can one show that there exists an infinite cluster for the random-cluster measure on $\bbZ^2$ with wired boundary conditions. Obviously, this is true for Bernoulli percolation since $\bbH\subset\bbZ^2$. Yet this fact is not sufficient to prove the claim for general cluster-weights $q>1$ since the wiring on $\bbZ^2\times\{0\}$ may influence the measure inside the upper half-plane by favoring open edges.

Let us now prove the proposition.

\begin{proof}
  Consider the random-cluster measure on $\bbU$ with free boundary
  conditions. Let us prove that ${\bf 0}:=(0,0,0)$ and $(0,0,-k)$ are
  connected to each other with probability decaying exponentially fast
  in $k\ge0$. The Borel-Cantelli lemma would then imply that finitely
  many pairs of integers $k\ge 0$ and $\ell> 0$ are such that
  $(0,0,-k)$ and $(0,0,\ell)$ are connected to each other. This
  immediately shows the existence of an infinite dual-open cluster in
  the dual model, which is the random cluster on $ \widetilde \bbU$ with wired boundary  conditions.

  To prove this exponential decay, we invoke \eqref{eq:main equation}
  again.

  Let $n>k+1\ge 2$. Let $U_n$ be the subdomain of $\bbU$ generated by
  $\bbU\cap[-n,n]^3$, except that the edge $f_k:=\{(0,0,-k),(0,-1,-k)\}$
  is removed. $U_n$ can be seen as the Dobrushin domain with
  $a=b=(0,0,0)$. Let $V$ be the set of all medial vertices in $U_n$ with
  four incident medial edges. Medial vertices in $V$ corresponds to primal edges in
  $\bbU\cap[-n,n]^3$ except the edge $f_k$. Let $e_1,e_2,e_3,e_4$ be
  the four medial edges incident to the medial vertex corresponding to
  $f_k$. These four medial edges are incident to exactly one vertex in
  $V$.  Two of them, say $e_1$ and $e_2$, border the vertex $(0,0,-k)$. Let $E$ be the set of all medial edges adjacent to exactly one
  vertex in $V$, distinct from $e_1,e_2,e_a,e_b$.
  Equation~\eqref{eq:main equation} implies
  \begin{equation}
    -\eta(e_1)F(e_1)-\eta(e_2)F(e_2)-\sum_{e\in E}\eta(e)F(e) = F(e_a) -F(e_b).\label{eq:37}
  \end{equation}
  
  The Dobrushin boundary conditions are simply the free boundary
  conditions in this case. Set $\tilde\sigma=\mathrm i \hat\sigma$ which is a
  real number {\em since we assume that $q>4$}. Without loss of
  generality, we can assume $\tilde\sigma>0$. We find
  \begin{equation}
-\eta(e_1)F(e_1)-\eta(e_2)F(e_2)=e^{\tilde \sigma 2\pi (k+1)}
(e^{\pi\tilde\sigma/2}-1)\P[U_n][0]{{\bf 0}\lr(0,0,-k)},\label{eq:38}
\end{equation}
and
\begin{equation}
  F(e_a) -F(e_b)=e^{\tilde \sigma 2\pi}-1\label{eq:39}.
\end{equation}
One can arrange the edges of $E$ in pairs, by putting together two
medial edges when they border the same (primal) vertex. Using this
pairing, one sees that 
\begin{equation}
  \label{eq:40}
  \sum_{e\in E}\eta(e)F(e)\le 0.
\end{equation}
Plugging Equations \eqref{eq:38},\eqref{eq:39} and \eqref{eq:40} in
\eqref{eq:37}, we deduce that there exists $C=C(q)>0$ not depending on
$k$ or $n$ so that
\begin{align}\nonumber
\P[U_n][0]{{\bf 0}\lr(0,0,-k)}&\le Ce^{-2\pi\tilde \sigma k}.
\end{align}
Letting $n$ tend to infinity, and using the finite energy property (recall that the edge $f_k$ is closed in $U_n$, but not necessarily in $\bbU$), we
conclude that
$$\P[\bbU][0]{{\bf 0}\lr(0,0,-k)}\le Ce^{-2\pi\tilde\sigma k}.$$
\end{proof}

\begin{remark}
The fact that $\sigma\in[0,1]$ for $q\le 4$ and $\sigma=1-{\rm i}\bbR_+$ for $q>4$ explains the difference of behavior between $q\le 4$ and $q>4$ random-cluster models.
\end{remark}
\section{Proofs of other theorems}

\subsection{Proof of Theorem~\ref{thm:spatial mixing}}

This section contains the proof of Theorem~\ref{thm:spatial mixing}.

\begin{lemma}\label{lem:no one arm}
Let $k\le n$ and $\xi$ some arbitrary boundary conditions on $\partial\Lambda_n$. There exist two couplings $\mathbf P$ and $\mathbf Q$ on configurations $(\omega_\xi,\omega_1)$ with the following properties:
\begin{itemize}
\item $\omega_\xi$ and $\omega_1$ have respective laws $\PP[\Lambda_n][\xi]$ and $\PP[\Lambda_n][1]$.
\item $\mathbf P$-almost surely, if $\omega_1^*$ contains a dual-open dual-circuit in $\Lambda_n^*\setminus \Lambda_k^*$ and $\Gamma^*$ is the outermost such circuit, then $\Gamma^*$ is also closed in $\omega_\xi$, and furthermore $\omega_1$ and $\omega_\xi$ coincide inside $\Gamma^*$.
\item $\mathbf Q$-almost surely, if $\omega_\xi$ contains an open circuit in $\Lambda_n\setminus \Lambda_k$ and $\tilde \Gamma$ is the outermost such circuit, then $\tilde \Gamma$ is also open in $\omega_1$, and furthermore $\omega_1$ and $\omega_\xi$ coincide inside $\tilde \Gamma$.
\end{itemize}
\end{lemma}

\begin{proof} We start by explaining how to sample $\PP[\Lambda_n][\xi]$. The Domain Markov property enables us to construct a configuration as follows. Consider uniform random variables $U_e$ on $[0,1]$ for every edge $e$. Choose an edge $e_1$ and declare it open if 
$U_{e_1}$ is smaller or equal to $\P[\Lambda_n][\xi]{\omega(e_1)=1}$. Choose another edge $e_2$ and declare it open if $U_{e_2}\le\P[\Lambda_n][\xi]{\omega(e_2)=1|\omega(e_1)}$. We iterate this procedure for every edge. Also note that we can stop the procedure after a certain number of edges and sample the rest of the edges according to the right conditional law. The domain Markov property guarantees that the measure thus obtained is $\PP[\Lambda_n][\xi]$. Note that the choice of the next edge can be random, as long as it depends only on the state of edges discovered so far.

Of course, the previous construction is useless for one measure, but it becomes interesting if we consider two measures:
one may sample both configurations based on the same random variables $U_e$ with a specific way of choosing the next edges. Let us now describe the way we are choosing the edges:
\begin{itemize}
\item {\em Construction of }{\bf P}: After $t$ steps, the edge $e_{t+1}\in E_{\Lambda_n}\setminus E_{\Lambda_k}$ is chosen in such a way that it has one end-point connected to  $\partial\Lambda_n$ by an open path in $\omega_1$, until it is not possible anymore. Then sample all remaining edges at once according to the correct conditional law. 
If there is a closed circuit surrounding $\Lambda_k$ in $\omega_1$, then there is a time $t$ such that at time $t+1$, no undiscovered edges has an end-point which is connected to the boundary in $\omega_1$. Since this procedure guarantees that $\omega_1\ge\omega_\xi$, no such edges is connected to the boundary in $\omega_\xi$ as well.  Therefore, the configuration sampled inside the remaining domain is a random-cluster model with free boundary conditions in both cases.  In particular, both configurations coincide in $\Lambda_k$.

\item {\em Construction of }{\bf Q}: After $t$ steps, the edge $e_{t+1}\in E_{\Lambda_n}\setminus E_{\Lambda_k}$ is chosen in such a way that one end-point of $e^*_{t+1}$ is dual-connected to  $\partial\Lambda_n^*$ by a dual-open path in $\omega_\xi^*$, until it is not possible anymore. Then sample all remaining edges according to the correct conditional law. 
 If there is an open circuit surrounding $\Lambda_k$ in $\omega_\xi$, then there is a first time $t$ such that the open circuit is discovered exactly at time $t$. Once again, $\omega_1\ge\omega_\xi$ and this circuit is also open in $\omega_1$. Then, the configuration inside the connected component of $\Lambda_k$ in $\Lambda_n\setminus\{e_1,e_2,\dots,e_t\}$ will be sampled according to a random-cluster configuration with wired boundary conditions. In particular, both configurations coincide in $\Lambda_k$.\end{itemize}\end{proof}

\begin{proof}[Proof of Theorem~\ref{thm:spatial mixing}]
It is clearly sufficient to prove that there exists $\alpha>0$ such that
\begin{equation*}
\big|\P[\Lambda_n][\xi]{A}-\P[\Lambda_n][1]{A}\big|\leq \left(\frac{k}{n}\right)^\alpha\P[\Lambda_n][\xi]{A}
\end{equation*}
for any event $A$ depending on edges in $\Lambda_k$. Let $E$ be the event that there exists a dual-open dual-circuit in $\omega_\xi^*$ included in $\Lambda_n^*\setminus\Lambda_k^*$. We deduce
\begin{align*}
\P[\Lambda_n][\xi]{A}&\ge\P[\Lambda_n][\xi]{A\cap E}=\mathbf Q[\omega_\xi\in A\cap E]\ge \mathbf Q[\omega_1\in A\cap E]\\
&=\P[\Lambda_n][1]{A\cap E}\ge (1-(k/n)^\alpha)\P[\Lambda_n][1]{A}\end{align*}
where in the third inequality, we used the fact that if $\omega_1$ belongs to $A\cap E$, then $\omega_1$ and therefore $\omega_\xi$ belong to $E$. Since on $E$, $\omega_1$ and $\omega_\xi$ coincide in $\Lambda_k$, then $\omega_\xi\in A$. The existence of $\alpha$ in the last inequality follows exactly as in the proof of Lemma~\ref{lem:one arm} from Property {\bf P5a} applied in concentric annuli $\Lambda_{k2^{i+1}}\setminus \Lambda_{k2^{i}+1}$ with $0\le i\le \log_2(n/k)$.

Reciprocally, if $F$ denotes the event that there is an open circuit in $\Lambda_n\setminus\Lambda_k$, we find
\begin{align*}
\P[\Lambda_n][1]{A}&\ge\P[\Lambda_n][1]{A\cap F}=\mathbf P[\omega_1\in A\cap F]\ge \mathbf P[\omega_\xi\in A\cap F]\\
&=\P[\Lambda_n][\xi]{A\cap F}\ge(1-(k/n)^\alpha)\P[\Lambda_n][\xi]{A}\end{align*}
where once again, we used in the third inequality that if $\omega_\xi\in A\cap F$, then $\omega_1$ is in $F$, and since $\omega_1$ and $\omega_\xi$ then coincide on $\Lambda_k$,  we get that $\omega_1\in A$. The last inequality is due to {\bf P5a} once again.
\end{proof}

\subsection{Proof of Theorem~\ref{thm:Loewner}}

In the statement of Theorem~\ref{thm:Loewner}, we consider an approximation of a simply connected domain $\Omega$ with two points $a$ and $b$ on its boundary. More precisely, $(\Omega_\delta,a_\delta,b_\delta)$ denotes a Dobrushin domain defined on the square lattice of mesh size $\delta$, i.e. $\delta\bbZ^2$. All the definitions and result extend to this context in a direct fashion.

A family of Dobrushin domain {\em approximates} a continuous domain $(\Omega,a,b)$ if $(\Omega_\delta,a_\delta,b_\delta)$ converges in the Carath\'eodory sense as $\delta$ tends to 0. This convergence is classical, we refer to \cite[Chapter 3]{Dum13} for details. Let us simply say that for smooth domains, it corresponds to the convergence in the Hausdorff sense.

In what follows, $\gamma_\delta$ denotes the exploration path in the Dobrushin domain $(\Omega_\delta,a_\delta,b_\delta)$.

In order to prove Theorem~\ref{thm:Loewner}, \cite{KemSmi12} shows that we only need to check the condition {\bf G2} defined now.
Consider a fixed simply connected domain $(\Omega,a,b)$ and a parametrized continuous curve $\Gamma$ from $a$ to $b$ in $\Omega$. A connected set $C$ is said to disconnect $\Gamma(t)$ from $b$ if it disconnects a neighborhood of $\Gamma(t)$ from a neighborhood of $b$ in $\Omega\setminus\Gamma[0,t]$. 

Fig.~\ref{fig:six arm event} will help the reader here. For any annulus $A=A(z,r,R):=z+(\Lambda_R\setminus\Lambda_r)$, let 
$A_t$ be the subset of $\Omega$ satisfying $A_t:=\emptyset$ if $\partial (z+\Lambda_r)\cap\partial (\Omega\setminus\Gamma[0,t])=\emptyset$, and otherwise
\begin{align*}A_t:=\left\{\begin{array}{l}z\in A\setminus\Gamma[0,t]\text{ such that the connected component of $z$}\\
\text{in $A\setminus\Gamma[0,t]$ does not disconnect $\Gamma(t)$ from $b$ in $\Omega\setminus\Gamma[0,t]$}\end{array}\right\}.\end{align*}
Consider the exploration path $\gamma_\delta$ as a continuous curve
from $a_\delta^\diamond$ to $b_\delta^\diamond$ parametrized in such a
way that it goes along one medial vertex in time 1 (in particular,
after time $n$ the path explored $n$ medial-vertices). For simplicity,
once the path reaches $b_\delta^\diamond$, it remains at
$b_\delta^\diamond$ for any subsequent time.

\paragraph{Condition G2}{\em There exists $C>1$ such that for any $(\gamma_\delta)$ in $(\Omega_\delta,a_\delta,b_\delta)$, for any stopping time $\tau$ and any annulus $A=A(z,r,R)$ with $0<Cr<R$, 
$$\phi_{\Omega_\delta}^{a_\delta,b_\delta}\Big(\gamma_\delta[\tau,\infty]\text{ makes a crossing of $A$ contained in $A_\tau$}\Big|\gamma_\delta[0,\tau]\Big)
< \tfrac12.$$
Above, ``$\gamma_\delta[\tau,\infty]$  makes a crossing of $A$
  contained in $A_\tau$'' means that there exists a sub-path
  $\gamma_\delta[t_1,t_2]$ of the continuous path
  $\gamma_\delta[\tau,\infty]$ that intersects both the boundary of
  $z+\Lambda_r$ and the boundary of $z+\Lambda_R$,  and such that
  $\gamma_\delta[t_1,t_2]\subset A_\tau$.  
}

\begin{figure}
\begin{center}
\includegraphics[width=0.80\textwidth]{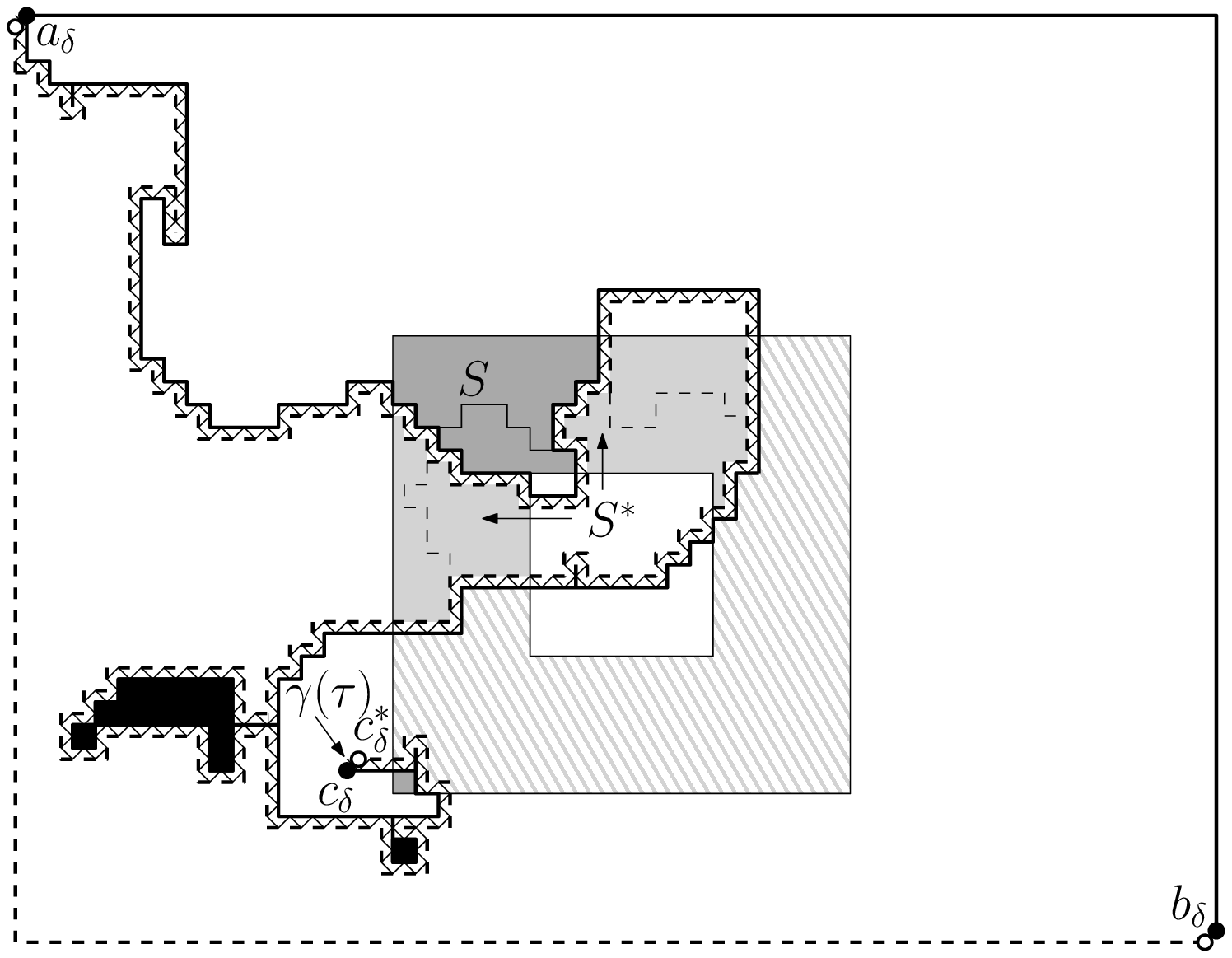}
\end{center}
\caption{\label{fig:six arm event}The dashed area is a connected component of $A\setminus\gamma_\delta[0,\tau]$ which is disconnecting $\gamma_\delta(\tau)$ from $b_\delta^\diamond$, or equivalently $c_\delta$ from $b_\delta$, and which is therefore not in $A_\tau$.
 The black parts are not included in the slit domain since they correspond to connected components that are not containing $b_\delta^\diamond$. Conditioning on $\gamma_\delta[0,\tau]$ induces Dobrushin boundary conditions in the new domain.
 The dark grey area is $S$ and the light-gray $S^*$. We depicted a blocking open path in $S$ and a dual-open path in each connected component of $S^*$.  }
\end{figure}

Before proving Condition {\bf G2} (and therefore Theorem~\ref{thm:Loewner}), let us introduce the notion of slit Dobrushin domain. 
Fig.~\ref{fig:six arm event} provides an example of what it may be.  

Fix a Dobrushin domain $(\Omega_\delta,a_\delta,b_\delta)$ and consider the exploration path $\gamma_\delta$ in the loop representation on $\Omega_\delta$. The path $\gamma_\delta$ can be seen as a random parametrized curve (the parametrization being simply given by the number of steps along the curve  between a medial-vertex in $\gamma_\delta$ and $a_\delta^\diamond$). 
\begin{definition}\label{def:slit domain}The {\em slit domain} $\Omega_\delta\setminus\gamma_\delta[0,n]$ is defined as the subdomain of $\Omega_\delta$ constructed by removing all the primal edges crossed by $\gamma_\delta[0,n]$ and by keeping only the connected component of the remaining graph containing $b_\delta$. It is seen as a Dobrushin domain by fixing the points $c_\delta$ and $b_\delta$, where $c_\delta$ is the vertex of $\delta\bbZ^2$ bordered by the last medial edge of $\gamma_\delta[0,n]$. \end{definition}

Similarly, one may define the dual Dobrushin domain. The marked point is then $c_\delta^*$, where $c_\delta^*$ is the dual-vertex of $(\delta\bbZ^2)^*$ bordered by the last medial edge of $\gamma_\delta[0,n]$. It is worth mentioning that the construction is symmetric for the dual Dobrushin domain: the dual of the slit domain $\Omega_\delta\setminus\gamma[0,n]$ is simply the connected component of $b_\delta^*$ in the subgraph of $\Omega_\delta^*$ obtained  by removing the dual-edges crossed by the curve.

\begin{remark}
The notation $\Omega_\delta\setminus\gamma_\delta[0,n]$ could somewhat be misleading, since $\Omega_\delta$ is a subset of $\delta\bbZ^2$ and $\gamma_\delta[0,n]$ is a path of medial edges. Nevertheless, we allow ourselves some latitude here since we find this notation both concise and intuitive.
\end{remark}

If one starts with Dobrushin boundary conditions on $(\Omega_\delta,a_\delta,b_\delta)$, then conditionally on $\gamma_\delta[0,n]$ the law of the configuration inside $\Omega_\delta\setminus\gamma_\delta[0,n]$ is a random-cluster model with Dobrushin boundary conditions wired on $\partial_{b_\delta c_\delta}$ and free elsewhere. This comes from the fact that the exploration path $\gamma_\delta[0,n]$ ``slides between open edges and dual-open dual-edges'' and therefore the edges on its left must be open and the dual-edges on its right dual-open. This implies that the arc $\partial_{a_\delta c_\delta}$ must be wired (and therefore $\partial_{b_\delta c_\delta}$ is since $\partial_{b_\delta a_\delta}$ was already wired to start with) and the dual arc $\partial_{c_\delta^* a_\delta^*}$ is dual-wired.
\bigbreak

We are now in a position to prove Condition {\bf G2}.

\medbreak
\noindent{\em Proof of Condition {\bf G2}.}
Let $A(z,r,R)$ and $A_\tau$ as defined above. We can fix a realization of $\gamma_\delta[0,\tau]$, and work in the slit Dobrushin domain $(\Omega_\delta\setminus\gamma_\delta[0,\tau],c_\delta,b_\delta)$. 

See $A_\tau$ as the union of connected components of the Dobrushin
domain $\Omega_\delta^\diamond$ seen as an open domain of $\bbR^2$
minus the path $\gamma_\delta[0,\tau]$. We denote generically a
connected component by $\calC$ (we see it as a subset of $\bbR^2$).

The connected components can be divided into three classes:
\begin{itemize}[noitemsep,nolistsep]
\item $\partial\calC$ intersects both $\partial_{c_\delta b_\delta}^\diamond$ and $\partial_{b_\delta c_\delta}^\diamond$;
\item $\partial\calC$ intersects $\partial_{b_\delta c_\delta}^\diamond$ but not $\partial_{c_\delta b_\delta}^\diamond$.
\item $\partial\calC$ intersects $\partial_{c_\delta b_\delta}^\diamond$ but not $\partial_{b_\delta c_\delta}^\diamond$;
\end{itemize}
In fact, there cannot be any connected component of the first type. Indeed, let us assume that such a connected component $\calC$ does exist. Let $\gamma$ be a self-avoiding path in $\calC$ going from $\partial_{b_\delta c_\delta}^\diamond$ to $\partial_{c_\delta b_\delta}^\diamond$. Topologically,  $c_\delta^\diamond$ and $b_\delta^\diamond$ must be on two different sides of $\Gamma$ in $(\Omega_\delta\setminus \gamma_\delta[0,\tau])^\diamond\setminus\Gamma$
. But this means that $\calC$ disconnects $c_\delta$ from $b_\delta$, and therefore that $\calC$ is not part of $A_\tau$, which is contradictory.

We can therefore safely assume that the connected components are either of the second or third types. We now come back to the interpretation in terms of graphs.

Let $S$ be the subgraph of
$\Omega_\delta\setminus\gamma_\delta[0,\tau]$ given by the union of
the connected components (seen as primal graphs this time) of the
second type (see Fig.~\ref{fig:six arm event}). This set is a subset
of $A_\tau$. Furthermore, the boundary conditions induced by the
conditioning on $\gamma_\delta[0,\tau]$ are wired on $\partial
S\setminus\partial A_\tau$. Therefore, conditioned on
$\gamma_\delta[0,\tau]$ and the configuration outside $A(z,r,R)$, the
configuration $\omega$ in $S$ dominates $\omega'_{|S}$, where
$\omega'$ follows the law of a random-cluster model in $A(z,r,R)$ with
free boundary conditions. In particular, if there exists an open
circuit in $\omega'$ surrounding $z+\Lambda_r$ in $A(z,r,R)$, then the
restriction of this path to $S$ is also open in $\omega$ and it
disconnects $z+\partial\Lambda_r$ from $z+\partial\Lambda_R$ in $S$.
In particular, the exploration path $\gamma_\delta[\tau,\infty]$
cannot cross $A_\tau$ inside $S$ since this would require the
existence of a dual-open dual-path from the outer to the inner part of
$A_\tau^*$.

Property~{\bf P5a} implies that this open circuit exists in $\omega'$
with probability larger than a constant $c>0$ not depending on
$\delta$, and that therefore $\gamma_\delta[\tau,\infty]$ cannot cross
$A_\tau$ inside $S$ with probability larger than $c$ uniformly on the
configuration outside $A_\tau$.

Let now $S^*$ be the subgraph of $\Omega_\delta^*\setminus\gamma_\delta[0,\tau]$ given by the union of the connected components (seen as dual graphs) of the third type. The same reasoning for the dual model implies that with probability $c>0$, the exploration path $\gamma_\delta[\tau,\infty]$ cannot cross $A_\tau^*$ inside $S^*$.

Altogether, $\gamma_\delta[\tau,\infty]$ cannot cross $A_\tau$ with probability $c^2$. Now, Proposition~\ref{prop:classical RSW} shows that $c$ can be taken to be equal to $1-(1-c_2)^{\lfloor\log_2(R/r)\rfloor}$. Since $R/r\ge C$, we can guarantee that $c^2\ge 1/2$ by choosing $C$ large enough.

\subsection{Proof of Theorem \ref{thm:strongest RSW}} \label{subsection:scaling limit}

We omit certain details in the proof of Theorem~\ref{thm:strongest RSW} (for instance when we use {\bf P5}, since we already did it several times in this article).
We start by a lemma. Define the upper half-plane $\bbH=\bbZ\times\bbN$.

\begin{lemma}
Assume that there exists $c>0$ such that for any $n\ge 1$,
$$\P[\mathbb H][0]{[-n,n]\times\{0\}\longleftrightarrow \bbZ\times\{n\}}\ge c.$$
Then, for any $\alpha>0$ there exists $c(\alpha)>0$ such that for every $n\ge1$,
$$\P[{[-n,n]\times[0,\alpha n]}][0]{\calC_v([-n,n]\times[0,\alpha n])}\ge c(\alpha).$$
\end{lemma}

\begin{proof}
First, Property {\bf P5} allows us to choose $\beta>1$ so that for any $m\ge1$,
$$\P[\bbZ^2][1]{\calC_h([0,(\beta-1) m]\times[0,m])}\le \tfrac c3.$$
Let 
$$\calA(m):=\left\{[-m,m]\times\{0\}\lr[{[-\beta m,\beta m]\times[0,m]}] [-\beta m,\beta m]\times\{m\}\right\}.$$
With this choice of $\beta$, and because of the comparison between boundary conditions, we find that
$$\P[\bbH][0]{\calA(m)}\ge \tfrac c3.$$
Using Property {\bf P5} for the dual model, we deduce that there exists $c_1>0$ such that for any $m\ge1$,
\begin{equation}\P[{[-2\beta m,2\beta m]\times[0,2m]}][0]{\calF(m)\cap\calA(m)}\ge c_1,\end{equation}
where $\calF(m)$ is the existence of a dual-path in the
half-annulus $$[-2\beta m,2\beta m]\times[0,2m]\setminus [-\beta
m,\beta m]\times[0,m]$$ from $[\beta m,2\beta m]\times\{0\}$ to
$[-2\beta m,-\beta m]\times\{0\}$. (The occurrence of $\cal F_n$ and
$\cal A_n$ is illustrated on Fig.~\ref{fig:touching}.) 
After conditioning on the exterior most such dual path, the domain Markov property and the comparison between boundary conditions imply that
\begin{equation}\label{eq:lol}\P[{[-2\beta m,2\beta m]\times[0,2m]}][0]{\calA(m)}\ge c_1.\end{equation}

\begin{figure}[htbp]
  \centering
  \begin{minipage}[b]{.47\linewidth}
    \includegraphics[width=\linewidth]{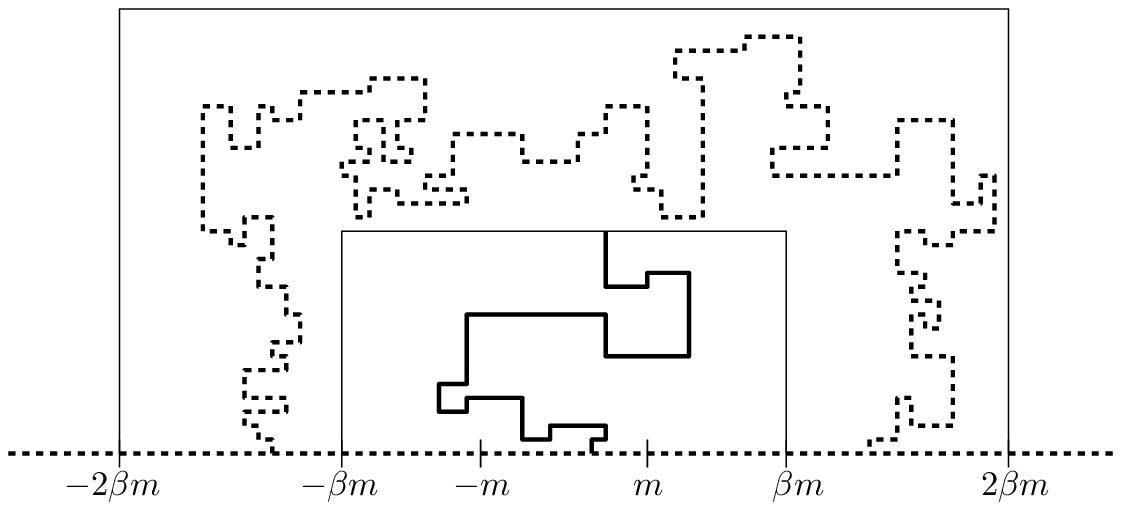}
    \caption{The occurrence of $\calA$ and $\calF$ in $\bbH$ with
      free boundary conditions on $\bbZ\times\{0\}$}
    \label{fig:touching}
  \end{minipage}
  \hfill
    \begin{minipage}[b]{.47\linewidth}
    \includegraphics[width=\linewidth]{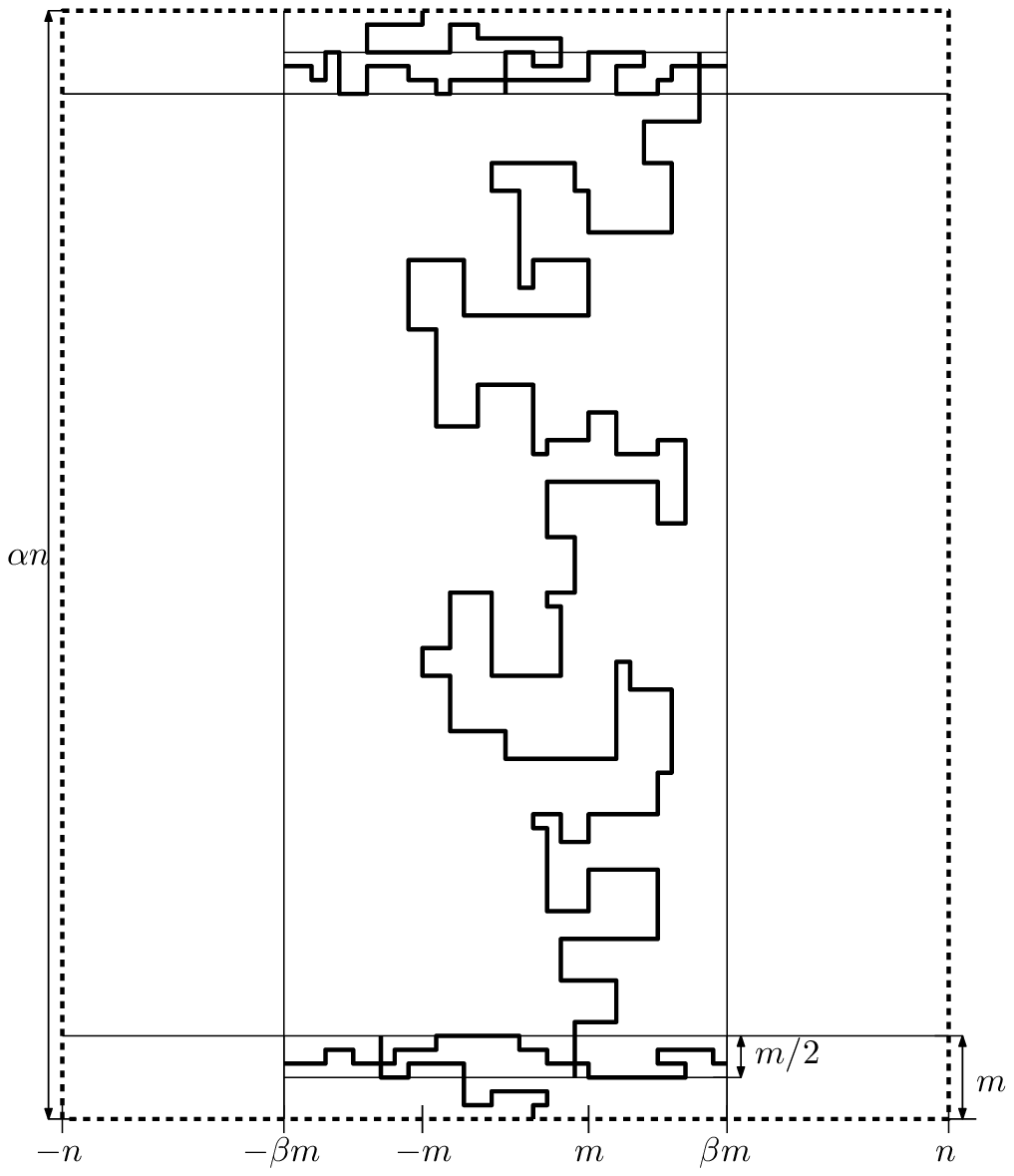}
    \caption{The occurrence of $\calA,\ldots,\cal E$ in
      $[-n,n]\times[0,\alpha n]$ with
      free boundary conditions}
    \label{fig:touchingCrossing}
  \end{minipage}
\end{figure}

We now fix $\alpha\ge 2$ and $n\ge1$. We assume without loss
of generality that $\alpha n$ is an integer. Set $m=\lfloor
n/(2\beta)\rfloor$. Define the following four events:
\begin{align*}
\calB(m)&=\big\{[-\beta m,\beta m]\times\{\alpha n-m\}\lr[{[-\beta m,\beta m]\times[\alpha n-m,\alpha n]}][-m, m]\times\{\alpha n\}\big\},\\
\calC(m)&=\calC_h\big([-\beta m,\beta m]\times [\tfrac m2,m]\big),\\
\calD (m)&= \calC_h\big([-\beta m,\beta m]\times [\alpha n-m,\alpha n-\tfrac m2]\big),\\
\calE (m)&=\calC_v\big([-\beta m,\beta m]\times [\tfrac m2,\alpha n-\tfrac m2]\big).
\end{align*}
Note that the event $\calB(m)$ is the equivalent of the event $\calA(m)$, but ``at the top of the rectangle''. By \eqref{eq:lol} and comparison between boundary conditions, $\P[{[-n,n]\times[0,\alpha n]}][0]{\calA(m)}$ and $\P[{[-n,n]\times[0,\alpha n]}][0]{\calB(m)}$ are larger or equal to $c_1$. Furthermore, Property {\bf P5} also implies that $\P[{[-n,n]\times[0,\alpha n]}][0]{\calC(m)}$ and $\P[{[-n,n]\times[0,\alpha n]}][0]{\calD(m)}$ are larger than $c_2>0$, and that $\P[{[-n,n]\times[0,\alpha n]}][0]{\calE(m)}$ is larger than $c_3(\alpha)>0$. 

Now observe that if the five events $\calA(m),\ldots,\calE(m)$ occur
simultaneously, then the rectangle $[-n,n]\times[0,\alpha n]$ is
crossed vertically (see Fig.~\ref{fig:touchingCrossing}). The FKG inequality thus implies that
\begin{align*}\P[{[-n,n]\times[0,\alpha n]}][0]{\calC_v([-n,n]\times[0,\alpha n])}&\ge\P[{[-n,n]\times[0,\alpha n]}][0]{\calA(m)\cap\calB(m)\cap\calC(m)\cap\calD(m)\cap\calE(m)}\\
&\ge c_1^2\cdot c_2^3.\end{align*} We just proved the result for $\alpha\ge 2$ with $c(\alpha)=c_1^2c_2^2c_3(\alpha)>0$. For $\alpha\in(0,2)$, simply apply the comparison between boundary conditions in the rectangle $[-\tfrac{\alpha n}4,\tfrac{\alpha n}4]\times[0,\alpha n]$ which has aspect ratio 2 and is included in $[-n,n]\times[0,\alpha n]$.
\end{proof}
The previous lemma shows that in order to prove Theorem~\ref{thm:strongest RSW}, we may focus on showing that there exists a constant $c>0$ such that 
for any $n\ge 1$,
$$\P[\mathbb H][0]{[-n,n]\times\{0\}\longleftrightarrow \bbZ\times\{n\}}\ge c.$$
In order to do so, we use the second-moment method. Let $p_n=\P[\mathbb H][0]{0\lr \Z\times\{n\}}$ and define
$$\mathsf N:=\sum_{x\in [-n,n]\times\{0\}}{\mathbf 1}_{x\leftrightarrow \bbZ\times\left\{n\right\}}.$$
By definition $\P[\mathbb H][0]{\mathsf N}=(2n+1)p_n$. Furthermore, 
\begin{align}\nonumber\P[\mathbb H][0]{\mathsf N^2}&\le (2n+1)\sum_{x\in [-2n,2n]\times\{0\}}\P[\mathbb H][0]{0,x\lr \Z\times\{n\}}\\
\nonumber&\le (2n+1)\cdot \Cl{blabla}\sum_{x\in [-2n,2n]\times\{0\}}\P[\mathbb H][0]{\Lambda_{2|x|}\lr \Z\times\{n\}}\P[\mathbb H][0]{0\lr \partial\Lambda_{|x|/4}}^2\\
&\le (2n+1)\cdot \Cl{blablabla}\sum_{x\in [-2n,2n]\times\{0\}}\P[\mathbb H][0]{0\lr \Z\times\{n\}}\P[\mathbb H][0]{0\lr \Z\times\{|x|\}}\nonumber\\
&= (2n+1)\cdot \Cr{blablabla}\cdot p_n\cdot 2\sum_{k=0}^{2n} p_k\label{eq:eq1}.
\end{align}
In the second, we used that 
$$\big\{0,x\lr \Z\times\{n\}\big\}\subset\big\{\Lambda_{2|x|}\lr \Z\times\{n\}\big\}\cap\big\{0\lr \partial\Lambda_{|x|/4}\big\}\cap\big\{x\lr x+\partial\Lambda_{|x|/4}\big\}$$ and Theorem~\ref{thm:spatial mixing} (more precisely a direct adaptation to the upper half-plane) to decouple the different events on the right side. In the third line, we used Property {\bf P5} as follows. Let $\calA$ be the event that 
\begin{itemize}[noitemsep,nolistsep]
\item $\partial\Lambda_{|x|/8}\lr\partial\Lambda_{4|x|}$,
\item there exists an open path in $\Lambda_{|x|/4}\setminus\Lambda_{|x|/8}$ disconnecting 0 from infinity in $\bbH$,
\item there exists an open path in $\Lambda_{4|x|}\setminus\Lambda_{2|x|}$ disconnecting 0 from infinity in $\bbH$.
\end{itemize}
With this definition, we see that if $0\lr \partial\Lambda_{|x|/4}$, $\Lambda_{2|x|}\lr \Z\times\{n\}$ and $\calA$ occur, then $0$ is connected to $\Z\times\{n\}$. Now, the event $\calA$ has probability bounded away from 0 uniformly in $x$ thanks to {\bf P5}, so that the FKG inequality directly implies the third inequality of \eqref{eq:eq1}. Similarly, one can use {\bf P5} to prove that $\P[\mathbb H][0]{0\lr \partial\Lambda_{|x|/4}}\le \Cl{156}
\P[\mathbb H][0]{0\lr \Z\times\{|x|\}}$.

Let us now state the following claim.
\bigbreak
\noindent{\em Claim: There exists $\Cr{constant200}>0$ such that for any $n\ge0$,} 
 $$\sum_{k=0}^n p_k\le \Cl{constant200}np_n.$$

 Before proving the claim, let us finish the proof of Theorem~\ref{thm:strongest RSW}. The claim and \eqref{eq:eq1} imply that 
 \begin{align*}\P[\mathbb H][0]{\mathsf N^2}&\le (2n+1)\Cr{blablabla}p_n2\Big(\Cr{constant200}np_n+np_n\Big)\le \Cr{blablabla}(\Cr{constant200}+1) (2n+1)2np_n^2\le \Cr{blablabla}(\Cr{constant200}+1) \P[\mathbb H][0]{\mathsf N}^2.
 \end{align*}
(We also used that $p_k\le p_n$ for $k\ge n$.) By applying the Cauchy-Schwarz inequality, we find
$$\P[\mathbb H][0]{[-n,n]\times\{0\}\longleftrightarrow\Z\times\{n\}}=\P[\mathbb H][0]{\mathsf N>0}\ge \frac{\P[\mathbb H][0]{\mathsf N}^2}{\P[\mathbb H][0]{\mathsf N^2}}\ge \frac1{\Cr{blablabla}(\Cr{constant200}+1)}>0.$$

\medbreak
\noindent{\em Proof of the Claim.} We use \eqref{eq:main equation} one more time.
Recall that $q<4$. Let $\mathbb V$ be the subgraph of $\bbZ^3$ defined as follows:
the vertices are given by $\Z^3$ and the edges by
 \begin{itemize}[noitemsep,nolistsep]
\item $[(x_1,x_2,x_3),(x_1,x_2+1,x_3)]$ for every $x_1,x_2,x_3\in \Z$,
 \item $[(x_1,x_2,x_3),(x_1+1,x_2,x_3)]$ for every $x_1,x_2,x_3\in \Z$ such that $x_1\neq 0$,
 \item  $[(0,x_2,x_3),(1,x_2,x_3)]$ for every $x_2\ge-n$ and $x_3< 0$,
 \item $[(0,x_2,x_3),(1,x_2,x_3+1)]$ for every $x_2<-n$ and $x_3<  0$,
 \item  $[(0,x_2,x_3),(1,x_2,x_3)]$ for every $x_2> n$ and $x_3\ge  0$,
 \item $[(0,x_2,x_3),(1,x_2,x_3+1)]$ for every $x_2\le n$ and $x_3\ge  0$.
 \end{itemize}
\medbreak
For $\theta>0$ and $n\ge0$, let 
$$V_{n,\theta}:=\big\{(x_1,x_2,x_3)\in\mathbb V:|x_1|\le n,|x_2|\le 2n\text{ and }|x_3|\le \theta\big\}.$$
We also set $a=(1,-n,0)$ and $b=(1,n,0)$. One can see $V_{n,\theta}$
as a Dobrushin domain constructed from the two medial vertices
$a^\diamond=(1,-n-\tfrac12,0)$ and $b^\diamond=(1,n+\tfrac12,0)$. (The
definitions need to be slightly adapted to this framework). The
boundary conditions are wired on the segment $\partial_{ba}$ between
$a$ and $b$, and free elsewhere. Let $\partial$ be the part of
$\partial_{ab}$ composed of vertices such that $|x_1|=n$ or
$|x_2|=2n$. One can choose $\theta=\theta(q)$ large enough such that
the following holds: for every vertex $x$ on
$\partial_{ab}\setminus\partial$, the two medial edges $e_1,e_2$ of
$\partial_{ab}^\diamond$ bordering $x$ satisfies
\begin{equation}
  \label{eq:41}
  \mathrm{Im}\left(\eta(e_1)\widehat F(e_1)+\eta(e_2)\widehat F(e_2)\right)\ge 0.
\end{equation}
Let $V$ be the set of medial vertices with four incident medial edges
in $V_{n,\theta}^\diamond$.  In order to apply  Equation \eqref{eq:main
  equation}, we wish to compute the sum
\begin{equation}
  \label{eq:42}
  S:=\sum_{\substack{e\text{ incident to exactly}\\  \text{one vertex in }V}} \mathrm{Im}\left(\eta(e)\widehat F(e)\right).
\end{equation} 
\begin{itemize}
\item The contribution of the medial edges $e_a,e_b$ to the
  sum $S$ is given by
  \begin{equation}
    \label{eq:43}
    \mathrm{Im}(e^{\mathrm i \hat \sigma3\pi/2}-1)=\sin(\hat\sigma3\pi/2)\ge0
  \end{equation}
\item The contribution of the  medial edges bordering a dual-vertex $u$ on the arc
  $\partial^\star_{ba}=\{\tfrac32\}\times[-n+\tfrac12,n-\tfrac12]$  is given by
  \begin{equation}
    \label{eq:44}   
    2\cos\left(\hat\sigma\tfrac{3\pi}4\right)
    \sin\left(\hat\sigma\tfrac{\pi}4\right) \sum_{u\in\partial_{ba}^*}\P[V_{n,\theta}][a,b]{u\lr[*] \partial_{ab}^*}.
  \end{equation}
\item Using Equation~\eqref{eq:41}, we find that the contribution
  $S(\partial_{ab})$ of all the medial edges bordering a vertex on the
  arc $\partial_{ab}$ satisfies
  \begin{equation}
    \label{eq:45}
    S(\partial_{ab})\ge -2\sum_{x\in \partial}\P[V_{n,\theta}][a,b]{x\lr \partial_{ba}}.
  \end{equation}
\end{itemize}

Using that $S=0$ together with the three equations above, we obtain
\begin{equation}
  \label{eq:46}
 2\cos\left(\hat\sigma\tfrac{3\pi}4\right)\sin\left(\hat\sigma\tfrac{\pi}4\right)
 \sum_{u\in\partial_{ba}^*}\P[V_{n,\theta}][a,b]{u\lr[*] \partial_{ab}^*}\le
 2\sum_{x\in \partial}\P[V_{n,\theta}][a,b]{x\lr \partial_{ba}}.
\end{equation}
Notice that the condition $1\le q<4$ implies that
$2\cos\left(\hat\sigma\tfrac{3\pi}4\right)\sin\left(\hat\sigma\tfrac{\pi}4\right)>0$. Then, using the comparison
between boundary conditions and Property {\bf P5}, we find that for every
$x$~on~$\partial$, $\P[V_{n,\theta}][a,b]{x\lr \partial_{ba}}\le \C p_n$.
Plugging this in Equation~\eqref{eq:46}, we obtain

\begin{equation}
  \sum_{u\in\partial_{ab}^*}\P[V_{n,\theta}][a,b]{u\lr[*] \partial_{ba}^*}\le \Cl{constant202}np_n.
\end{equation}

Now, for a dual-vertex $u$ of $\partial_{ab}^*$ at distance $k$ from $a$ or $b$ (let us assume without loss of generality that $u$ is at distance $k$ from $a$). Let $\Lambda_k^+=[-k,k]\times[0,k]$. If the following three events occur simultaneously:
\begin{itemize}[noitemsep,nolistsep]
\item $u\lr[*]u+\partial\Lambda_k^+$,
\item the half-annulus $u+\partial\Lambda_k^+\setminus\Lambda_{k/2}^+$ contains a dual-open dual path disconnecting $u$ from the $\partial_{ba}^*$ in $V_{n,\theta}$,
\item that there exists a dual-open dual path disconnecting $a$ from $\partial$ in 
$$\{u\in V_{n,\theta}:k/2\le |(u_1,u_2)-(0,n)|\le k\}$$
(this set is a topological rectangle winding around the singularity
$a$ at distance $k/2$ of it),
\end{itemize}
then $u$ is connected by a dual-open path to $\partial_{ba}^*$. Now, the second and third events occur with probability bounded away from 0 uniformly in $n$ thanks to {\bf P5} and the fact that $\theta(q)<\infty$ does not depend on $n$. By conditioning on these two events and using the comparison between boundary conditions, we find that
\begin{equation}\P[V_{n,\theta}][a,b]{u\lr[*] \partial_{ba}^*}\ge \cl{constant300}\P[u+\Lambda_k^+][1]{u\lr[*] u+\partial\Lambda_k^+}\ge \Cr{constant300}p_k,\end{equation}
where in the last inequality we used duality and the comparison between boundary conditions.
In conclusion, we find
$$2\Cr{constant300}\sum_{k=0}^np_k\le \Cr{constant202}np_n$$
and the conclusion follows immediately.
\begin{flushright}$\diamond$\end{flushright}

\begin{remark}
Note that $\theta=\theta(q)<\infty$ for any $q<4$, but that $\theta(q)$ tends to infinity as $q$ tends to 4 and there is no such $\theta(q)$ for $q=4$. This fact explains why this proof does not work for $q=4$ (and should not work since the statement is expected to be false). 
\end{remark}

\subsection{Proof of Theorem \ref{prop:uniqueness}}

The coupling between Potts and random-cluster models shows that
$$\mu_{\bbZ^2,\beta_c(q),q}[\sigma_x=\sigma_0]-\frac1q=\P[\bbZ^2,p_c(q),q][1]{0\lr x}.$$
Lemma~\ref{lem:one arm} implies the existence of $\eta_2>0$ such that 
$$\P[\bbZ^2,p_c(q),q][1]{0\lr x}\le \frac{1}{|x|^{\eta_2}},$$
thus implying the second inequality of Theorem~\ref{prop:uniqueness}.

In order to obtain the first, we do not need {\bf P5} but just Theorem~\ref{thm:weakRSW}. Indeed, if $\calC_h([0,n]\times[0,n])$ occurs, then one of the vertices of $[0,n]\times\{0\}$ is connected to distance $n$.
As a consequence, the union bounded implies that 
$$\P[\bbZ^2,p_c(q),q][1]{0\lr \partial\Lambda_n}\ge \frac{c}{n}.$$
Let $\calA(n)$ be the event that there exists an open circuit in $\Lambda_{2n}\setminus \Lambda_{n}$ surrounding the origin. The FKG inequality and Theorem~\ref{thm:weakRSW} one more time show that there exists $c'>0$ so that
\begin{align*}\P[\bbZ^2,p_c(q),q][1]{0\lr x}&\ge \P[\bbZ^2,p_c(q),q][1]{0\lr\partial\Lambda_{3|x|}}\P[\bbZ^2,p_c(q),q][1]{x\lr\partial\Lambda_{3|x|}}\P[\bbZ^2,p_c(q),q][1]{\calA(|x|)}\\
&\ge \frac{c'}{|x|^2}\end{align*}
and the proof follows by choosing $\eta_1$ small enough.

\bibliographystyle{alphaNames}
\bibliography{biblicomplete}

\end{document}